\documentclass[a4paper,11pt]{amsart}

\usepackage{amsmath,amsthm,amssymb,amsfonts,color,tikz,enumitem,url,hyperref,subcaption,comment}
\usepackage[margin=1in]{geometry}
\usetikzlibrary{decorations.pathreplacing,arrows.meta}

\newcommand{\R}{\mathbb{R}}
\newcommand{\Z}{\mathbb{Z}}
\newcommand{\N}{\mathbb{N}}
\newcommand{\GG}[1][]{\Gamma_{#1}\mathcal{G}_{#1}}

\newcommand{\st}{\operatorname{star}}
\newcommand{\lk}{\operatorname{link}}
\newcommand{\link}{\operatorname{link}}
\newcommand{\supp}{\operatorname{supp}}

\newcommand{\INT}{\mathrm{int}}
\newcommand{\llangle}{\left\langle\!\!\left\langle}
\newcommand{\rrangle}{\right\rangle\!\!\right\rangle}
\newcommand{\Stab}{\operatorname{Stab}}

\theoremstyle{definition}
\newtheorem{defn}{Definition}[section]
\newtheorem{rmk}[defn]{Remark}
\newtheorem*{ack}{Acknowledgements}

\theoremstyle{plain}
\newtheorem{lem}[defn]{Lemma}
\newtheorem{prop}[defn]{Proposition}
\newtheorem{cor}[defn]{Corollary}
\newtheorem{thm}[defn]{Theorem}

\newtheorem{thmx}{Theorem}
\newtheorem{corx}[thmx]{Corollary}

\title[Quasi-median graphs: acylindrical hyperbolicity and equations]{Acylindrical hyperbolicity of groups acting on quasi-median graphs and equations in graph products}
\author{Motiejus Valiunas}
\address{Instytut Matematyczny, Uniwersytet Wroc{\l}awski, Plac Grunwaldzki 2/4, 50-384 Wroc{\l}aw, Poland}
\email{valiunas@math.uni.wroc.pl}
\date{\today}

\begin{document}

\begin{abstract}
In this paper we study group actions on quasi-median graphs, or `CAT(0) prism complexes', generalising the notion of CAT(0) cube complexes. We consider hyperplanes in a quasi-median graph $X$ and define the \emph{contact graph} $\mathcal{C}X$ for these hyperplanes. We show that $\mathcal{C}X$ is always quasi-isometric to a tree, generalising a result of Hagen \cite{hagen}, and that under certain conditions a group action $G \curvearrowright X$ induces an \emph{acylindrical} action $G \curvearrowright \mathcal{C}X$, giving a quasi-median analogue of a result of Behrstock, Hagen and Sisto \cite{bhs}.

As an application, we exhibit an acylindrical action of a graph product on a quasi-tree, generalising results of Kim and Koberda for right-angled Artin groups \cite{kk13,kk}. We show that for many graph products $G$, the action we exhibit is the `largest' acylindrical action of $G$ on a hyperbolic metric space. We use this to show that the graph products of equationally noetherian groups over finite graphs of girth $\geq 6$ are equationally noetherian, generalising a result of Sela \cite{sela}.
\end{abstract}

\maketitle
\setcounter{tocdepth}{1}
\tableofcontents

\section{Introduction} \label{s:intro}

Group actions on CAT(0) cube complexes occupy a central role in geometric group theory. Such actions have been used to study many interesting classes of groups, such as right-angled Artin and Coxeter groups, many small cancellation and 3-manifold groups, and even finitely presented infinite simple groups, constructed by Burger and Mozes in \cite{bm}. Study of CAT(0) cube complexes is aided by their rich combinatorial structure, introduced by Sageev in \cite{sageev}.

In the present paper we study \emph{quasi-median graphs}, which can be viewed as a generalisation of CAT(0) cube complexes; see Definition \ref{d:qm}. In particular, one may think of quasi-median graphs as `CAT(0) prism complexes', consisting of \emph{prisms} -- cartesian products of (possibly infinite dimensional) simplices -- glued together in a non-positively curved way. In his PhD thesis \cite{genthesis}, Genevois introduced cubical-like combinatorial structure and geometry to study a wide class of groups acting on quasi-median graphs, including graph products, certain wreath products, and diagram products.

%Thus, theory developed to study CAT(0) cube complexes may be expected to be of use in the study of quasi-median graphs. In particular, we may utilize hyperplanes and related concepts, introduced in the case of CAT(0) cube complexes by Sageev \cite{sageev}, to study quasi-median graphs, as Genevois does in his thesis \cite{genthesis}.

In particular, given a quasi-median graph $X$, we study \emph{hyperplanes} in $X$: that is, the equivalence classes of edges of $X$, under the equivalence relation generated by letting two edges be equivalent if they induce a square or a triangle. Two hyperplanes are said to \emph{intersect} if two edges defining those hyperplanes are adjacent in a square, and \emph{osculate} if two edges defining those hyperplanes are adjacent but do not belong to a square; see Definition \ref{d:hplanes}. This allows us to define two other graphs related to $X$, which turn out to be useful in the study of groups acting on $X$.

\begin{defn} \label{d:contcross}
Let $X$ be a quasi-median graph. We define the \emph{contact graph} $\mathcal{C}X$ and the \emph{crossing graph} $\Delta X$ as follows. For the vertices, let $V(\mathcal{C}X) = V(\Delta X)$ be the set of hyperplanes of $X$. Two hyperplanes $H,H'$ are then adjacent in $\Delta X$ if and only if $H$ and $H'$ intersect; hyperplanes $H,H'$ are adjacent in $\mathcal{C}X$ if and only if $H$ and $H'$ either intersect or osculate.
\end{defn}

For a CAT(0) cube complex $X$, Hagen has shown that the contact graph $\mathcal{C}X$ is a quasi-tree -- that is, it is quasi-isometric to a tree \cite[Theorem 4.1]{hagen}. Here we generalise this result to quasi-median graphs.

\begin{thmx} \label{t:qtree}
For any quasi-median graph $X$, the contact graph $\mathcal{C}X$ is a quasi-tree, and there exists a $(91,182)$-quasi-isometry from a simplicial tree to $\mathcal{C}X$. In particular, there exists a universal constant $\delta > 0$ such that $\mathcal{C}X$ is $\delta$-hyperbolic for any quasi-median graph $X$.
\end{thmx}

We prove Theorem \ref{t:qtree} in Section \ref{ss:qtree}.

In this paper we study acylindrical hyperbolicity of groups acting on quasi-median graphs. 
\begin{defn} \label{d:ah}
Suppose a group $G$ acts on a metric space $(X,d)$ by isometries. Such an action is said to be \emph{acylindrical} if for every $\varepsilon > 0$, there exist constants $D_\varepsilon,N_\varepsilon > 0$ such that for all $x,y \in X$ with $d(x,y) \geq D_\varepsilon$, the number of elements $g \in G$ satisfying
\begin{equation*} %\label{e:acyldefn}
d(x,x^g) \leq \varepsilon \qquad \text{and} \qquad d(y,y^g) \leq \varepsilon
\end{equation*}
is bounded above by $N_\varepsilon$. Moreover, an action $G \curvearrowright X$ by isometries on a hyperbolic metric space $X$ is said to be \emph{non-elementary} if orbits under this action are unbounded and $G$ is not virtually cyclic.

A group $G$ is then said to be \emph{acylindrically hyperbolic} if it possesses a non-elementary acylindrical action on a hyperbolic metric space.
\end{defn}
Acylindrically hyperbolic groups form a large family, including non-elementary hyperbolic and relatively hyperbolic groups, mapping class groups of most surfaces, and $\mathrm{Out}(F_n)$ for $n \geq 3$ \cite{osinAH}. This family also includes `most' hierarchically hyperbolic groups \cite[Corollary 14.4]{bhs}, and in particular `most' groups $G$ that act properly and cocompactly on a CAT(0) cube complex with a `factor system': see \cite{bhs}. The following result shows that, more generally, many groups acting on quasi-median graphs are acylindrically hyperbolic.

In the following theorem, we say a group action $G \curvearrowright X$ is \emph{special} if there are no two hyperplanes $H,H'$ of $X$ such that $H$ and $H'$ intersect but $H^g$ and $H'$ osculate for some $g \in G$, and there is no hyperplane $H$ that intersects or osculates with $H^g \neq H$ for some $g \in G$. We say a collection $\mathcal{S}$ of sets is \emph{uniformly finite} if there exists a constant $D \in \N$ such that each $S \in \mathcal{S}$ has cardinality $\leq D$.

\begin{thmx} \label{t:main}
Let $G$ be a group acting specially on a quasi-median graph $X$, and suppose vertices in $\Delta X/G$ have uniformly finitely many neighbours.
\begin{enumerate}[label={\normalfont({\roman*})}]
\item \label{i:tmain-qi} If $\Delta X$ is connected and $\Delta X/G$ has finitely many vertices, then the inclusion $\Delta X \hookrightarrow \mathcal{C}X$ is a quasi-isometry.
\item \label{i:tmain-acyl} If stabilisers of vertices under $G \curvearrowright X$ are uniformly finite, then the induced action $G \curvearrowright \mathcal{C}X$ is acylindrical. In particular, if the orbits under $G \curvearrowright \mathcal{C}X$ are unbounded, then $G$ is either virtually cyclic or acylindrically hyperbolic.
\end{enumerate}
\end{thmx}

We prove part \ref{i:tmain-qi} of Theorem \ref{t:main} in Section \ref{ss:qi}, and part \ref{i:tmain-acyl} in Section \ref{s:acyl}.

Note that a large class of examples of group actions on CAT(0) cube complexes with a factor system comes from special actions \cite[Corollaries 8.8 and 14.5]{bhs}. Theorem \ref{t:main} \ref{i:tmain-acyl} generalises this result to quasi-median graphs. We also show that several other hierarchically hyperbolic space-like results on CAT(0) cube complexes generalise to quasi-median graphs: for instance, existence of `hierarchy paths', see \cite[Theorem A (2)]{bhs} and Proposition \ref{p:liftinggeodesics}.

The main application of Theorems \ref{t:qtree} and \ref{t:main} we give is to study graph products of groups. In particular, let $\Gamma$ be a simplicial graph and let $\mathcal{G} = \{ G_v \mid v \in V(\Gamma) \}$ be a collection of non-trivial groups. The \emph{graph product} $\GG$ of the groups $G_v$ over $\Gamma$ is defined as the group
\[
\GG = \left( \mathop{*}_{v \in V(\Gamma)} G_v \right) \bigg/ \llangle g_v^{-1}g_w^{-1}g_vg_w \,\middle|\, g_v \in G_v, g_w \in G_w, (v,w) \in E(\Gamma) \rrangle.
\]
For example, for a complete graph $\Gamma$ we have $\GG \cong \prod_{v \in V(\Gamma)} G_v$, while for discrete $\Gamma$ we have $\GG \cong \mathop{*}_{v \in V(\Gamma)} G_v$. The applicability of the results above to graph products follows from the following result of Genevois.

\begin{thm}[Genevois {\cite[Propositions 8.2 and 8.11]{genthesis}}] \label{t:gengp}
Let $\Gamma$ be a simplicial graph, let $\mathcal{G} = \{ G_v \mid v \in V(\Gamma) \}$ be a collection of non-trivial groups, and let $S = \bigcup_{v \in V(\Gamma)} G_v \setminus \{1\} \subseteq \GG$. Then the Cayley graph $X$ of $\GG$ with respect to $S$ is quasi-median. Moreover, the action of $\GG$ on $X$ is free on vertices and special.
\end{thm}

Moreover, if $\GG$ and $X$ are as in Theorem \ref{t:gengp}, then, as noted before \cite[Lemma 8.8]{genthesis}, the hyperplanes of $X$ are of the form $H_v^g$ for some $g \in \GG$ and for a unique $v \in V(\Gamma)$, where $H_v$ is the hyperplane dual to the clique spanned by $G_v \leq \GG = V(X)$ (meaning $H_v$ is dual to any edge in that clique). Thus, vertices of $\Delta X / \GG$ are of the form $H_v^{\GG}$ for $v \in V(\Gamma)$. It follows from \cite[Lemma 8.12]{genthesis} that $H_u^{\GG}$ and $H_v^{\GG}$ are adjacent in $\Delta X / \GG$ if and only if $u$ and $v$ are adjacent in $\Gamma$, and so we may deduce the following result.
\begin{lem}[Genevois \cite{genthesis}] \label{l:gengp-dxiso}
The graph $\Delta X / \GG$ is isomorphic to $\Gamma$.
\end{lem}

An important subclass of graph products are right-angled Artin groups (RAAGs): indeed, if $G_v \cong \Z$ then $\GG$ is the RAAG associated to $\Gamma$. In this case, a vertex $v \in V(\Gamma)$ is usually identified with a generator of $G_v$. In \cite{kk13} Kim and Koberda constructed the \emph{extension graph} $\Gamma^e$ of a RAAG $G = \GG$ as a graph with vertex set $V(\Gamma^e) = \{ v^g \in G \mid g \in G, v \in V(\Gamma) \}$, where $g^v$ and $h^w$ are adjacent in $\Gamma^e$ if and only if they commute as elements of $G$. This graph turns out to be the same as the crossing graph $\Delta X$ of the Cayley graph $X$ defined in Theorem \ref{t:gengp}.

In fact, Kim and Koberda showed that, given that $|V(\Gamma)| \geq 2$ and both $\Gamma$ and its complement $\Gamma^C$ are connected, $\Gamma^e$ is quasi-isometric to a tree \cite{kk13} and the action of $G$ on $\Gamma^e$ by conjugation is non-elementary acylindrical \cite{kk}. In this paper we generalise these results to arbitrary graph products; this follows as a special case of Theorems \ref{t:qtree} and \ref{t:main}. As a special case, we recover hyperbolicity of the extension graph $\Gamma^e$ and acylindricity of the action $\GG \curvearrowright \Gamma^e$, providing an alternative (shorter and more geometric) argument to the ones presented in \cite{kk13,kk}. In the following corollary, a graph $\Gamma$ is said to have \emph{bounded degree} if there exists a constant $D \in \N$ such that each vertex of $\Gamma$ has degree $\leq D$.
\begin{corx} \label{c:gp}
Let $\Gamma$ be a simplicial graph, let $\mathcal{G} = \{ G_v \mid v \in V(\Gamma) \}$ be a collection of non-trivial groups, and let $X$ be the quasi-median graph defined in Theorem \ref{t:gengp}. Then $\mathcal{C}X$ is a quasi-tree, and if $\Gamma$ has bounded degree then the induced action $\GG \curvearrowright \mathcal{C}X$ is acylindrical. Moreover, if $|V(\Gamma)| \geq 2$ and the complement $\Gamma^C$ of $\Gamma$ is connected, then either $\Gamma\mathcal{G} \cong C_2 * C_2$ is the infinite dihedral group, or this action is non-elementary.
\end{corx}
The hyperbolicity of $\mathcal{C}X$ and the acylindricity of the action follow immediately from Theorems \ref{t:qtree}, \ref{t:main}, \ref{t:gengp} and Lemma \ref{l:gengp-dxiso}, while non-elementarity is shown in Section \ref{ss:nonelem}.

It is worth noting that Minasyan and Osin have already shown in \cite{mo} that if $|V(\Gamma)| \geq 2$ and the complement of $\Gamma$ is connected, then $\Gamma\mathcal{G}$ is either infinite dihedral or acylindrically hyperbolic. However, their proof is not direct and does not provide an explicit acylindrical action on a hyperbolic space. The aim of Corollary \ref{c:gp} is to describe such an action.

We also show that in many cases the action of $\GG$ on $\mathcal{C}X$ is, in the sense of Abbott, Balasubramanya and Osin \cite{abo}, the `largest' acylindrical action of $\GG$ on a hyperbolic metric space: see Section \ref{ss:AHacc}. In particular, we show that many graph products are strongly $\mathcal{AH}$-accessible. This generalises the analogous result for right-angled Artin groups \cite[Theorem 2.18 (c)]{abo}.

\begin{corx} \label{c:AHacc}
Let $\Gamma$ be a finite simplicial graph and let $\mathcal{G} = \{ G_v \mid v \in V(\Gamma) \}$ be a collection of infinite groups. Suppose that for each isolated vertex $v \in V(\Gamma)$, the group $G_v$ is strongly $\mathcal{AH}$-accessible. Then $\GG$ is strongly $\mathcal{AH}$-accessible. Furthermore, if $\Gamma$ has no isolated vertices, then the action $\GG \curvearrowright \mathcal{C}X$, where $X$ is as in Theorem \ref{t:gengp}, is the largest acylindrical action of $\GG$ on a hyperbolic metric space.
\end{corx}

We prove Corollary \ref{c:AHacc} in Section \ref{ss:AHacc}.

\begin{rmk}
After the first version of this preprint was made available, it has been brought to the author's attention that most of the results stated in Corollary \ref{c:gp} follow from the results in \cite{gen1,gen2,gm}. Moreover, a special case of Corollary \ref{c:AHacc} (when the vertex groups $G_v$ are hierarchically hyperbolic) follows from the results in \cite{abd,br}. See Remarks \ref{r:ah} and \ref{r:AHacc} for details.
\end{rmk}

As an application, we use Corollary \ref{c:gp} to study the class of equationally noetherian groups, defined as follows.

\begin{defn} \label{d:en}
Given $n \in \N$, let $F_n$ denote the free group of rank $n$ with a free basis $X_1,\ldots,X_n$. Given a group $G$, an element $s \in F_n$ and a tuple $(g_1,\ldots,g_n) \in G^n$, we write $s(g_1,\ldots,g_n) \in G$ for the element obtained by replacing every occurence of $X_i$ in $s$ with $g_i$, and evaluating the resulting word in $G$. Given a subset $S \subseteq F_n$, the \emph{solution set} of $S$ in $G$ is
\[
V_G(S) = \{ (g_1,\ldots,g_n) \in G^n \mid s(g_1,\ldots,g_n) = 1 \text{ for all } s \in S \}.
\]
A group $G$ is said to be \emph{equationally noetherian} if for any $n \in \N$ and any subset $S \subseteq F_n$, there exists a \emph{finite} subset $S_0 \subseteq S$ such that $V_G(S_0) = V_G(S)$.
\end{defn}

Many classes of groups are known to be equationally noetherian. For example, groups that are linear over a field -- in particular, right-angled Artin groups -- are equationally noetherian \cite[Theorem B1]{bmr}. It is easy to see that the class of equationally noetherian groups is preserved under taking subgroups and direct products; a deep and non-trivial argument shows that the same is true for free products:

\begin{thm}[Sela {\cite[Theorem 9.1]{sela}}] \label{t:sela}
Let $G$ and $H$ be equationally noetherian groups. Then $G * H$ is equationally noetherian.
\end{thm}

Using methods of Groves and Hull developed for acylindrically hyperbolic groups \cite{gh}, we generalise Theorem \ref{t:sela} to a wider class of graph products.

\begin{thmx} \label{t:en}
Let $\Gamma$ be a finite simplicial graph of girth $\geq 6$, and let $\mathcal{G} = \{ G_v \mid v \in V(\Gamma) \}$ be a collection of equationally noetherian groups. Then the graph product $\GG$ is equationally noetherian.
\end{thmx}

We prove Theorem \ref{t:en} in Section \ref{s:eqn}.

The paper is structured as follows. In Section \ref{s:prelim}, we define quasi-median graphs and give several results that are used in later sections. In Section \ref{s:geom}, we analyse the geometry of the contact graph and its relation to crossing graph, and prove Theoren \ref{t:qtree} and Theorem \ref{t:main} \ref{i:tmain-qi}. In Section \ref{s:acyl}, we consider the action of a group $G$ on a quasi-median graph $X$, and prove Theorem \ref{t:main} \ref{i:tmain-acyl}. In Section \ref{s:gp}, we consider the particular case when $G = \GG$ is a graph product and $X$ is the quasi-median graph associated to it, and deduce Corollaries \ref{c:gp} and \ref{c:AHacc}. In Section \ref{s:eqn}, we apply these results to prove Theorem \ref{t:en}.

\begin{ack}
I am deeply grateful to Anthony Genevois for his PhD thesis filled with many great ideas, for discussions which inspired the current argument and for his comments on this manuscript. I would also like to thank Jason Behrstock, Daniel Groves, Mark Hagen, Michael Hull, Thomas Koberda, Armando Martino and Ashot Minasyan for valuable discussions. A special thanks goes to the anonymous referee for their insightful comments and for pointing out a gap in the proof of the previous version of Theorem \ref{t:en}.
\end{ack}

\section{Preliminaries} \label{s:prelim}

Throughout the paper, we use the following conventions and notation. By a graph $X$, we mean an undirected simple (simplicial) graph, and we write $V(X)$ and $E(X)$ for the vertex and edge sets of $X$, respectively. Moreover, we write $d_X(-,-)$ for the combinatorial metric on $X$ -- thus, we view $X$ as a geodesic metric space. We consider the set $\N$ of natural numbers to include $0$.

Given a group $G$, all actions of $G$ on a set $X$ are considered to be right actions, $\theta: X \times G \to X$, and are written as $\theta(x,g) = x^g$ or $\theta(x,g) = xg$. Note that this results in perhaps unusual terminology when we consider a Cayley graph $\mathrm{Cay}(G,S)$: in our case it has edges of the form $(g,sg)$ for $g \in G$ and $s \in S$.

\subsection{Quasi-median graphs}

In this section we introduce quasi-median graphs and basic results that we use throughout the paper. Most of the definitions and results in this section were introduced by Genevois in his thesis \cite{genthesis}. We therefore refer the interested reader to \cite{genthesis} for further discussion and results on applications of quasi-median graphs to geometric group theory.

\begin{defn} \label{d:qm}
Let $X$ be a graph, let $x_1,x_2,x_3 \in V(X)$ be three vertices, and let $k \in \N$. We say a triple $(y_1,y_2,y_3) \in V(X)^3$ is a \emph{$k$-quasi-median} of $(x_1,x_2,x_3)$ if (see Figure \ref{f:dqm-qm}):
\begin{enumerate}[label={({\roman*})}]
\item \label{i:dqm-geod} $y_i$ and $y_j$ lie on a geodesic between $x_i$ and $x_j$ for any $i \neq j$;
\item \label{i:dqm-eql} $k = d_X(y_1,y_2)=d_X(y_1,y_3)=d_X(y_2,y_3)$; and
\item $k$ is as small as possible subject to \ref{i:dqm-geod} and \ref{i:dqm-eql}.
\end{enumerate}
We say $(y_1,y_2,y_3) \in V(X)^3$ is a \emph{quasi-median} of $(x_1,x_2,x_3) \in V(X)^3$ if it is a $k$-quasi-median for some $k$. A $0$-quasi-median is called a \emph{median}.

We say a graph $X$ is a \emph{quasi-median graph} if (see Figure \ref{f:dqm-graphs}):
\begin{enumerate}[label={({\roman*})}]
\item every triple of vertices has a unique quasi-median;
\item $K_{1,1,2}$ is not isomorphic to an induced subgraph of $X$; and
\item if $Y \cong C_6$ is a subgraph of $X$ such that the embedding $Y \hookrightarrow X$ is isometric, then the convex hull of $Y$ in $X$ is isomorphic to the $3$-cube.
\end{enumerate}
\end{defn}

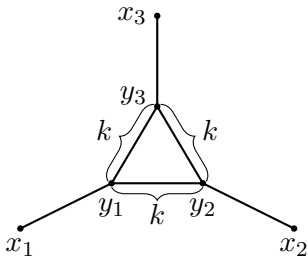
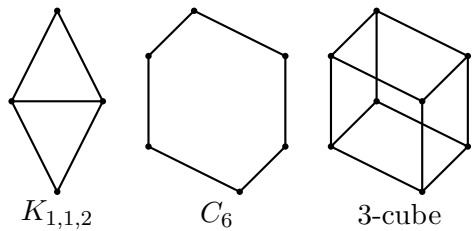
\begin{figure}[ht]
\begin{subfigure}[b]{0.37\textwidth}
\centering
\begin{tikzpicture}[scale=0.6]
\fill (0,0) circle (2pt) node [below] {$x_1$};
\fill (2,1) circle (2pt) node [below,yshift=-2pt] {$y_1$};
\fill (6,0) circle (2pt) node [below] {$x_2$};
\fill (4,1) circle (2pt) node [below,yshift=-2pt] {$y_2$};
\fill (3,4.7) circle (2pt) node [left] {$x_3$};
\fill (3,2.7) circle (2pt) node [left,yshift=4pt] {$y_3$};

\draw [thick] (0,0) -- (2,1);
\draw [thick] (6,0) -- (4,1);
\draw [thick] (3,4.7) -- (3,2.7);
\draw [thick] (4,1) -- (2,1);
\draw [thick] (3,2.7) -- (4,1);
\draw [thick] (2,1) -- (3,2.7);
\draw [decorate,decoration={brace,amplitude=5pt},xshift=0pt,yshift=-2.4pt] (4,1) -- (2,1) node [midway,yshift=-10pt] {$k$};
\draw [decorate,decoration={brace,amplitude=5pt},xshift=2.1pt,yshift=1.2pt] (3,2.7) -- (4,1) node [midway,xshift=10pt,yshift=5pt] {$k$};
\draw [decorate,decoration={brace,amplitude=5pt},xshift=-2.1pt,yshift=1.2pt] (2,1) -- (3,2.7) node [midway,xshift=-10pt,yshift=5pt] {$k$};
\end{tikzpicture}
\caption{A $k$-quasi-median $(y_1,y_2,y_3)$ of $(x_1,x_2,x_3)$.} \label{f:dqm-qm}
\end{subfigure}
\hfill
\begin{subfigure}[b]{0.55\textwidth}
\centering
\begin{tikzpicture}[scale=0.6]
\begin{scope}[xshift=-3cm]
\fill (0,0) circle (2pt);
\fill (2,0) circle (2pt);
\fill (1,2) circle (2pt);
\fill (1,-2) circle (2pt);
\draw [thick] (0,0) -- (2,0) -- (1,2) -- (0,0) -- (1,-2) -- (2,0);
\node at (1,-2.5) {$K_{1,1,2}$};
\end{scope}

\begin{scope}[xshift=0]
\fill (0,1) circle (2pt);
\fill (1,2) circle (2pt);
\fill (3,1) circle (2pt);
\fill (3,-1) circle (2pt);
\fill (2,-2) circle (2pt);
\fill (0,-1) circle (2pt);
\draw [thick] (0,1) -- (1,2) -- (3,1) -- (3,-1) -- (2,-2) -- (0,-1) -- cycle;
\node at (1.5,-2.5) {$C_6$};
\end{scope}

\begin{scope}[xshift=4cm]
\fill (0,1) circle (2pt);
\fill (1,2) circle (2pt);
\fill (3,1) circle (2pt);
\fill (3,-1) circle (2pt);
\fill (2,-2) circle (2pt);
\fill (0,-1) circle (2pt);
\fill (1,0) circle (2pt);
\fill (2,0) circle (2pt);
\draw [thick] (0,1) -- (1,2) -- (3,1) -- (3,-1) -- (2,-2) -- (0,-1) -- cycle;
\draw [thick] (0,1) -- (2,0) -- (3,1);
\draw [thick] (2,0) -- (2,-2);
\draw [thick] (0,-1) -- (1,0) -- (3,-1);
\draw [thick] (1,0) -- (1,2);
\node at (1.5,-2.5) {$3$-cube};
\end{scope}
\end{tikzpicture}
\caption{The graphs $K_{1,1,2}$, $C_6$ and the $3$-cube.\\~} \label{f:dqm-graphs}
\end{subfigure}
\caption{Graphs appearing in Definition \ref{d:qm}.} \label{f:dqm}
\end{figure}

There are many equivalent characterisations of quasi-median graphs: see \cite[Theorem 1]{bmw}. In this paper we think of quasi-median graphs as generalisations of median graphs. Recall that a graph $X$ is called a \emph{median graph} if every triple of vertices of $X$ has a unique median. In particular, every median graph is quasi-median; more precisely, it is known that a graph is median if and only if it is quasi-median and triangle-free: see \cite[Corollary 2.92]{genthesis}, for instance.

\begin{rmk}
The class of \emph{quasi-median graphs} is reminiscent of the class of a \emph{bucolic graphs}, introduced in \cite{bccgo}, in that both of these generalise median graphs to a class that also includes finite complete graphs. Neither of these two classes includes the other: in essence, bucolic graphs differ from quasi-median ones in that they do not contain infinite complete subgraphs, but may contain an induced subgraph $K_{1,1,2}$ (see Figure \ref{f:dqm-graphs}); compare \cite[Definition 2.1]{genthesis} and \cite[Definition 2.11]{bccgo}. We chose to work with quasi-median graphs as they are more suitable for our applications: for instance, the Cayley graph $X$ of a graph product $\GG$ defined in Theorem \ref{t:gengp} is quasi-median, but not bucolic unless $\mathcal{G}$ is a collection of \emph{finite} groups. %We also believe that exclusion of induced subgraphs $K_{1,1,2}$ makes the hyperplanes in quasi-median graphs (see Definition \ref{d:hplanes}) more tractable.
\end{rmk}

In what follows, a \emph{clique} is a maximal complete subgraph, a \emph{triangle} is a complete graph on $3$ vertices, and a \emph{square} is a complete bipartite graph on two sets of $2$ vertices each.

\begin{defn} \label{d:hplanes}
Let $X$ be a quasi-median graph. Let $\sim$ be the equivalence relation on $E(X)$ generated by the equivalences $e \sim f$ when $e$ and $f$ either are two sides of a triangle or opposite sides of a square. A \emph{hyperplane} $H$ is an equivalence class $[e]$ for some $e \in E(X)$; in this case, we say $H$ is the hyperplane \emph{dual} to $e$ (or, alternatively, $H$ is the hyperplane dual to any clique containing $e$). Given a hyperplane $H$ dual to $e \in E(X)$, the \emph{carrier} of $H$, denoted by $\mathcal{N}(H)$, is the full subgraph of $X$ induced by $[e] \subseteq E(X)$; a \emph{fibre} of $H$ is a connected component of $\mathcal{N}(H) \setminus J$, where $J$ is the union of the interiors of all the edges in $[e]$.

Given two edges $e,e' \in E(X)$ with a common endpoint ($p$, say) that do not belong to the same clique, let $H$ and $H'$ be the hyperplanes dual to $e$ and $e'$, respectively. We then say $H$ and $H'$ \emph{intersect} (or \emph{intersect at $p$}) if $e$ and $e'$ are adjacent edges in a square, and we say $H$ and $H'$ \emph{osculate} (or \emph{osculate at $p$}) otherwise.

Finally, given two vertices $p,q \in V(X)$ and a hyperplane $H$, we say $H$ \emph{separates} $p$ from $q$ if every path between $p$ and $q$ contains an edge dual to $H$. More generally, we say $H$ \emph{separates} two subgraphs $P,Q \subseteq X$ if $H$ does not separate any two vertices of $P$ or any two vertices of $Q$, but it separates a vertex of $P$ from a vertex of $Q$.
Given a path $\gamma$ in $X$, we also say $H$ \emph{crosses} $\gamma$ if $\gamma$ contains an edge dual to $H$.
\end{defn}

Another important concept in the study of quasi-median graphs are gated subgraphs. Such subgraphs coincide with \emph{convex subgraphs} for median graphs, but in general form a larger class in quasi-median graphs.

\begin{defn} \label{d:gated}
Let $X$ be a quasi-median graph, let $Y \subseteq X$ be a full subgraph, and let $v \in V(X)$. We say $p \in V(Y)$ is a \emph{gate} for $v$ in $Y$ if, for any $q \in V(Y)$, there exists a geodesic in $X$ between $v$ and $q$ passing through $p$. We say a full subgraph $Y \subseteq X$ is a \emph{gated subgraph} if every vertex of $X$ has a gate in $Y$.
\end{defn}

The following result says that the subgraphs of interest to us are gated. Here, by convention, given two graphs $Y$ and $Z$ we denote by $Y \times Z$ the $1$-skeleton of the square complex obtained as a cartesian product of $Y$ and $Z$. The facts that $C$, $\mathcal{N}(H)$ and $F$ are gated subgraphs follow from Lemmas 2.16, 2.24 and 2.29 (respectively) in \cite{genthesis}, while the facts about the isomorphism $\Psi: \mathcal{N}(H) \to F \times C$ follow from its construction as well as Lemmas 2.28 and 2.29 in \cite{genthesis}. Note that, in our terminilogy, a `fibre' is a priori different from a `main fibre' defined in \cite[Definition 2.27]{genthesis}; however, these two graphs are isomorphic by \cite[Lemma 2.29]{genthesis}.

\begin{prop}[Genevois {\cite[Section 2.2]{genthesis}}] \label{p:cfgated}
Let $X$ be a quasi-median graph, $H$ a hyperplane dual to a clique $C$, and $F$ a fibre of $H$. Then $\mathcal{N}(H)$, $C$ and $F$ are gated subgraphs of $X$. Moreover, there exists a graph isomorphism $\Psi: \mathcal{N}(H) \to F \times C$, and the cliques dual to $H$ (respectively the fibres of $H$) are precisely the subgraphs $\Psi^{-1}(\{p\} \times C)$ for vertices $p \in V(F)$ (respectively $\Psi^{-1}(F \times \{p\})$ for vertices $p \in V(C)$).
\end{prop}

\subsection{Special actions}

In this section we describe the hypotheses that we impose on group actions on quasi-median graphs. We first define what it means for an action on a quasi-median graph to be special.

\begin{defn}
Let $X$ be a quasi-median graph, and let $G$ be a group acting on it by graph isomorphisms. We say the action $G \curvearrowright X$ is \emph{special} if
\begin{enumerate}[label={({\roman*})}]
\item no two hyperplanes in the same orbit under $G \curvearrowright X$ intersect or osculate; and
\item given two hyperplanes $H$ and $H'$ that intersect, $H^g$ and $H'$ do not osculate for any $g \in G$.
\end{enumerate}
\end{defn}

Special actions on CAT(0) cube complexes were introduced by Haglund and Wise in \cite{hw}. Notably, there it is shown that, in our terminology, if a group $G$ acts specially, cocompactly and without `orientation-inversions' of hyperplanes on a CAT(0) cube complex $X$, then the fundamental group of the quotient $X/G$ embeds in a right-angled Artin group. 

It is clear from \cite[Lemma 2.25]{genthesis} that no hyperplane in a quasi-median graph can self-intersect or self-osculate. The next lemma says that, moreover, the action of the trivial group on a quasi-median graph is special. %This result should be obvious to experts -- however, the author has not been able to find this result in the literature, so the proof of the lemma is given here for completeness. 
Recall that two hyperplanes are said to \emph{interosculate} if they both intersect and osculate.

\begin{lem} \label{l:nointosc}
In a quasi-median graph $X$, no two hyperplanes can interosculate.
\end{lem}

\begin{proof}
Suppose for contradiction that hyperplanes $H$ and $H'$ intersect at $p$ and osculate at $q$ for some $p,q \in V(X)$, and assume without loss of generality that $p$ and $q$ are chosen in such a way that $d_X(p,q)$ is as small as possible. It is clear that $p \neq q$: see, for instance, \cite[Lemma 2.13]{genthesis}. On the other hand, since $\mathcal{N}(H)$ and $\mathcal{N}(H')$ are gated (and therefore convex) by Proposition \ref{p:cfgated}, and as $p,q \in \mathcal{N}(H) \cap \mathcal{N}(H')$, it follows that a geodesic between $p$ and $q$ lies in $\mathcal{N}(H) \cap \mathcal{N}(H')$. In particular, if $r$ is a vertex on this geodesic, then $H$ and $H'$ either intersect at $r$ or osculate at $r$; by minimality of $d_X(p,q)$, it then follows that $d_X(p,q)=1$.

Let $e$ be the edge joining $p$ and $q$, and let $K$ be the hyperplane dual to $e$. It follows from Proposition \ref{p:cfgated} that $K \neq H$ and $K \neq H'$: indeed, if we had $K=H$ (say), then $K=H$ and $H'$ would intersect at $q$, contradicting the choice of $q$. Thus $K$ is distinct from $H$ and $H'$, and so $e$ belongs to a fibre of $H$ and a fibre of $H'$. It then follows from Proposition \ref{p:cfgated} that $K$ intersects both $H$ and $H'$ at $q$, and that the graph $Y$ shown in Figure \ref{f:lnio} is a subgraph of $X$.

We now claim that the embedding $Y \hookrightarrow X$ is isometric. Indeed, as $H$, $H'$ and $K$ are distinct hyperplanes, no two vertices $p',q' \in V(Y)$ with $d_Y(p',q')=2$ can be joined by an edge in $X$, as that would create a triangle in $X$ with edges dual to different hyperplanes. It is thus enough to show that if $p',q' \in V(Y)$ and $d_Y(p',q')=3$, then $d_X(p',q')=3$. Up to relabelling $H$, $H'$ and $K$, we may assume without loss of generality that $p'=s$ and $q'=q$. Now it is clear that $d_X(s,q) \neq 1$: otherwise, $p_1s$ and $q_1q$ are opposite sides in a square in $X$, contradicting the fact that $H \neq H'$. Thus, suppose for contradiction that $d_X(s,q) = 2$. But then the triple $(p_1,s,t)$ is a quasi-median of $(p_1,s,q)$ for some vertex $t \in V(X)$, and the edges $p_1s$, $p_1t$, $q_1q$ are dual to the same hyperplane, again contradicting the fact that $H \neq H'$. Thus the embedding $Y \hookrightarrow X$ is isometric, as claimed.

But now the embedding of the $C_6 \subseteq Y$ formed by vertices $s$, $p_1$, $q_1$, $q$, $q_2$ and $p_2$ into $X$ is also isometric, and so the convex hull of this $C_6$ in $X$ is a $3$-cube. Thus there exists a vertex $u \in V(X)$ joined by edges to $s$, $p_2$ and $q_2$. This implies that $H$ and $H'$ intersect at $q$, contradicting the choice of $q$. Thus $H$ and $H'$ cannot interosculate.
\end{proof}

\begin{figure}[ht]
\centering
\begin{tikzpicture}
\draw [red,thick] (0,1) -- (2,2) node [midway,above] {$H\ $};
\draw [blue,thick] (2,2) -- (4,1) node [midway,above] {$\ H'$};
\draw [green!50!black,thick] (4,1) -- (4,-1) node [midway,right] {$K$};
\draw [red,thick] (4,-1) -- (2,-2);
\draw [blue,thick] (2,-2) -- (0,-1);
\draw [green!50!black,thick] (0,-1) -- (0,1);
\draw [blue,thick] (0,1) -- (2,0);
\draw [red,thick] (2,0) -- (4,1);
\draw [green!50!black,thick] (2,0) -- (2,-2) node [midway,right,black] {$e$};
\draw [dashed] (0,1) -- (1,0);
\draw [dashed] (2,2) -- (1,0);
\draw [dashed] (2,-2) -- (1,0);
\fill (0,1) circle (1.5pt) node [left] {$p_1$};
\fill (2,2) circle (1.5pt) node [above] {$s$};
\fill (4,1) circle (1.5pt) node [right] {$p_2$};
\fill (4,-1) circle (1.5pt) node [right] {$q_2$};
\fill (2,-2) circle (1.5pt) node [below] {$q$};
\fill (0,-1) circle (1.5pt) node [left] {$q_1$};
\fill (1,0) circle (1pt) node [below] {$t$};
\fill (2,0) circle (1.5pt) node [above] {$p$};
\end{tikzpicture}
\caption{Proof of Lemma \ref{l:nointosc}: the graph $Y$ (solid edges) and the vertex $t \in V(X)$.} \label{f:lnio}
\end{figure}
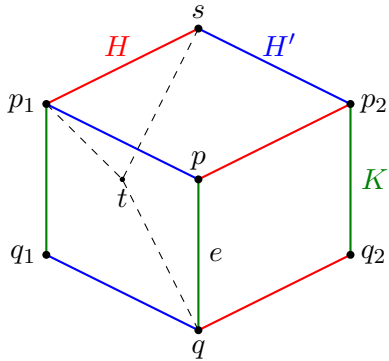

\begin{rmk} \label{r:swap}
We use Lemma \ref{l:nointosc} in the following setting. Let $\gamma$ be a geodesic in a quasi-median graph $X$, let $e$ and $e'$ be two consecutive edges of $\gamma$, and let $H$ and $H'$ be the hyperplanes dual to $e$ and $e'$, respectively. Suppose that $H$ and $H'$ intersect. It then follows from Lemma \ref{l:nointosc} that $H$ and $H'$ cannot osculate at the common endpoint $p$ of $e$ and $e'$, and therefore $H$ and $H'$ must intersect at $p$. In particular, $X$ contains a square with edges $e$, $e'$, $f$ and $f'$, in which $f$ and $f'$ are the edges opposite to $e$ and $e'$, respectively. We may then obtain another geodesic $\gamma'$ in $X$ (with the same endpoints as $\gamma$) by replacing the subpath $ee'$ of $\gamma$ with $f'f$. We refer to the operation of replacing $\gamma$ by $\gamma'$ as \emph{swapping $e$ and $e'$ on $\gamma$}.
\end{rmk}

%In fact, we will work in slightly more generality, as shown by the following definition.

%\begin{defn} \label{d:woio}
%Let $N \in \N$. A quasi-median graph $X$ is said to be \emph{$N$-interosculation-free} if there exists a partition $V(\mathcal{C}X)=\mathfrak{H}_1 \sqcup \cdots \sqcup \mathfrak{H}_N$ of hyperplanes in $X$ and there exists a function $\mathfrak{f}: \{1,\ldots,N\}^2 \to \{0,1\}$ such that, given any hyperplanes $A_i \in \mathfrak{H}_i$ and $A_j \in \mathfrak{H}_j$ with $\mathcal{N}(A_i) \cap \mathcal{N}(A_j) \neq \varnothing$, $A_i$ and $A_j$ intersect if $\mathfrak{f}(i,j)=1$ and osculate if $\mathfrak{f}(i,j)=0$. We say $X$ is \emph{interosculation-free} if it is $N$-interosculation-free for some $N$.
%\end{defn}

%Note that if a group $G$ acts on a quasi-median graph $X$ weakly specially with finitely many orbits of hyperplanes, then $X$ is interosculation-free: indeed, for the partition $\mathfrak{H}_1 \sqcup \cdots \sqcup \mathfrak{H}_N$ we may simply take the partition of $V(\mathcal{C}X)$ into orbits under $G$.

\subsection{Geodesics in quasi-median graphs}

Here we record two results on geodesics in a quasi-median graph. The first one of these is due to Genevois.

\begin{prop}[Genevois {\cite[Proposition 2.30]{genthesis}}] \label{p:geod}
A path in a quasi-median graph $X$ is a geodesic if and only if it intersects any hyperplane at most once. In particular, the distance between two vertices of $X$ is equal to the number of hyperplanes separating them. \qed
\end{prop}

\begin{lem} \label{l:cancellingplane}
Let $p,q,r \in V(X)$ be vertices of a quasi-median graph $X$ such that some hyperplane separates $q$ from $p$ and $r$. Then there exists a hyperplane $C$ separating $q$ from $p$ and $r$ and geodesics $\gamma_p$ (respectively $\gamma_r$) between $q$ and $p$ (respectively $q$ and $r$) such that $q$ is an endpoint of the edges of $\gamma_p$ and $\gamma_r$ dual to $C$.
\end{lem}

\begin{proof}
Let $C$ be a hyperplane separating $q$ from $p$ and $r$, let $\gamma_p$ (respectively $\gamma_r$) be a geodesic between $q$ and $p$ (respectively $q$ and $r$), and let $c_p$ and $c_r$ be the edges of $\gamma_p$ and $\gamma_r$ (respectively) dual to $C$. Let $q_p$ and $q_p'$, $q_r$ and $q_r'$ be the endpoints of $c_p$, $c_r$ (respectively), labelled so that $C$ does not separate $q$, $q_p$ and $q_r$. Suppose, without loss of generality, that $\gamma_p$ and $C$ are chosen in such a way that $d_X(q,q_p)$ is as small as possible, and that $\gamma_r$ is chosen so that $d_X(q,q_r)$ is as small as possible (subject to the choice of $\gamma_p$ and $C$). See Figure \ref{f:lcancellingplane}.

We first claim that $q = q_p$. Indeed, suppose not, and let $c_p' \neq c_p$ be the other edge of $\gamma_p$ with endpoint $q_p$. Let $C_p'$ be the hyperplane dual to $c_p'$. Then $C_p'$ does not separate $q_p$ and $p$ (as $\gamma_p$ is a geodesic), nor $q$ and $r$ (by minimality of $d_X(q,q_p)$), but it separates $q_p$ (and so $p$) from $q$ (and so $r$). On the other hand, $C$ separates $q_p$ from $p$ (as $\gamma_p$ is a geodesic) and $q$ from $r$ (as $\gamma_r$ is a geodesic). Therefore, $C$ and $C_p'$ must intersect. But then we may swap $c_p$ and $c_p'$ on $\gamma_p$ (see Remark \ref{r:swap}), contradicting minimality of $d_X(q,q_p)$. Thus we must have $q = q_p$.

We now claim that $q = q_r$. Indeed, suppose not, and let $c_r' \neq c_r$ be the other edge of $\gamma_r$ with endpoint $q_r$. Let $C_r'$ be the hyperplane dual to $c_r'$. Then $C_r'$ does not separate $q$ and $q_p'$ (as $C$ is the only hyperplane separating $q=q_p$ and $q_p'$), nor $q_r$ and $r$ (as $\gamma_r$ is a geodesic), but it separates $q$ (and so $q_p'$) from $q_r$ (and so $r$). On the other hand, $C$ separates $q_r$ from $r$ (as $\gamma_r$ is a geodesic) and $q$ from $q_p'$. Therefore, $C$ and $C_r'$ must intersect. But then we may swap $c_r$ and $c_r'$ on $\gamma_r$, contradicting minimality of $d_X(q,q_r)$. Thus we must have $q = q_r$.
\end{proof}

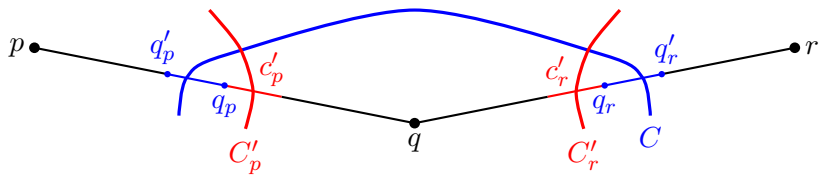
\begin{figure}[ht]
\begin{tikzpicture}
\draw [thick] (0,0) -- (1.75,-0.35);
\draw [thick] (3.25,-0.65) -- (5,-1) -- (6.75,-0.65);
\draw [thick] (8.25,-0.35) -- (10,0);
\draw [blue,thick] (1.75,-0.35) -- (2.5,-0.5);
\draw [blue,thick] (8.25,-0.35) -- (7.5,-0.5);
\draw [blue,very thick] plot [smooth] coordinates { (1.9,-0.9) (2,-0.4) (2.5,-0.1) (5,0.5) (7.5,-0.1) (8,-0.4) (8.1,-0.9) } node [below] {$C$};
\draw [red,thick] (2.5,-0.5) -- (3.25,-0.65) node [above] {$c_p'\,\,\,\,$};
\draw [red,thick] (7.5,-0.5) -- (6.75,-0.65) node [above] {$\,\,\,\,c_r'$};
\draw [red,very thick] plot [smooth] coordinates { (2.3,0.5) (2.7,0) (2.875,-0.575) (2.775,-1.075) } node [below] {$C_p'$};
\draw [red,very thick] plot [smooth] coordinates { (7.7,0.5) (7.3,0) (7.125,-0.575) (7.225,-1.075) } node [below] {$C_r'$};

\fill (0,0) circle (2pt) node [left] {$p$};
\fill (5,-1) circle (2pt) node [below] {$q$};
\fill (10,0) circle (2pt) node [right] {$r$};
\fill [blue] (1.75,-0.35) circle (1.2pt) node [above] {$q_p'\,\,$};
\fill [blue] (2.5,-0.5) circle (1.2pt) node [below] {$q_p$};
\fill [blue] (7.5,-0.5) circle (1.2pt) node [below] {$q_r$};
\fill [blue] (8.25,-0.35) circle (1.2pt) node [above] {$\,\,q_r'$};
\end{tikzpicture}
\caption{Proof of Lemma \ref{l:cancellingplane}.} \label{f:lcancellingplane}
\end{figure}

\section{Geometry of the contact graph} \label{s:geom}

Here we analyse the geometry of the contact graph $\mathcal{C}X$ of a quasi-median graph $X$. In Section \ref{ss:qi} we show that, under certain conditions, $\mathcal{C}X$ is quasi-isometric to $\Delta X$, and prove Theorem \ref{t:main} \ref{i:tmain-qi}. In Section \ref{ss:qtree} we prove that $\mathcal{C}X$ is a quasi-tree (Theorem \ref{t:qtree}).

\subsection{Contact and crossing graphs} \label{ss:qi}

%In this section we will see that, under certain conditions, the contact graph $\mathcal{C}X$ and the crossing graph $\Delta X$ of a quasi-median graph $X$ are quasi-isometric.

The following proposition allows us to lift geodesics in $\mathcal{C}(X)$ back to $X$. This generalises the existence of `hierarchy paths' in CAT(0) cube complexes \cite[Theorem A(2)]{bhs} to arbitrary quasi-median graphs. Moreover, the same result applies when $\mathcal{C}X$ is replaced by $\Delta X$, as long as $\Delta X$ is connected.

\begin{prop} \label{p:liftinggeodesics}
Let $\Gamma = \mathcal{C}X$ or $\Gamma = \Delta X$, and let $A,B \in V(\Gamma)$ be hyperplanes in the same connected component of $\Gamma$. Let $p \in V(X)$ (respectively $q \in V(X)$) be a vertex in $\mathcal{N}(A)$ (respectively $\mathcal{N}(B)$). Then there exists a geodesic $A=A_0,\ldots,A_m=B$ in $\Gamma$ and vertices $p=p_0,\ldots,p_{m+1}=q \in V(X)$ such that $p_i \in \mathcal{N}(A_{i-1}) \cap \mathcal{N}(A_i)$ for $1 \leq i \leq m$ and $d_X(p,q) = \sum_{i=0}^m d_X(p_i,p_{i+1})$.
\end{prop}

\begin{proof}
By assumption, there exists a geodesic $A=A_0,A_1,\ldots,A_m=B$ in $\Gamma$. For $1 \leq i \leq m$, let $p_i \in V(X)$ be a vertex in the carriers of both $A_{i-1}$ and $A_i$, and let $p_0 = p$, $p_{m+1} = q$. Suppose that the $A_i$ and the $p_i$ are chosen in such a way that $D = \sum_{i=0}^m d_X(p_i,p_{i+1})$ is as small as possible. We claim that $D = d_X(p,q)$.

Let $\gamma_i$ be a geodesic between $p_i$ and $p_{i+1}$ for $0 \leq i \leq m$. Suppose for contradiction that $D > d_X(p,q)$: this means that $\gamma_0 \gamma_1 \cdots \gamma_m$ is not a geodesic. Therefore, there exists a hyperplane $C$ separating $p_i$ and $p_{i+1}$ as well as $p_j$ and $p_{j+1}$ for some $i < j$.
Let $c_i$ (respectively $c_j$) be the edge of $\gamma_i$ (respectively $\gamma_j$) dual to $C$.

As hyperplane carriers are gated (and therefore convex), any hyperplane separating $p_i$ and $p_{i+1}$ either is or intersects $A_i$ for $0 \leq i \leq m$. Now note that $j-i \leq 2$: indeed, we have $d_\Gamma(A_i,C) \leq 1$ and $d_\Gamma(A_j,C) \leq 1$, so $j-i = d_\Gamma(A_i,A_j) \leq 1+1 = 2$. In particular, $j-i \in \{1,2\}$.

We now claim that $j = i+1$. Indeed, suppose for contradiction that $j = i+2$. Let $p_{i+1}'$ (respectively $p_{i+2}'$) be the endpoint of $c_i$ (respectively $c_{i+2}$) closer to $p_i$ (respectively $p_{i+3}$). Then we have
\begin{equation} \label{e:manydX}
\begin{aligned}
&d_X(p_i,p_{i+1}) + d_X(p_{i+1},p_{i+2}) + d_X(p_{i+2},p_{i+3}) \\ &\qquad = d_X(p_i,p_{i+1}') + d_X(p_{i+1}',p_{i+1}) + d_X(p_{i+1},p_{i+2}) \\ &\qquad\qquad+ d_X(p_{i+2},p_{i+2}') + d_X(p_{i+2}',p_{i+3}) \\ &\qquad \geq d_X(p_i,p_{i+1}') + d_X(p_{i+1}',p_{i+2}') + d_X(p_{i+2}',p_{i+3}),
\end{aligned}
\end{equation}
with equality if and only if $\gamma_i'\gamma_{i+1}\gamma_{i+2}'$ is a geodesic, where $\gamma_i'$ (respectively $\gamma_{i+2}'$) is the portion of $\gamma_i$ (respectively $\gamma_{i+2}$) between $p_{i+1}'$ and $p_{i+1}$ (respectively $p_{i+2}$ and $p_{i+2}'$). But $\gamma_i'\gamma_{i+1}\gamma_{i+2}'$ cannot be a geodesic as it passes through two edges dual to $C$, and so strict inequality in \eqref{e:manydX} holds. We may then replace $A_{i+1}$, $p_{i+1}$ and $p_{i+2}$ with $C$, $p_{i+1}'$ and $p_{i+2}'$, respectively, contradicting minimality of $D$. Thus $j = i+1$, as claimed.

Therefore, $C$ separates $p_{i+1}$ from $p_i$ and $p_{i+2}$. By Lemma \ref{l:cancellingplane}, we may assume (after modifying $C$, $\gamma_i$ and $\gamma_{i+1}$ if necessary) that $p_{i+1}$ is an endpoint of both $c_i$ and $c_{i+1}$. As $c_i$ and $c_{i+1}$ are dual to the same hyperplane, it follows that they belong to the same clique. In particular (as carriers of hyperplanes are gated and so contain their triangles) this whole clique belongs to $\mathcal{N}(A_i) \cap \mathcal{N}(A_{i+1})$. If $r_{i+1} \neq p_{i+1}$ is the other endpoint of $c_i$, then $d_X(p_i,r_{i+1}) < d_X(p_i,p_{i+1})$ and $d_X(r_{i+1},p_{i+2}) \leq d_X(p_{i+1},p_{i+2})$. We may therefore replace $p_{i+1}$ by $r_{i+1}$, contradicting minimality of $D$. Thus $D = d_X(p,q)$, as claimed.
\end{proof}

Taking $\Gamma=\Delta X$ and $p=q$ in Proposition \ref{p:liftinggeodesics} immediately gives the following.

\begin{cor} \label{c:localosc}
Let $A,B \in V(\Delta X)$ be hyperplanes in the same connected component of $\Delta X$ osculating at a point $p \in V(X)$. Then there exists a geodesic $A=A_0,\ldots,A_m=B$ in $\Delta X$ such that $A_{i-1}$ and $A_i$ intersect at $p$ for $1 \leq i \leq m$. \qed
\end{cor}

\begin{lem} \label{l:nolongosc}
Suppose a group $G$ acts on $X$ specially with $N$ orbits of hyperplanes. Let $A$ and $B$ be hyperplanes that osculate and belong to the same connected component of $\Delta X$. Then $d_{\Delta X}(A,B) \leq \max\{2,N-1\}$.
\end{lem}

\begin{proof}
Let $p \in V(X)$ be such that $A$ and $B$ osculate at $p$. By Corollary \ref{c:localosc}, there exists a geodesic $A = A_0,A_1,\ldots,A_m = B$ in $\Delta X$ such that $A_{i-1}$ and $A_i$ intersect at $p$ for each $i$. Let $i_1,\ldots,i_k \in \N$, satisfying $0 = i_1 < i_2 < \cdots < i_k = m+1$, be such that $A_{i_j}^G = A_{i_{j+1}-1}^G$ for $1 \leq j \leq k-1$ (for instance, we may take $i_j = j-1$). Suppose this is done so that $k$ is as small as possible. Clearly, this implies $A_{i_j}^G \neq A_{i_{j'}}^G$ whenever $1 \leq j < j' \leq k-1$: otherwise, we may replace $i_1,\ldots,i_k$ by $i_1,\ldots,i_j,i_{j'+1},\ldots,i_k$, contradicting minimality of $k$. In particular, $k \leq N+1$; as $m \geq 2$, note also that $k \geq 2$. We will consider the cases $k = 2$ and $k \geq 3$ separately.

Suppose first that $k \geq 3$. We claim that $i_{j+1}-i_j \leq 1$ whenever $1 \leq j \leq k-1$. Indeed, note that whenever $1 \leq j \leq k-2$, $p \in \mathcal{N}(A_{i_j}) \cap \mathcal{N}(A_{i_{j+1}})$, and so $A_{i_j}$ and $A_{i_{j+1}}$ must either intersect or osculate. But $A_{i_{j+1}}$ intersects $A_{i_{j+1}-1}$, and $A_{i_j}^G = A_{i_{j+1}-1}^G$: therefore, as the action $G \curvearrowright X$ is special, it follows that $A_{i_j}$ and $A_{i_{j+1}}$ must intersect. In particular, $i_{j+1}-i_j = d_{\Delta X}(A_{i_j},A_{i_{j+1}}) \leq 1$ for $1 \leq j \leq k-2$. For $j = k-1$, we may similarly note that $\mathcal{N}(A_{i_j-1}) \cap \mathcal{N}(A_{i_{j+1}-1}) \neq \varnothing$ and so $A_{i_j-1}$ and $A_m=A_{i_{j+1}-1}$ must intersect: thus $i_{j+1}-i_j = d_{\Delta X}(A_{i_j-1},A_{i_{j+1}-1}) \leq 1$ in this case as well. In particular, we get
\[
d_{\Delta X}(A,B) = m = i_k-i_1-1 = \left( \sum_{j=1}^{k-1} (i_{j+1}-i_j) \right) - 1 \leq k-2 \leq N-1,
\]
as required.

Suppose now that $k = 2$. Similarly to the case $k \geq 3$, we may note that $p \in \mathcal{N}(A_1) \cap \mathcal{N}(A_m)$, and so, as $A_0$ and $A_1$ intersect and as $A_0^G = A_m^G$, it follows that $A_1$ and $A_m$ intersect. Thus $m-1 = d_{\Delta X}(A_1,A_m) \leq 1$ and so $d_{\Delta X}(A,B) = m \leq 2$, as required.
\end{proof}

\begin{proof}[Proof of Theorem \ref{t:main} \ref{i:tmain-qi}]
It is clear that $d_{\mathcal{C}X}(A,B) \leq d_{\Delta X}(A,B)$ for any hyperplanes $A$ and $B$, as $\Delta X$ is a subgraph of $\mathcal{C}X$. Conversely, Lemma \ref{l:nolongosc} implies that $d_{\Delta X}(A,B) \leq \overline{N} d_{\mathcal{C}X}(A,B)$ for any hyperplanes $A$ and $B$, where $\overline{N}=\max\{2,N-1\}$.
\end{proof}

\begin{rmk}
We note that all the assumptions for Theorem \ref{t:main} \ref{i:tmain-qi} are necessary. Indeed, it is clear that $\Delta X$ needs to be connected. To show necessity of the other two conditions, consider the following. Let $G_0 = \langle S \mid R \rangle$ be the group with generators $S = \{ a_{i,j} \mid (i,j) \in \Z^2 \}$ and relators $R = \bigcup_{(i,j) \in \Z^2} \{ a_{i,j}^2, [a_{i,j},a_{i,j+1}], [a_{i,j},a_{i+1,j}] \}$; this is the (infinitely generated) right-angled Coxeter group associated to a `grid' in $\mathbb{R}^2$: a graph $\Gamma$ with $V(\Gamma) = \Z^2$, where $(i,j)$ and $(i',j')$ are adjacent if and only if $|i-i'| + |j-j'| = 1$. Let $X$ be the Cayley graph of $G_0$ with respect to $S$.

Then $X$ is a quasi-median (and, indeed, median) graph by \cite[Proposition 8.2]{genthesis}. Furthermore, by the results in \cite[Chapter 8]{genthesis}, $\Delta X$ is connected, and if $H_{i,j}$ is the hyperplane dual to the edge $(1,a_{i,j})$ of $X$ (for $(i,j) \in \Z^2$) then $d_{\mathcal{C}X}(H_{0,0},H_{i,j}) \leq 1$ but $d_{\Delta X}(H_{0,0},H_{i,j}) = |i|+|j|$ for all $(i,j)$. Thus the inclusion $\Delta X \hookrightarrow \mathcal{C}X$ cannot be a quasi-isometry. Moreover, by Theorem \ref{t:qtree}, we know that $\mathcal{C}X$ is a quasi-tree, whereas the inclusion into $\Delta X$ of the subgraph spanned by $\{ H_{i,j} \mid (i,j) \in \Z^2 \}$ (which is isomorphic to the `grid' $\Gamma$) is isometric, and so $\Delta X$ cannot be hyperbolic -- therefore, $\Delta X$ and $\mathcal{C}X$ are not quasi-isometric in this case.

It follows from \cite[Proposition 8.11]{genthesis} that the usual action of $G_0$ on $X$ is special -- however, there are infinitely many orbits of hyperplanes under this action. On the other hand, let $G = G_0 \rtimes \Z^2$, where the action of $\Z^2 = \langle x,y \mid xy=yx \rangle$ on $G_0$ is given by $a_{i,j}^{x^ny^m} = a_{i+n,j+m}$; this can be thought of as an example of a \emph{graph-wreath product}, see \cite{km} for details. Then it is easy to see that the action of $G$ on $G_0$ extends to an action of $G$ on $X$. This action is transitive on hyperplanes, and therefore not special.
\end{rmk}

\subsection{Hyperbolicity} \label{ss:qtree}

We show here that $\mathcal{C}X$ is a quasi-tree, proving Theorem \ref{t:qtree}.

\begin{prop} \label{p:midpointsep}
Let $A,B \in V(\mathcal{C}X)$ be two hyperplanes and suppose that $d_{\mathcal{C}X}(A,B) \geq 2$. Then there exists a midpoint $M$ of a geodesic between $A$ and $B$ in $\mathcal{C}X$ and a hyperplane $C$ separating $\mathcal{N}(A)$ and $\mathcal{N}(B)$ such that $d_{\mathcal{C}X}(M,C) \leq 3/2$.
\end{prop}

\begin{proof}
By Proposition \ref{p:cfgated}, we know that $\mathcal{N}(A)$ and $\mathcal{N}(B)$ are gated. It then follows from \cite[Lemma 2.36]{genthesis} that there exist vertices $p \in V(\mathcal{N}(A))$ and $q \in V(\mathcal{N}(B))$ such that any hyperplane separating $p$ from $q$ also separates $\mathcal{N}(A)$ from $\mathcal{N}(B)$. % let $o \in V(X)$ be a vertex in $\mathcal{N}(B)$, let $p$ be the gate for $o$ in $\mathcal{N}(A)$, and let $q$ be the gate for $p$ in $\mathcal{N}(B)$. 
Let $A=A_0,\ldots,A_m=B \in V(\mathcal{C}X)$ and $p=p_0,\ldots,p_{m+1}=q \in V(X)$ be as given by Proposition \ref{p:liftinggeodesics} in the case $\Gamma = \mathcal{C}X$, and let $M$ be the midpoint of the former geodesic. It is clear that $\mathcal{N}(A_i) \cap \mathcal{N}(A_j) = \varnothing$ whenever $|i-j| \geq 2$; in particular, $p_i \neq p_{i+1}$ whenever $1 \leq i \leq m-1$.

Now let $i = \lfloor m/2 \rfloor \in \{1,\ldots,m-1\}$, and let $C$ be any hyperplane separating $p_i$ and $p_{i+1}$. By the choice of the $p_j$, there exists a geodesic between $p$ and $q$ passing through $p_i$ and $p_{i+1}$: therefore, $C$ separates $p$ and $q$. Therefore, by the choice of $p$ and $q$, $C$ also separates $\mathcal{N}(A)$ from $\mathcal{N}(B)$.
%We claim that $C$ also separates $\mathcal{N}(A)$ and $\mathcal{N}(B)$: that is, $C$ does not separate any two vertices in $\mathcal{N}(A)$ nor any two vertices in $\mathcal{N}(B)$. Indeed, let $p' \in V(X)$ (respectively $q' \in V(X)$) be any vertex in $\mathcal{N}(A)$ (respectively $\mathcal{N}(B)$). By the choice of $q$, there is a geodesic between $p$ and $q'$ passing through $q$; as $C$ separates $p$ and $q$, it follows that it does not separate $q$ and $q'$. In particular, by considering a special case $q' = o$, we see that $C$ does not separate $q$ and $o$, and so it does separate $o$ and $p$. But now, by the choice of $p$, there is a geodesic between $o$ and $p'$ passing through $p$; as $C$ separates $o$ and $p$, it follows that it does not separate $p$ and $p'$. Thus, $C$ separates $p$ and $q$ but does not separate $p$ and $p'$, nor $q$ and $q'$, therefore it must separate $p'$ and $q'$. Hence $C$ separates $A$ and $B$, as claimed.
Finally, note that as $C$ separates $p_i,p_{i+1} \in \mathcal{N}(A_i)$, we have $d_{\mathcal{C}X}(A_i,C) \leq 1$. Therefore,
\[
d_{\mathcal{C}X}(M,C) \leq d_{\mathcal{C}X}(M,A_i) + d_{\mathcal{C}X}(A_i,C) = \left|\frac{m}{2}-i\right| + d_{\mathcal{C}X}(A_i,C) \leq \frac{1}{2}+1 = \frac{3}{2},
\]
as required.
\end{proof}

\begin{defn} \label{d:bottleneck}
For a geodesic metric space $Y$ and two points $x,y \in Y$ we say a point $m \in Y$ is a \emph{midpoint} between $x$ and $y$ if $d_Y(m,x) = d_Y(m,y) = \frac{1}{2} d_Y(x,y)$.

Let $D \in \N$. A geodesic metric space $Y$ is said to \emph{satisfy the $D$-bottleneck criterion} if for any points $x,y \in Y$, there exists a midpoint $m$ between $x$ and $y$ such that any path between $x$ and $y$ passes through a point $p$ such that $d_Y(p,m) \leq D$.
\end{defn}

\begin{thm}[Manning {\cite[Theorem 4.6]{manning}}] \label{t:manning}
A geodesic metric space $Y$ is a quasi-tree if and only if there exists a constant $D$ such that $Y$ satisfies the $D$-bottleneck criterion. If this is the case, then there exists a $(26D,52D)$-quasi-isometry from a simplicial tree to $Y$. \qed
\end{thm}

\begin{rmk}
In \cite{manning}, the last part of Theorem \ref{t:manning} is not explicitly stated, that is, the pair of constants $(26D,52D)$ is not explicitly computed. However, in \cite{manning} Manning constructs a map $\beta: T \to Y$, where $T$ is a simplicial tree, such that edges in $T$ are mapped to geodesic segments in $Y$, such that $d_Y(y,\beta(T)) < 20D$ for any $y \in Y$, and such that
\begin{equation} \label{e:manning}
8D d_T(v,w) - 16D \leq d_Y(\beta(v),\beta(w)) \leq 26D d_T(v,w)
\end{equation}
for any \emph{vertices} $v,w \in V(T)$. It is easy to check, using \eqref{e:manning} and the fact that any point in $T$ is within distance $\frac{1}{2}$ from a vertex, that we also have
\[
8D d_T(v,w) - 50D \leq d_Y(\beta(v),\beta(w)) \leq 26D d_T(v,w) + 52D
\]
for any \emph{points} $v,w \in T$, and so $\beta$ is a $(26D,52D)$-quasi-isometry, as required.

%the statement of this theorem is given for a general geodesic metric space (not necessarily a graph), and the definition of bottleneck criterion given there is stronger: instead of taking $v,w$ to be vertices of $\Gamma$ in Definition \ref{d:bottleneck}, Manning allows $v,w$ to be any points of $\Gamma$. However, as any point in a graph is within distance $\frac{1}{2}$ of a vertex, it is easy to see that in our setting the definition given here is equivalent to the one given in \cite{manning} (up to possibly modifying the constant $D$).
\end{rmk}

\begin{proof}[Proof of Theorem \ref{t:qtree}]
We claim that $\mathcal{C}X$ satisfies the $7/2$-bottleneck criterion. Thus, let $a,b \in \mathcal{C}X$ be any two points, and let $A,B \in V(\mathcal{C}X)$ be two hyperplanes lying on a geodesic between $a$ and $b$ such that $d_{\mathcal{C}X}(A,a)< 1$ and $d_{\mathcal{C}X}(B,b) < 1$. Note that if $M$ is a midpoint between $A$ and $B$, then there exists a midpoint $m$ between $a$ and $b$ such that $d_{\mathcal{C}X}(m,M) \leq \frac{1}{2}$.

If $d_{\mathcal{C}X}(a,b) < 7$, then any path between $a$ and $b$ passes through $a$, and $d_{\mathcal{C}X}(a,m) = d_{\mathcal{C}X}(a,b)/2 < 7/2$ for any midpoint $m$ between $a$ and $b$, so the $7/2$-bottleneck criterion is satisfied.

On the other hand, if $d_{\mathcal{C}X}(a,b) \geq 7$, then $d_{\mathcal{C}X}(A,B) > 5$ and so we may define points $M$ and $C$ as given by Proposition \ref{p:midpointsep}; let also $m$ be a midpoint $m$ between $a$ and $b$ such that $d_{\mathcal{C}X}(m,M) \leq \frac{1}{2}$. Let $B_1,\ldots,B_n$ be hyperplanes lying on a path in $\mathcal{C}X$ between $a$ and $b$ (in order), so that $A = B_r,B_{r+1},\ldots,B_{n+s} = B$ is a path in $\mathcal{C}X$ between $A$ and $B$, where $r,s \in \{0,1\}$. Choose vertices $q_1,\ldots,q_n \in V(X)$ such that $q_i \in \mathcal{N}(B_{i-1}) \cap \mathcal{N}(B_i)$ for all $i$. As $q_r \in \mathcal{N}(A)$, $q_{n+s} \in \mathcal{N}(B)$, and as $C$ separates $A$ and $B$, it follows that $C$ separates $q_r$ and $q_{n+s}$, and so it separates $q_i$ and $q_{i+1}$ for some $i$. But as $q_i,q_{i+1} \in \mathcal{N}(B_i)$, it follows that $d_{\mathcal{C}X}(C,B_i) \leq 1$. Then
\[
d_{\mathcal{C}X}(M,B_i) \leq d_{\mathcal{C}X}(M,C) + d_{\mathcal{C}X}(C,B_i) \leq \frac{3}{2}+1 = \frac{5}{2},
\]
implying, in particular, that $B_i \notin \{A,B\}$ (as $d_{\mathcal{C}X}(M,A) = d_{\mathcal{C}X}(M,B) = d_{\mathcal{C}X}(A,B)/2 > 5/2$), and so $B_i$ lies on the chosen path between $a$ and $b$. But then $d_{\mathcal{C}X}(m,B_i) \leq d_{\mathcal{C}X}(m,M) + d_{\mathcal{C}X}(M,B_i) \leq 3 < \frac{7}{2}$, and so $7/2$-bottleneck criterion is again satisfied, as required.

In particular, Theorem \ref{t:manning} implies that $\mathcal{C}X$ is a quasi-tree, and there exists a $(91,182)$-quasi-isometry from a simplicial tree $T$ to $\mathcal{C}X$. This implies that there exists a $(K,C)$-quasi-isometry $\gamma: \mathcal{C}X \to T$, where $K = 91$ and $C = (3 \cdot 91 + 4) \cdot 182$. The $\gamma$-images of edges in a geodesic triangle of $\mathcal{C}X$ are then $(K,C)$-quasi-geodesics in the tree $T$, and such quasi-geodesics must be within Hausdorff distance $R$ from geodesics with the same endpoints, where $R = R(K,C)$ is a universal constant: see, for instance, \cite[Chapter III.H, Theorem 1.7]{bridson}. This implies that geodesic triangles in $\mathcal{C}X$ must be $\delta$-slim, where $\delta = K(2R+C)$, proving the last part of the Theorem.
\end{proof}

\section{Acylindricity} \label{s:acyl}

In this section we prove Theorem \ref{t:main} \ref{i:tmain-acyl}. To do this, in Section \ref{ss:ss} we introduce the notion of contact sequences (see Definition \ref{d:cs}) and show the main technical result we need to prove Theorem \ref{t:main} \ref{i:tmain-acyl}: namely, Proposition \ref{p:technical}. In Section \ref{ss:constech} we use this to deduce Theorem \ref{t:main} \ref{i:tmain-acyl}.

Throughout this section, let $X$ be a quasi-median graph.

\subsection{Contact sequences} \label{ss:ss}

In this subsection, for a gated subgraph $Y \leq X$ and a collection $\mathcal{H}$ of hyperplanes in $X$, we denote by $Y'_{\mathcal{H}} \subseteq V(X)$ the set of vertices $v \in V(X)$ for which there exists a vertex $p_v \in V(Y)$ such that all hyperplanes separating $v$ from $p_v$ are in $\mathcal{H}$. Moreover, we write $Y_{\mathcal{H}}$ for the full subgraph of $X$ spanned by $Y'_{\mathcal{H}}$.

\begin{lem} \label{l:YHgated}
Let $Y \leq X$ be a gated subgraph and let $\mathcal{H}$ be a collection of hyperplanes in $X$. Then the subgraph $Y_{\mathcal{H}}$ of $X$ is gated.
\end{lem}

\begin{proof}
Suppose for contradiction that $Y_{\mathcal{H}}$ is not gated, and let $v \in V(X)$ be a vertex that does not have a gate in $Y_{\mathcal{H}}$. Let $p$ be the gate for $v$ in $Y$. Let $\hat{p}$ be a vertex of $Y_{\mathcal{H}}$ on a geodesic between $v$ and $p$ with $d_X(v,\hat{p})$ minimal. By our assumption, $\hat{p}$ is not a gate for $v$ in $Y_{\mathcal{H}}$, and so there exists a vertex $\hat{q} \in V(Y_{\mathcal{H}})$ such that no geodesic between $v$ and $\hat{q}$ passes through $\hat{p}$. Let $q$ be the gate of $\hat{q}$ in $Y$. Let $\gamma_p$, $\gamma_q$, $\delta$, $\hat\delta$, $\eta$ be geodesics between $\hat{p}$ and $p$, $\hat{q}$ and $q$, $p$ and $q$, $\hat{p}$ and $\hat{q}$, $v$ and $\hat{p}$ (respectively), as shown in Figure \ref{f:lYHgated}.

Since both $\eta$ and $\hat\delta$ are geodesics, and since $\eta\hat\delta$ is not (by the choice of $\hat{q}$), it follows from Lemma \ref{l:cancellingplane} that we may assume, without loss of generality, that there exists a hyperplane $C$ and edges $c_1$, $c_2$ of $\eta$, $\hat\delta$ (respectively), both of which are dual to $C$ and have $\hat{p}$ as an endpoint. But as $p$ is the gate for $v$ in $Y$, as $\eta\gamma_p$ is a geodesic by the choice of $\hat{p}$, and as $q \in Y$, it follows that $\eta\gamma_p\delta$ is a geodesic. Therefore, by Proposition \ref{p:geod} $H$ cannot cross $\gamma_p\delta$, and so $H$ does not separate $\hat{p}$ and $q$. As $H$ separates $\hat{p}$ and $\hat{q}$, it follows that $H$ separates $\hat{q}$ and $q$ and so crosses $\gamma_q$. In particular, since $\hat{q} \in V(Y_{\mathcal{H}})$ and since $q \in V(Y)$ is a gate for $\hat{q}$ in $Y$, it follows that all hyperplanes crossing $\gamma_q$ are in $\mathcal{H}$, and therefore $H \in \mathcal{H}$. But then the endpoint $p' \neq \hat{p}$ of $c_1$ is separated from $p \in V(Y)$ only by hyperplanes in $\mathcal{H}$; this contradicts the choice of $\hat{p}$. Thus $Y_{\mathcal{H}}$ is gated, as claimed.
\end{proof}

\begin{figure}[ht]
\centering
\begin{tikzpicture}
\filldraw [dashed,green!50!black,fill=green!8,rounded corners=10pt] (5,-1) rectangle (10.5,3);
\node [green!50!black] at (10.15,-0.65) {$Y_{\mathcal{H}}$};
\filldraw [dashed,green!50!black,fill=green!20,rounded corners=10pt] (8,-0.5) rectangle (10,2.5);
\node [green!50!black] at (9.65,-0.15) {$Y$};

\fill (2.5,0) circle (2pt) node [below] {$v$};
\fill (5,0) circle (2pt) node [below] {$\,\,\,\,\hat{p}$};
\fill (5.75,2) circle (2pt) node [above] {$\hat{q}$};
\fill (8,0) circle (2pt) node [below] {$p\,\,\,\,$};
\fill (8,2) circle (2pt) node [above] {$q\,\,\,\,$};

\draw [thick] (2.5,0) -- (3.5,0) node [below] {$\eta$} -- (5,0);
\draw [thick] (5,0) -- (6.5,0) node [below] {$\gamma_p$} -- (8,0);
\draw [thick] (5.75,2) -- (7,2) node [above] {$\gamma_q$} -- (8,2);
\draw [thick] (5,0) -- (5.45,1.2) node [left] {$\hat\delta$} -- (5.75,2);
\draw [thick] (8,0) arc (-90:0:1) node [right] {$\delta$} arc (0:90:1);

\fill (4.5,0) circle (1.2pt) node [below] {$p'\,\,$};
\draw [red,very thick] plot [smooth] coordinates { (4.7,-0.5) (4.75,0) (5,0.4) (6,1) (6.5,2) (6.55,2.5) } node [above] {$H$};
\end{tikzpicture}
\caption{Proof of Lemma \ref{l:YHgated}.} \label{f:lYHgated}
\end{figure}
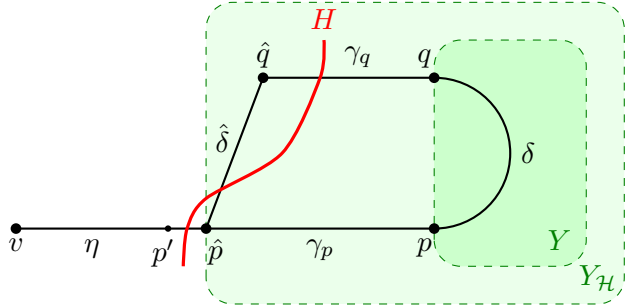

Now let a group $G$ act on a quasi-median graph $X$. This induces an action of $G$ on the crossing graph $\Delta X$. Let $\mathcal{H}$ be the set of orbits of vertices under $G \curvearrowright \Delta X$ -- alternatively, the set of orbits of hyperplanes under $G \curvearrowright X$. We may regard each element of $\mathcal{H}$ as a collection of hyperplanes -- thus, for instance, given $\mathcal{H}_0 \subseteq \mathcal{H}$ we may write $\bigcup \mathcal{H}_0$ for the set of all hyperplanes whose orbits are elements of $\mathcal{H}_0$.

Let $n \in \N$, and let $\mathcal{H}_1,\ldots,\mathcal{H}_n$ be subsets of $\mathcal{H}$. Pick a vertex (a `basepoint') $o \in V(X)$, and define the subgraphs $Y_0,\ldots,Y_n \subseteq X$ inductively: set $Y_0 = \{o\}$ and $Y_i = (Y_{i-1})_{\bigcup \mathcal{H}_i}$ for $1 \leq i \leq n$. By Lemma \ref{l:YHgated}, $Y_n$ is a gated subgraph. We denote $Y_n$ as above by $Y(o,\mathcal{H}_1,\ldots,\mathcal{H}_n)$, and we denote the gate for $v \in V(X)$ in $Y_n$ by $\mathfrak{g}(v;o,\mathcal{H}_1,\ldots,\mathcal{H}_n)$.

\begin{defn} \label{d:cs}
Let $H,H' \in V(\mathcal{C}X)$, and let $p,p' \in V(X)$ be such that $p \in \mathcal{N}(H)$ and $p' \in \mathcal{N}(H')$. Let $n = d_{\mathcal{C}X}(H,H')$. Given any geodesic $H=H_0,\ldots,H_n=H'$ in $\mathcal{C}X$ and vertices $p=p_0,p_1,\ldots,p_{n+1}=p' \in V(X)$ such that $p_i,p_{i+1} \in \mathcal{N}(H_i)$ for $0 \leq i \leq n$, we call $\mathfrak{S} = (H_0,\ldots,H_n,p_0,\ldots,p_{n+1})$ a \emph{contact sequence} for $(H,H',p,p')$. 

Given a contact sequence $\mathfrak{S} = (H_0,\ldots,H_n,p_0,\ldots,p_{n+1})$ for $(H,H',p,p')$ and a vertex $v \in V(X)$, we say $(g_0,\ldots,g_n) \in V(X)^{n+1}$ is the \emph{$v$-gate} for $\mathfrak{S}$ if $g_i$ is the gate for $v$ in $\mathcal{N}(H_i)$ for $0 \leq i \leq n$. We furthermore denote the tuples $(d_X(p_n,g_n),\ldots,d_X(p_0,g_0))$ and $(d_X(p_1,g_0),\ldots,d_X(p_{n+1},g_n))$ by $\mathfrak{C}_\diagup(\mathfrak{S},v)$ and $\mathfrak{C}_\diagdown(\mathfrak{S},v)$, respectively. We say a contact sequence $\mathfrak{S}$ for $(H,H',p,p')$ is \emph{$v$-minimal} if for any other contact sequence $\mathfrak{S}'$ for $(H,H',p,p')$ we have either $\mathfrak{C}_\diagup(\mathfrak{S},v) \leq \mathfrak{C}_\diagup(\mathfrak{S}',v)$ or $\mathfrak{C}_\diagdown(\mathfrak{S},v) \leq \mathfrak{C}_\diagdown(\mathfrak{S}',v)$ in the lexicographical order.

Finally, suppose a group $G$ acts on $X$. Given a vertex $v \in V(X)$ and a contact sequence $\mathfrak{S} = (H_0,\ldots,H_n,p_0,\ldots,p_{n+1})$ for $(H,H',p,p')$ with a $v$-gate $(g_0,\ldots,g_n)$, we say $(\mathcal{H}_0,\ldots,\mathcal{H}_n,\mathcal{H}_0',\ldots,\mathcal{H}_n')$, where $\mathcal{H}_i,\mathcal{H}_i' \subseteq V(\mathcal{C}X/G)$, is the \emph{$(v,G)$-orbit sequence} for $\mathfrak{S}$ if \[ \mathcal{H}_i = \{ H^G \mid H \in V(\mathcal{C}X) \text{ separates } p_i \text{ from } g_i \} \] and \[ \mathcal{H}_i' = \{ H^G \mid H \in V(\mathcal{C}X) \text{ separates } p_{i+1} \text{ from } g_i \} \] for $0 \leq i \leq n$.
\end{defn}

It is clear that given any $H$, $H'$, $p$ and $p'$ as in Definition \ref{d:cs}, there exists a contact sequence for $(H,H',p,p')$. As the lexicographical order is a well-ordering of $\N^n$, it follows that a $v$-minimal contact sequence exists as well.

%We now prove that, under the hypotheses of Theorem \ref{t:main} \ref{i:tmain-acyl} and an additional assumption that $\Delta X$ is connected, the action $G \curvearrowright \Delta X$ is acylindrical. Note that, if $\Delta X$ is connected, this implies acylindricity of $G \curvearrowright \mathcal{C}X$ by Theorem \ref{t:main} \ref{i:tmain-qi}. The main ingredient of our proof is the following technical proposition.

\begin{prop} \label{p:technical}
Suppose a group $G$ acts specially on a quasi-median graph $X$. Let $H,H' \in V(\mathcal{C}X)$, let $p,p' \in V(X)$ be such that $p \in \mathcal{N}(H)$ and $p' \in \mathcal{N}(H')$, and let $v \in V(X)$. Let $\mathfrak{S} = (H_0,\ldots,H_n,p_0,\ldots,p_{n+1})$ be a $v$-minimal contact sequence for $(H,H',p,p')$ with $v$-gate $(g_0,\ldots,g_n)$ and $(v,G)$-orbit sequence $(\mathcal{H}_0,\ldots,\mathcal{H}_n,\mathcal{H}_0',\ldots,\mathcal{H}_n')$. Write $\mathfrak{g}_i := \mathfrak{g}(v;p,\mathcal{H}_0,\ldots,\mathcal{H}_i)$ and $\mathfrak{g}_i' := \mathfrak{g}(v;p',\mathcal{H}_n',\ldots,\mathcal{H}_i')$ for $0 \leq i \leq n$.
Then,
\begin{enumerate}[label={\normalfont({\roman*})}]
\item \label{i:ptechnical-main} $\mathfrak{g}_n = \mathfrak{g}_0'$;
\item \label{i:ptechnical-noosc} no hyperplane from $\bigcup \mathcal{H}_i$ osculates with a hyperplane from $\bigcup \mathcal{H}_j'$ whenever $i > j$; and
\item \label{i:ptechnical-gateclose} for $1 \leq i \leq n$, the hyperplanes separating $\mathfrak{g}_{i-1}$ from $\mathfrak{g}_i$ (respectively $\mathfrak{g}_i'$ from $\mathfrak{g}_{i-1}'$) are precisely the hyperplanes separating $p_i$ from $g_i$ (respectively $p_i$ from $g_{i-1}$).
\end{enumerate}
\end{prop}

\begin{proof}
Induction on $n$.

For $n = 0$, we claim that $\mathfrak{g}_0 = g_0$. Indeed, by definition of $\mathcal{H}_0$ we have $g_0 \in Y(p_0,\mathcal{H}_0)$, and so there exists a geodesic $\eta$ between $p$ and $v$ passing through $g_0$ and $\mathfrak{g}_0$. Suppose for contradiction $\mathfrak{g}_0 \neq g_0$, let $a \subseteq \eta$ be the edge with endpoint $g_0$ such that the other endpoint $q_a \neq g_0$ of $a$ satisfies $d_X(v,g_0) > d_X(v,q_a)$, and let $A$ be the hyperplane dual to $a$; see Figure \ref{f:ptech-n0}. Then $g_0 \in \mathcal{N}(H_0) \cap \mathcal{N}(A)$, and so $H_0$ and $A$ either coincide, or intersect, or osculate. As $A$ separates $p$ and $\mathfrak{g}_0$, we know that $A^g$ separates $p$ and $g_0$ and so $A^g$ and $H_0$ either coincide or intersect for some $g \in G$. Thus, as the action $G \curvearrowright X$ is special, it follows that $A$ and $H_0$ cannot osculate, and therefore they either coincide or intersect. But then we also have $q_a \in \mathcal{N}(H_0)$, contradicting the choice of $g_0$. Therefore, $\mathfrak{g}_0 = g_0$, as claimed. A symmetric argument shows that $\mathfrak{g}_0' = g_0$, and so the conclusion of the proposition is clear.

Suppose now that $n \geq 1$, and let $\hat{\mathfrak{g}}_i'=\mathfrak{g}(v;p_n,\mathcal{H}_{n-1}',\ldots\mathcal{H}_i')$ for $0 \leq i \leq n$ (so that $\hat{\mathfrak{g}}_n' = p_n$). Notice that $(H_0,\ldots,H_{n-1},p_0,\ldots,p_n)$ is a $v$-minimal contact sequence for $(H,H_{n-1},p,p_n)$. Thus, by the inductive hypothesis
we have
\begin{enumerate}[label=({\roman*}')]
\item \label{i:ptech-pf-main1} $\mathfrak{g}_{n-1} = \hat{\mathfrak{g}}_0'$;
\item \label{i:ptech-pf-noosc1} no hyperplane from $\bigcup \mathcal{H}_i$ osculates with a hyperplane from $\bigcup \mathcal{H}_j'$ whenever $n-1 \geq i > j$;
\item \label{i:ptech-pf-gateclose1} for $1 \leq i \leq n-1$, the hyperplanes separating $\mathfrak{g}_{i-1}$ from $\mathfrak{g}_i$ (respectively $\hat{\mathfrak{g}}_i'$ from $\hat{\mathfrak{g}}_{i-1}'$) are precisely the hyperplanes separating $p_i$ from $g_i$ (respectively $p_i$ from $g_{i-1}$).
\end{enumerate}
Moreover, let $\hat{\mathfrak{g}}_i=\mathfrak{g}(v;p_1,\mathcal{H}_1,\ldots\mathcal{H}_i)$ for $0 \leq i \leq n$ (so that $\hat{\mathfrak{g}}_0 = p_1$), and notice that $(H_1,\ldots,H_n,p_1,\ldots,p_{n+1})$ is a $v$-minimal contact sequence for $(H_1,H',p_1,p')$. Thus, by the inductive hypothesis
we have
\begin{enumerate}[label=({\roman*}'')]
\item \label{i:ptech-pf-main2} $\hat{\mathfrak{g}}_n = \mathfrak{g}_1'$;
\item \label{i:ptech-pf-noosc2} no hyperplane from $\bigcup \mathcal{H}_i$ osculates with a hyperplane from $\bigcup \mathcal{H}_j'$ whenever $i > j \geq 1$;
\item \label{i:ptech-pf-gateclose2} for $2 \leq i \leq n$, the hyperplanes separating $\hat{\mathfrak{g}}_{i-1}$ from $\hat{\mathfrak{g}}_i$ (respectively $\mathfrak{g}_i'$ from $\mathfrak{g}_{i-1}'$) are precisely the hyperplanes separating $p_i$ from $g_i$ (respectively $p_i$ from $g_{i-1}$).
\end{enumerate}
Finally, the proof of the $n = 0$ case above shows that $\mathfrak{g}(v;p_i,\mathcal{H}_i) = g_i = \mathfrak{g}(v;p_{i+1},\mathcal{H}_i')$ for $0 \leq i \leq n$.

Now let $q = \hat{\mathfrak{g}}_{n-1}$ and note that we also have $q = \hat{\mathfrak{g}}_1'$: this is clear if $n = 1$ and follows from the inductive hypothesis if $n \geq 2$. Let $\mathcal{A}, \mathcal{B}, \mathcal{A}', \mathcal{B}' \subseteq V(\mathcal{C}X)$ be the sets of hyperplanes separating $q$ from $\mathfrak{g}_{n-1}$, $q$ from $\mathfrak{g}_1'$, $\mathfrak{g}_1'$ from $\mathfrak{g}_n$, $\mathfrak{g}_{n-1}$ from $\mathfrak{g}_n$, respectively; see Figure \ref{f:ptech-general}. We claim that $\mathcal{A} = \mathcal{A}'$ and $\mathcal{B} = \mathcal{B}'$. We will show this in steps, proving part \ref{i:ptechnical-noosc} of the Proposition along the way.

\begin{figure}[ht]
\begin{subfigure}[b]{0.3\textwidth}
\centering
\begin{tikzpicture}[scale=0.66]
\filldraw [dashed,green!50!black,fill=green!20,rounded corners=10pt] (-0.5,-1.4) rectangle (1.5,2);
\node [green!50!black] at (0.6,-1) {$\mathcal{N}(H)$};

\draw [thick] plot [smooth] coordinates { (0,0) (1.5,1.5) (1.9,1.7) (3,1.5) (4,2) (4.8,3) };
\fill (0,0) circle (1.8pt) node [below] {$p$};
\fill (1.5,1.5) circle (1.8pt) node [left,yshift=3pt] {$g_0$};
\fill (1.9,1.7) circle (1.8pt) node [above] {$\ q_a$};
\fill (3,1.5) circle (1.8pt) node [below] {$\mathfrak{g}_0$};
\fill (4.8,3) circle (1.8pt) node [above] {$v$};
\node [below] at (4,2) {$\eta$};

\draw [red,very thick] plot [smooth] coordinates { (1.5,2.5) (1.75,1.5) (1.75,0.5) } node [below] {$A$};
\draw [red,very thick] plot [smooth] coordinates { (0,1.5) (0.8,1) (1,0.5) } node [below] {$A^g$};
\end{tikzpicture}
\caption{Case $n = 0$.} \label{f:ptech-n0}
\end{subfigure}
\hfill
\begin{subfigure}[b]{0.69\textwidth}
\centering
\begin{tikzpicture}[scale=0.8]
\filldraw [dashed,green!50!black,fill=green!20,rounded corners=10pt] (-0.9,1) rectangle (2,-1.2);
\filldraw [dashed,green!50!black,fill=green!20,rounded corners=10pt] (9.3,1) rectangle (6,-1.2);
\node [green!50!black] at (0.8,-0.8) {$\mathcal{N}(H)$};
\node [green!50!black] at (7.2,-0.8) {$\mathcal{N}(H')$};

\fill (0,0) circle (1.5pt) node [below,xshift=-3pt] {$p_0=p$};
\fill (1,1) circle (1.5pt) node [above left,xshift=6pt] {$\mathfrak{g}_0=g_0$};
\fill (2,0) circle (1.5pt) node [right,yshift=-3pt] {$p_1=\hat{\mathfrak{g}}_0$};
\fill (8,0) circle (1.5pt) node [below,yshift=3pt,xshift=9pt] {$p'=p_{n+1}\ $};
\fill (7,1) circle (1.5pt) node [above right,xshift=-6pt] {$g_n=\mathfrak{g}_n'$};
\fill (6,0) circle (1.5pt) node [left,yshift=-3pt] {$\hat{\mathfrak{g}}_n'=p_n\!$};
\fill (4,2) circle (1.5pt) node [below] {$q$};
\fill (3,3) circle (1.5pt) node [left,yshift=5pt] {$\hat{\mathfrak{g}}_0'=\mathfrak{g}_{n-1}$};
\fill (5,3) circle (1.5pt) node [right,yshift=5pt] {$\mathfrak{g}_1'=\hat{\mathfrak{g}}_n$};
\fill (4,4) circle (1.5pt) node [above] {$\mathfrak{g}_n\ \ $};
\fill (6,5) circle (1.5pt) node [right] {$v$};
\node at (4,-0.5) {$\cdots$};

\draw [thick,rounded corners=10pt] (0,0) -- (5,5) -- (6,5);
\draw [thick] (2,0) -- (1,1);
\draw [thick] (2,0) -- (5,3);
\draw [thick] (8,0) -- (4,4);
\draw [thick] (6,0) -- (7,1);
\draw [thick] (6,0) -- (3,3);

\draw [red,very thick] (1,0.5) -- (3.4,2.9);
\draw [red,very thick] (1.25,0.25) -- (3.6,2.6);
\draw [red,very thick] (1.5,0) -- (3.9,2.4);
\node [red] at (3.6,2.8) {$\mathcal{A}$};
\draw [blue,very thick] (7,0.5) -- (4.6,2.9);
\draw [blue,very thick] (6.75,0.25) -- (4.4,2.6);
\draw [blue,very thick] (6.5,0) -- (4.1,2.4);
\node [blue] at (4.4,2.8) {$\mathcal{B}$};
\draw [red!80!black,very thick] (4.1,3.6) -- (4.4,3.9);
\draw [red!80!black,very thick] (4.35,3.35) -- (4.6,3.6);
\draw [red!80!black,very thick] (4.6,3.1) -- (4.9,3.4);
\node [red!80!black] at (4.8,3.8) {$\mathcal{A}'$};
\draw [blue!70!black,very thick] (3.9,3.6) -- (3.6,3.9);
\draw [blue!70!black,very thick] (3.65,3.35) -- (3.4,3.6);
\draw [blue!70!black,very thick] (3.4,3.1) -- (3.1,3.4);
\node [blue!70!black] at (3.2,3.8) {$\mathcal{B}'$};
\end{tikzpicture}
\caption{Case $n \geq 1$.} \label{f:ptech-general}
\end{subfigure}
\caption{Proof of Proposition \ref{p:technical}: general setup.} \label{f:ptech-genset}
\end{figure}
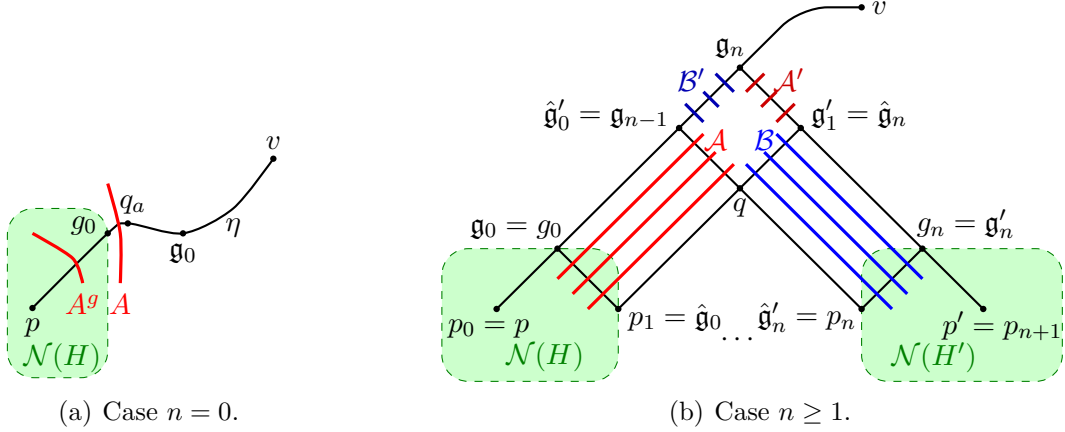

\begin{description}
\item[$\mathcal{A} \cap \mathcal{B} = \varnothing$:] Suppose for contradiction that there exists some hyperplane $A \in \mathcal{A} \cap \mathcal{B}$. As $A \in \mathcal{A}$, we know that $A$ separates $\hat{\mathfrak{g}}_1'$ from $\hat{\mathfrak{g}}_0'$, and so by \ref{i:ptech-pf-gateclose1} above it also separates $p_1$ from $g_0$: thus $d_{\mathcal{C}X}(H_0,A) \leq 1$. Similarly, as $A \in \mathcal{B}$, by \ref{i:ptech-pf-gateclose2} above we know that $A$ separates $p_n$ from $g_n$ and therefore $d_{\mathcal{C}X}(H_n,A) \leq 1$. Hence, $n = d_{\mathcal{C}X}(H_0,H_n) \leq 2$, and so either $n = 1$ or $n = 2$.

Let $\alpha$, $\beta$ be geodesics between $p_1$ and $g_0$, $p_n$ and $g_n$, respectively, and let $a \subseteq \alpha$ and $b \subseteq \beta$ be the edges dual to $A$. As $a$ and $b$ lie on geodesics with endpoint $v$, we may pick endpoints $q_a$ and $q_b$ of $a$ and $b$, respectively, such that $A$ does not separate $q_a$, $q_b$ and $v$.

Suppose first that $n = 2$: see Figure \ref{f:ptech-ABn2}. Note that in this case $H_0,A,H_2$ is a geodesic in $\mathcal{C}X$ and that $d_X(p_2,g_2) > d_X(q_b,g_2)$ and $d_X(p_1,g_0) > d_X(q_a,g_0)$. Moreover, since $q_a$ lies on a geodesic between $p_1$ and $g_0$, we have $q_a \in \mathcal{N}(H_0)$; similarly, $q_b \in \mathcal{N}(H_2)$. Furthermore, by the construction we know that $q_a,q_b \in \mathcal{N}(A)$. We may therefore replace $p_1$, $p_2$ and $H_1$ by $q_a$, $q_b$ and $A$, respectively, contradicting $v$-minimality of $\mathfrak{S}$. Thus $n \neq 2$.

Suppose now that $n = 1$. Then $A$ separates $p_1$ from both $g_0$ and $g_1$. By Lemma \ref{l:cancellingplane}, we may then without loss of generality assume that $p_1$ is an endpoint (distinct from $q_a$ and $q_b$) of both $a$ and $b$. Now note that both $a$ and $b$ are edges on a geodesic between $p_1$ and $v$, so we must have $a = b$, and in particular $q_a = q_b$; see Figure \ref{f:ptech-ABn1}. Since $A$ separates $p_1$ from both $g_0$ and $g_1$, it intersects or coincides with both $H_0$ and $H_1$, and so $q_a \in \mathcal{N}(H_0) \cap \mathcal{N}(H_1)$. We may therefore replace $p_n$ by $q_a$; but we have $d_X(p_1,g_1) > d_X(q_a,g_1)$ and $d_X(p_1,g_0) > d_X(q_a,g_0)$, contradicting $v$-minimality of $\mathfrak{S}$. Thus no such hyperplane $A \in \mathcal{A} \cap \mathcal{B}$ can exist and so $\mathcal{A} \cap \mathcal{B} = \varnothing$, as claimed.

\item[$\mathcal{A} \cap \mathcal{B}' = \varnothing$:] This is clear, as $\mathfrak{g}_{n-1}=\hat{\mathfrak{g}}_0'$ lies on a geodesic between $q = \hat{\mathfrak{g}}_1'$ and $\mathfrak{g}_n$, and so no hyperplane can separate $\mathfrak{g}_{n-1}$ from both $q$ and $\mathfrak{g}_n$.

\item[$\mathcal{A}' \cap \mathcal{B} = \varnothing$:] Let $\mathcal{C}$ be the set of hyperplanes separating $\mathfrak{g}_n$ and $v$. We first claim that $\mathcal{A}' \cap \mathcal{B} = \mathcal{A}' \cap \mathcal{C}$. Indeed, let $B \in \mathcal{A}' \cap \mathcal{B}$. Since $B \in \mathcal{B}$, it separates $q$ and $\mathfrak{g}_1'$; as $\mathfrak{g}_1'=\hat{\mathfrak{g}}_n$ lies on a geodesic between $q$ and $v$, $B$ cannot separate $\mathfrak{g}_1'$ and $v$. But as $B \in \mathcal{A}'$, it separates $\mathfrak{g}_1'$ and $\mathfrak{g}_n$, and so $B$ must separate $\mathfrak{g}_n$ and $v$. Therefore, $B \in \mathcal{A}' \cap \mathcal{C}$. Conversely, let $C \in \mathcal{A}' \cap \mathcal{C}$. Since $C \in \mathcal{C}$, it separates $\mathfrak{g}_n$ and $v$; as $\mathfrak{g}_n$ lies on a geodesic between $q$ and $v$, $C$ cannot separate $q$ and $\mathfrak{g}_n$. But as $C \in \mathcal{A}'$, it separates $\mathfrak{g}_1'$ and $\mathfrak{g}_n$, and so $C$ must separate $q$ and $\mathfrak{g}_1'$. Therefore, $C \in \mathcal{A}' \cap \mathcal{B}$, and so $\mathcal{A}' \cap \mathcal{B} = \mathcal{A}' \cap \mathcal{C}$, as claimed.

Now suppose for contradiction that there exists a hyperplane $A \in \mathcal{A}' \cap \mathcal{B} = \mathcal{A}' \cap \mathcal{B} \cap \mathcal{C}$. Let $\gamma$ be a geodesic between $\mathfrak{g}_n$ and $v$, and let $c \subseteq \gamma$ be the edge dual to $A$. By Lemma \ref{l:cancellingplane}, we may without loss of generality assume that $\mathfrak{g}_n$ is an endpoint of $c$: see Figure \ref{f:ptech-ApB}.

Now let $q_c \neq \mathfrak{g}_n$ be the other endpoint of $c$. Note that since $A \in \mathcal{B}$, we have $A^G \in \mathcal{H}_n$. Therefore, it follows that $q_c$ is separated from $\mathfrak{g}_{n-1}$ only by hyperplanes in $\bigcup \mathcal{H}_n$; as $d_X(v,\mathfrak{g}_n) > d_X(v,q_c)$, this contradicts the definition of $\mathfrak{g}_n$. Thus $\mathcal{A} \cap \mathcal{B}' = \varnothing$, as claimed.

\item[$\mathcal{A} \subseteq \mathcal{A}'$ and $\mathcal{B} \subseteq \mathcal{B}'$:]  As $\mathcal{A} \cap \mathcal{B} = \varnothing = \mathcal{A} \cap \mathcal{B}'$, every hyperplane separating $q$ and $\mathfrak{g}_{n-1}$ does not separate $q$ and $\mathfrak{g}_1'$, nor $\mathfrak{g}_{n-1}$ and $\mathfrak{g}_n$, thus it separates $\mathfrak{g}_1'$ and $\mathfrak{g}_n$. It follows that $\mathcal{A} \subseteq \mathcal{A}'$. Similarly, as $\mathcal{A} \cap \mathcal{B} = \varnothing = \mathcal{A}' \cap \mathcal{B}$, we get $\mathcal{B} \subseteq \mathcal{B}'$.

\item[Part \ref{i:ptechnical-noosc}:] By \ref{i:ptech-pf-noosc1} and \ref{i:ptech-pf-noosc2} above, it is enough to show that no hyperplane from $\bigcup \mathcal{H}_n$ osculates with a hyperplane from $\bigcup \mathcal{H}_0'$. Thus, let $A$ (respectively $B$) be a hyperplane separating $p_1$ and $g_0$ (respectively $p_n$ and $g_n$), so that $A^G \in \mathcal{H}_0'$ and $B^G \in \mathcal{H}_n$. It is now enough to show that $A^g$ and $B^h$ do not osculate for any $g,h \in G$.

But as $A$ separates $p_1$ from $g_0$, we know from \ref{i:ptech-pf-gateclose1} that it also separates $\hat{\mathfrak{g}}_1'=q$ from $\hat{\mathfrak{g}}_0'=\mathfrak{g}_{n-1}$, that is, $A \in \mathcal{A}$. Similarly, as $B$ separates $p_n$ and $g_n$, we know from \ref{i:ptech-pf-gateclose2} that $B \in \mathcal{B}$. But as $\mathcal{A} \cap \mathcal{B} = \varnothing = \mathcal{A} \cap \mathcal{B}'$ and as $\mathcal{B} \subseteq \mathcal{B}'$, it follows that $A$ separates $q$ and $\mathfrak{g}_1'$ from $\mathfrak{g}_{n-1}$ and $\mathfrak{g}_n$, while $B$ separates $q$ from $\mathfrak{g}_1'$ and $\mathfrak{g}_{n-1}$ from $\mathfrak{g}_n$. Therefore, $A$ and $B$ must intersect. But as the action $G \curvearrowright X$ is special, it follows that $A^g$ and $B^h$ do not osculate for any $g,h \in G$. Thus no hyperplane from $\bigcup \mathcal{H}_n$ osculates with a hyperplane from $\bigcup \mathcal{H}_0'$, and so part \ref{i:ptechnical-noosc} holds, as required.

\item[$\mathcal{A}' \cap \mathcal{B}' = \varnothing$:] Suppose for contradiction that $A \in \mathcal{A}' \cap \mathcal{B}'$ is a hyperplane. Let $\alpha'$ be a geodesic between $\mathfrak{g}_1'$ and $\mathfrak{g}_n$, let $a \subseteq \alpha'$ be the edge dual to $A$, and let $q_a,q_a'$ be the endpoints of $a$ so that $A$ does not separate $\mathfrak{g}_1'$ and $q_a$. Suppose, without loss of generality, that $\alpha'$ and $A$ are chosen in such a way that $d_X(\mathfrak{g}_1',q_a)$ is as small as possible.

We now claim that $\mathfrak{g}_1' = q_a$. Indeed, suppose not, and let $a' \neq a$ be the other edge on $\alpha'$ with endpoint $q_a$. Let $A' \in \mathcal{A}'$ be the hyperplane dual to $a'$; see Figure \ref{f:ptech-ApBp}. Then $A'$ does not separate $q$ and $\mathfrak{g}_1'$ (as $\mathcal{A}' \cap \mathcal{B} = \varnothing$), nor $\mathfrak{g}_{n-1}$ and $\mathfrak{g}_n$ (by minimality of $d_X(\mathfrak{g}_1',q_a)$), but it separates $\mathfrak{g}_1'$ (and so $q$) from $\mathfrak{g}_n$ (and so $\mathfrak{g}_{n-1}$). In particular, $A' \in \mathcal{A}$, and so $A' \in \bigcup \mathcal{H}_0'$. On the other hand, $A \in \mathcal{B}' \subseteq \bigcup \mathcal{H}_n$, and so $A$ and $A'$ cannot oscullate by part \ref{i:ptechnical-noosc}. It follows that $A$ and $A'$ must intersect, and therefore we may swap $a$ and $a'$ on $\alpha'$, contradicting minimality of $d_X(\mathfrak{g}_1',q_a)$. Thus $\mathfrak{g}_1' = q_a$, as claimed.

But now $q_a'$ is separated from $q$ just by hyperplanes in $\bigcup \mathcal{H}_n$. Furthermore, $A$ cannot separate $\mathfrak{g}_n$ and $v$ (as $\mathfrak{g}_n$ lies on a geodesic between $\mathfrak{g}_{n-1}$ and $v$, and as $A$ separates $\mathfrak{g}_{n-1}$ and $\mathfrak{g}_n$), nor $\mathfrak{g}_n$ and $q_a'$ (as $\alpha'$ is a geodesic), but $A$ separates $q_a'$ (and so $\mathfrak{g}_n$ and $v$) from $\mathfrak{g}_1'$. In particular, $d_X(v,\mathfrak{g}_1') > d_X(v,q_a')$, contradicting the fact that $\mathfrak{g}_1' = \hat{\mathfrak{g}}_n$. Thus $\mathcal{A}' \cap \mathcal{B}' = \varnothing$, as required.

\item[$\mathcal{A} = \mathcal{A}'$ and $\mathcal{B} = \mathcal{B}'$:] We have already shown $\mathcal{A} \subseteq \mathcal{A}'$ and $\mathcal{B} \subseteq \mathcal{B}'$. Conversely, as $\mathcal{A}' \cap \mathcal{B} = \varnothing = \mathcal{A}' \cap \mathcal{B}'$, every hyperplane separating $\mathfrak{g}_1'$ and $\mathfrak{g}_n$ does not separate $q$ and $\mathfrak{g}_1'$, nor $\mathfrak{g}_{n-1}$ and $\mathfrak{g}_n$, thus it separates $q$ and $\mathfrak{g}_{n-1}$. It follows that $\mathcal{A}' \subseteq \mathcal{A}$ and so $\mathcal{A} = \mathcal{A}'$. Similarly, as $\mathcal{A} \cap \mathcal{B}' = \varnothing = \mathcal{A}' \cap \mathcal{B}'$, we get $\mathcal{B} = \mathcal{B}'$.
\end{description}

Now part \ref{i:ptechnical-gateclose} of the Proposition follows immediately. Indeed, given \ref{i:ptech-pf-gateclose1} and \ref{i:ptech-pf-gateclose2}, it is enough to show that the hyperplanes separating $\mathfrak{g}_{n-1}$ from $\mathfrak{g}_n$ (respectively $\mathfrak{g}_1'$ from $\mathfrak{g}_0'$) are precisely the hyperplanes separating $\hat{\mathfrak{g}}_{n-1}$ from $\hat{\mathfrak{g}}_n$ (respectively $\hat{\mathfrak{g}}_1'$ from $\hat{\mathfrak{g}}_0'$). But this, and so \ref{i:ptechnical-gateclose}, follows from the fact that $\mathcal{A} = \mathcal{A}'$ and $\mathcal{B} = \mathcal{B}'$.

Finally, we are left to show part \ref{i:ptechnical-main}. We know that $\mathcal{A}' = \mathcal{A} \subseteq \bigcup \mathcal{H}_0'$, and so $\mathfrak{g}_n \in Y(\mathfrak{g}_1',\mathcal{H}_0') \subseteq Y(p',\mathcal{H}_n',\ldots,\mathcal{H}_0')$. In particular, there exists a geodesic between $v$ and $\mathfrak{g}_n$ passing through $\mathfrak{g}(v;p',\mathcal{H}_n',\ldots,\mathcal{H}_0') = \mathfrak{g}_0'$. But a symmetric argument can show that there exists a geodesic between $v$ and $\mathfrak{g}_0'$ passing through $\mathfrak{g}_n$. Thus $\mathfrak{g}_n = \mathfrak{g}_0'$, proving \ref{i:ptechnical-main}.
\end{proof}

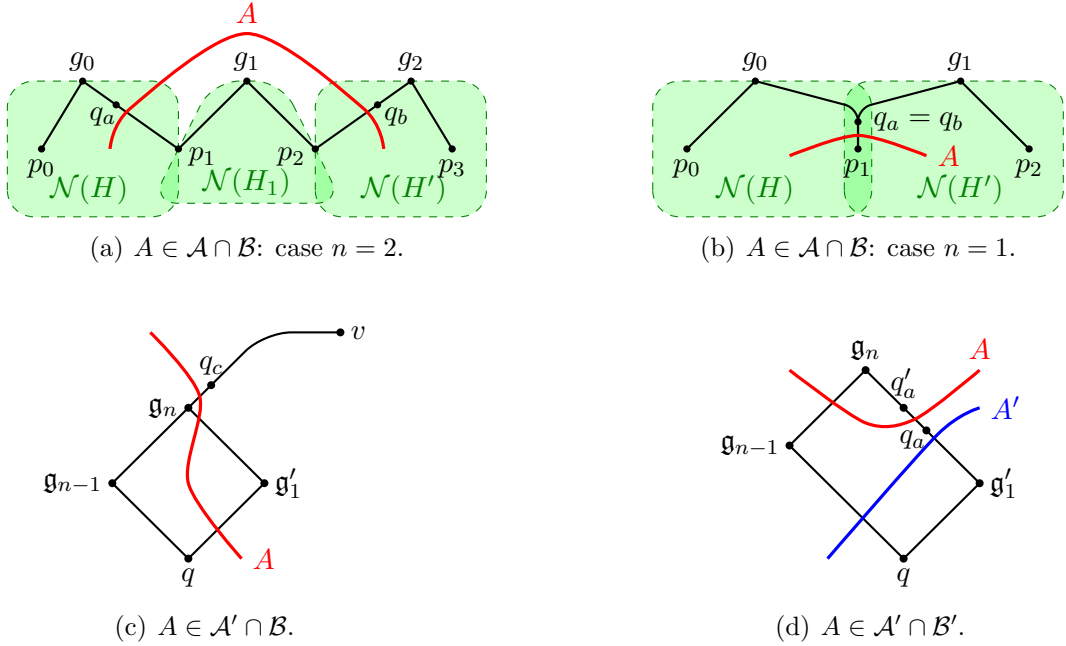
\begin{figure}[ht]
\begin{subfigure}[b]{0.52\textwidth}
\centering
\begin{tikzpicture}[scale=0.9]
\fill [green,rounded corners=10pt,opacity=0.2] (-0.5,-1) rectangle (2,1);
\fill [green,rounded corners=10pt,opacity=0.2] (4,-1) rectangle (6.5,1);
\fill [green,rounded corners=10pt,opacity=0.2] (3.5,1) -- (4.4,-0.8) -- (1.6,-0.8) -- (2.5,1) -- cycle;
\draw [dashed,green!50!black,rounded corners=10pt] (-0.5,-1) rectangle (2,1);
\draw [dashed,green!50!black,rounded corners=10pt] (4,-1) rectangle (6.5,1);
\draw [dashed,green!50!black,rounded corners=10pt] (3.5,1) -- (4.4,-0.8) -- (1.6,-0.8) -- (2.5,1) -- cycle;
\node [green!50!black] at (0.7,-0.6) {$\mathcal{N}(H)$};
\node [green!50!black] at (5.3,-0.6) {$\mathcal{N}(H')$};
\node [green!50!black] at (3,-0.5) {$\mathcal{N}(H_1)$};

\draw [thick] (0,0) -- (0.6,1) -- (2,0) -- (3,1) -- (4,0) -- (5.4,1) -- (6,0);
\fill (0,0) circle (1.5pt) node [below] {$p_0$};
\fill (0.6,1) circle (1.5pt) node [above] {$g_0$};
\fill (2,0) circle (1.5pt) node [right,yshift=-3pt] {$p_1$};
\fill (3,1) circle (1.5pt) node [above] {$g_1$};
\fill (4,0) circle (1.5pt) node [left,yshift=-3pt] {$p_2$};
\fill (5.4,1) circle (1.5pt) node [above] {$g_2$};
\fill (6,0) circle (1.5pt) node [below] {$p_3$};
\fill (1.09,0.65) circle (1.5pt) node [below,xshift=-5pt,yshift=3pt] {$q_a$};
\fill (4.91,0.65) circle (1.5pt) node [below,xshift=7pt,yshift=3pt] {$q_b$};

\draw [very thick,red] plot [smooth] coordinates { (1,0) (1.3,0.6) (3,1.7) (4.7,0.6) (5,0) };
\node [red,above] at (3,1.7) {$A$};
\end{tikzpicture}
\caption{$A \in \mathcal{A} \cap \mathcal{B}$: case $n = 2$.} \label{f:ptech-ABn2}
\end{subfigure}
\hfill
\begin{subfigure}[b]{0.47\textwidth}
\centering
\begin{tikzpicture}[scale=0.9]
\fill [green,rounded corners=10pt,opacity=0.2] (-0.5,-1) rectangle (2.7,1);
\fill [green,rounded corners=10pt,opacity=0.2] (2.3,-1) rectangle (5.5,1);
\draw [dashed,green!50!black,rounded corners=10pt] (-0.5,-1) rectangle (2.7,1);
\draw [dashed,green!50!black,rounded corners=10pt] (2.3,-1) rectangle (5.5,1);
\node [green!50!black] at (1,-0.6) {$\mathcal{N}(H)$};
\node [green!50!black] at (4,-0.6) {$\mathcal{N}(H')$};

\draw [thick] (0,0) -- (1,1);
\draw [thick,rounded corners=5pt] (1,1) -- (2.5,0.6) -- (2.5,0);
\draw [thick] (5,0) -- (4,1);
\draw [thick,rounded corners=5pt] (4,1) -- (2.5,0.6) -- (2.5,0);
\fill (0,0) circle (1.5pt) node [below] {$p_0$};
\fill (1,1) circle (1.5pt) node [above] {$g_0$};
\fill (2.5,0) circle (1.5pt) node [below] {$p_1$};
\fill (2.5,0.4) circle (1.5pt) node [right,xshift=2pt] {$q_a=q_b$};
\fill (4,1) circle (1.5pt) node [above] {$g_1$};
\fill (5,0) circle (1.5pt) node [below] {$p_2$};

\draw [very thick,red] plot [smooth] coordinates { (1.5,-0.1) (2.5,0.2) (3.5,-0.1) } node [right] {$A$};
\end{tikzpicture}
\caption{$A \in \mathcal{A} \cap \mathcal{B}$: case $n = 1$.} \label{f:ptech-ABn1}
\end{subfigure}
\vspace{0.5em}

\begin{subfigure}[b]{0.45\textwidth}
\centering
\begin{tikzpicture}
\draw [thick] (-1,0) -- (0,-1) -- (1,0) -- (0,1);
\draw [thick,rounded corners=10pt] (-1,0) -- (1,2) -- (2,2);
\fill (-1,0) circle (1.5pt) node [left] {$\mathfrak{g}_{n-1}$};
\fill (0,-1) circle (1.5pt) node [below] {$q$};
\fill (1,0) circle (1.5pt) node [right] {$\mathfrak{g}_1'$};
\fill (0,1) circle (1.5pt) node [left] {$\mathfrak{g}_n$};
\fill (0.3,1.3) circle (1.5pt) node [above] {$q_c$};
\fill (2,2) circle (1.5pt) node [right] {$v$};
\draw [very thick,red] plot [smooth] coordinates { (-0.5,2) (0.15,1.15) (0,0) (0.7,-1) } node [right] {$A$};
\end{tikzpicture}
\caption{$A \in \mathcal{A}' \cap \mathcal{B}$.} \label{f:ptech-ApB}
\end{subfigure}
\hfill
\begin{subfigure}[b]{0.45\textwidth}
\centering
\begin{tikzpicture}
\draw [thick] (-1,0) -- (0.5,-1.5) -- (1.5,-0.5) -- (0,1) -- cycle;
\fill (-1,0) circle (1.5pt) node [left] {$\mathfrak{g}_{n-1}$};
\fill (0.5,-1.5) circle (1.5pt) node [below] {$q$};
\fill (1.5,-0.5) circle (1.5pt) node [right] {$\mathfrak{g}_1'$};
\fill (0,1) circle (1.5pt) node [above] {$\mathfrak{g}_n$};
\fill (0.5,0.5) circle (1.5pt) node [above] {$q_a'$};
\fill (0.8,0.2) circle (1.5pt) node [left,yshift=-3.5pt] {$q_a\!\!$};
\draw [very thick,red] plot [smooth] coordinates { (-1,1) (0,0.3) (0.65,0.35) (1.5,1) } node [above] {$A$};
\draw [very thick,blue] plot [smooth] coordinates { (-0.5,-1.5) (0.95,0.15) (1.5,0.5) } node [right] {$A'$};
\end{tikzpicture}
\caption{$A \in \mathcal{A}' \cap \mathcal{B}'$.} \label{f:ptech-ApBp}
\end{subfigure}
\caption{Proof of Proposition \ref{p:technical}: showing that $\mathcal{A} = \mathcal{A}'$ and $\mathcal{B} = \mathcal{B}'$.} \label{f:ptech}
\end{figure}

\subsection{Consequences of Proposition \ref{p:technical}} \label{ss:constech}

\begin{cor} \label{c:technical}
Let a group $G$ act specially on a quasi-median graph $X$. Let $H,H',K,K' \in V(\mathcal{C}X)$, and let $p,p',v,v' \in V(X)$ be such that $p \in \mathcal{N}(H)$, $p' \in \mathcal{N}(H')$, $v \in \mathcal{N}(K)$ and $v' \in \mathcal{N}(K')$. Suppose that $d_{\mathcal{C}X}(H,K) \geq d_{\mathcal{C}X}(H,H')+d_{\mathcal{C}X}(K,K')+3$.

If $\mathfrak{S}$ is a $v$-minimal contact sequence for $(H,H',p,p')$, then $\mathfrak{S}$ is also $v'$-minimal. Furthermore, if $(\mathcal{H}_0,\ldots,\mathcal{H}_n,\mathcal{H}_0',\ldots,\mathcal{H}_n')$ is the $(v,G)$-orbit sequence for $\mathfrak{S}$, then $\mathfrak{g}(v;p,\mathcal{H}_0,\ldots,\mathcal{H}_n) = \mathfrak{g}(v';p,\mathcal{H}_0,\ldots,\mathcal{H}_n)$.
\end{cor}

\begin{proof}
Let $m = d_{\mathcal{C}X}(K,K')$, and let $K = K_0,\ldots,K_m = K'$ be a geodesic in $\mathcal{C}X$. For $1 \leq i \leq m$, choose a vertex $v_i \in \mathcal{N}(K_{i-1}) \cap \mathcal{N}(K_i)$; let also $v_0 = v$ and $v_{m+1} = v'$. Let $n = d_{\mathcal{C}X}(H,H')$.

Given a contact sequence $\mathfrak{S} = (H_0,\ldots,H_n,p_0,\ldots,p_{n+1})$ for $(H,H',p,p')$ and any $v \in V(X)$, the tuples $\mathfrak{C}_\diagup(\mathfrak{S},v)$ and $\mathfrak{C}_\diagdown(\mathfrak{S},v)$ only depend on the gates for $v$ in the $\mathcal{N}(H_i)$, $0 \leq i \leq n$. In particular, if for all hyperplanes $A \in V(\mathcal{C}X)$ with $d_{\mathcal{C}X}(H,A) \leq n$ the gates for $v$ and $v'$ in $\mathcal{N}(A)$ coincide, then the set of $v$-minimal contact sequences for $(H,H',p,p')$ coincides with the set of $v'$-minimal ones.

Thus, let $A \in V(\mathcal{C}X)$ be a hyperplane with $d_{\mathcal{C}X}(H,A) \leq n$, and suppose for contradiction that $g \neq g'$, where $g$ and $g'$ are the gates for $v$ and $v'$ (respectively) in $\mathcal{N}(A)$. Let $B$ be a hyperplane separating $g$ and $g'$. Since $B$ separates two points in $\mathcal{N}(A)$, we must have $d_{\mathcal{C}X}(A,B) \leq 1$, and so $d_{\mathcal{C}X}(H,B) \leq n+1$. On the other hand, as $B$ separates the gates of $v$ and $v'$ in a gated subgraph, $B$ must also separate $v=v_0$ and $v'=v_{m+1}$. Thus $B$ must separate $v_i$ and $v_{i+1}$ for some $i \in \{ 0,\ldots,m \}$. As $v_i,v_{i+1} \in \mathcal{N}(K_i)$, it follows that $d_{\mathcal{C}X}(B,K_i) \leq 1$. In particular, $d_{\mathcal{C}X}(B,K) \leq d_{\mathcal{C}X}(B,K_i) + d_{\mathcal{C}X}(K_i,K) \leq i+1 \leq m+1$. But then $d_{\mathcal{C}X}(H,K) \leq d_{\mathcal{C}X}(H,B)+d_{\mathcal{C}X}(B,K) \leq n+m+2$, contradicting our assumption. Thus we must have $g = g'$, and so the set of $v$-minimal contact sequences for $(H,H',p,p')$ coincides with the set of $v'$-minimal ones. In particular, $\mathfrak{S}$ is a $v'$-minimal structural sequence for $(H,H',p,p')$, and so the conclusion of Proposition \ref{p:technical} holds if $v$ is replaced by $v'$ as well.

Now suppose for contradiction that $\mathfrak{g}_n(v) = \mathfrak{g}(v;p,\mathcal{H}_0,\ldots,\mathcal{H}_n)$ is not equal to $\mathfrak{g}_n(v') = \mathfrak{g}(v';p,\mathcal{H}_0,\ldots,\mathcal{H}_n)$. Let $B$ be a hyperplane separating $\mathfrak{g}_n(v)$ from $\mathfrak{g}_n(v')$. Then $B$ separates gates for $v$ and $v'$ in a gated subgraph, and so as above we get $d_{\mathcal{C}X}(B,K) \leq m+1$. On the other hand, since $B$ separates $\mathfrak{g}_n(v)$ from $\mathfrak{g}_n(v')$, it follows that $B$ separates $p$ from either $\mathfrak{g}_n(v)$ or $\mathfrak{g}_n(v')$: without loss of generality, suppose the former. Then $B$ must separate $\mathfrak{g}(v;p,\mathcal{H}_0,\ldots,\mathcal{H}_{j-1})$ and $\mathfrak{g}(v;p,\mathcal{H}_0,\ldots,\mathcal{H}_j)$ for some $j \in \{ 0,\ldots,n \}$. By Proposition \ref{p:technical} \ref{i:ptechnical-gateclose}, it then follows that $B$ separates $p_j$ from $g_j$, and so $d_{\mathcal{C}X}(B,H_j) \leq 1$; in particular, $d_{\mathcal{C}X}(H,B) \leq d_{\mathcal{C}X}(H,H_j)+d_{\mathcal{C}X}(H_j,B) \leq j+1 \leq n+1$. Therefore, $d_{\mathcal{C}X}(H,K) \leq d_{\mathcal{C}X}(H,B)+d_{\mathcal{C}X}(B,K) \leq n+m+2$, again contradicting our assumption. Thus we must have $\mathfrak{g}_n(v) = \mathfrak{g}_n(v')$, as required.
\end{proof}

\begin{lem} \label{l:fewplanes}
Suppose $G$ acts specially on $X$. Let $D \in \N$, and suppose every vertex of $\Delta X/G$ has at most $D$ neighbours. If $v,w \in V(X)$, then there exist at most $(D+1)^2$ hyperplanes $H \in V(\mathcal{C}X)$ such that $w \in \mathcal{N}(H)$ and $w$ is not the gate for $v$ in $\mathcal{N}(H)$.
\end{lem}

\begin{proof}
Let $\mathcal{U} \subseteq V(X)$ be the set of vertices $u \in V(X)$ such that $d_X(u,w)=1$ and $d_X(v,w) = d_X(v,u)+1$. We claim that $|\mathcal{U}| \leq D+1$. Indeed, suppose there exist $k$ distinct vertices $u_1,\ldots,u_k \in \mathcal{U}$, and let $H_i$ be the hyperplane separating $w$ and $u_i$ for $1 \leq i \leq k$. It is clear that $H_i \neq H_j$ whenever $i \neq j$: indeed, if $H_i = H_j = H$ then by Proposition \ref{p:geod} $H$ cannot separate $v$ from either $u_i$ or $u_j$, and therefore $u_i=u_j$, hence $i=j$. Since $w \in \mathcal{N}(H_i) \cap \mathcal{N}(H_j)$ for every $i,j$ and since the action $G \curvearrowright X$ is special, it also follows that $H_i^G \neq H_j^G$ whenever $i \neq j$.

We now claim that $H_i$ and $H_j$ intersect for every $i \neq j$. Indeed, $H_i$ cannot separate $u_i$ from $v$ (by Proposition \ref{p:geod}), nor $w$ from $u_j$ (as $H_i \neq H_j$), but it does separate $w$ (and so $u_j$) from $u_i$ (and so $v$). On the other hand, a symmetric argument shows that $H_j$ separates $w$ and $u_i$ from $u_j$ and $v$. Thus $H_i$ and $H_j$ must intersect, as claimed. Therefore, $d_{\Delta X}(H_i,H_j) = 1$ and so, as $H_i^G \neq H_j^G$, we have $d_{\Delta X/G}(H_i^G,H_j^G) = 1$. In particular, $\{ H_1^G,\ldots,H_k^G \}$ are $k$ vertices of a clique in $\Delta X/G$, and so by our assumption it follows that $k \leq D+1$. Thus $|\mathcal{U}| \leq D+1$, as claimed.

Now let $u \in \mathcal{U}$, and let $\mathcal{H} \subseteq V(\mathcal{C}X)$ be the set of hyperplanes $H \in V(\mathcal{C}X)$ such that $u,w \in \mathcal{N}(H)$. It is then enough to show that $|\mathcal{H}| \leq D+1$. Thus, let $H_1,H_2,\ldots,H_k \in \mathcal{H}$ be $k$ distinct hyperplanes, where $H_1$ is the hyperplane separating $u$ and $w$. As $w \in \mathcal{N}(H_i) \cap \mathcal{N}(H_j)$ for every $i,j$ and as $G \curvearrowright X$ is special, it is clear that $H_i^G \neq H_j^G$ for any $i \neq j$. Furthermore, it is clear (see, for instance, Proposition \ref{p:cfgated}) that $H_1$ and $H_j$ intersect for every $j \neq 1$. In particular, $d_{\Delta X}(H_1,H_j) = 1$, and so $d_{\Delta X/G}(H_1^G,H_j^G) = 1$. As by assumption $H_1^G$ has at most $D$ neighbours in $\Delta X/G$, it follows that $k \leq D+1$, and so $|\mathcal{H}| \leq D+1$, as required.
\end{proof}

\begin{thm} \label{t:acyl}
Suppose a group $G$ acts specially on a quasi-median graph $X$, and suppose there exists some $D \in \N$ such that $|\Stab_G(w)| \leq D$ for any $w \in V(X)$ and any vertex of $\Delta X/G$ has at most $D$ neighbours. Then the induced action $G \curvearrowright \mathcal{C}X$ is acylindrical, and the acylindricity constants $D_\varepsilon$ and $N_\varepsilon$ can be expressed as functions of $\varepsilon$ and $D$ only.
\end{thm}

\begin{proof}
Let $\varepsilon \in \N$. We claim that the acylindricity condition in Definition \ref{d:ah} is satisfied for $D_\varepsilon = 2\varepsilon+6$ and $N_\varepsilon = N^{2(\varepsilon+3)}D/(N-1)^2$, where $N = (D+1)^2 2^{D+1}$.

Indeed, let $h,k \in \Delta X$ be such that $d_{\mathcal{C}X}(h,k) \geq D_\varepsilon$. Let $H,K \in V(\mathcal{C}X)$ be hyperplanes such that $d_{\mathcal{C}X}(H,h) \leq 1/2$ and $d_{\mathcal{C}X}(K,k) \leq 1/2$, and note that we have $d_{\mathcal{C}X}(H,K) \geq D_\varepsilon-1 = 2\varepsilon+5$. Let $\mathcal{G}_\varepsilon(h,k) = \{ g \in G \mid d_{\mathcal{C}X}(h,h^g) \leq \varepsilon, d_{\mathcal{C}X}(k,k^g) \leq \varepsilon \}$, and note that we have $\mathcal{G}_\varepsilon(h,k) \subseteq \mathcal{G}_{\varepsilon+1}(H,K)$. We thus aim to show that $|\mathcal{G}_{\varepsilon+1}(H,K)| \leq N_\varepsilon$.

Pick vertices $v \in \mathcal{N}(K)$ and $p \in \mathcal{N}(H)$, and an element $g \in \mathcal{G}_{\varepsilon+1}(H,K)$. Let $\mathfrak{S}=(H_0,\ldots,H_n,p_0,\ldots,p_{n+1})$ be a $v$-minimal contact sequence for the tuple $(H,H^g,p,p^g)$ with $v$-gate $(g_0,\ldots,g_n)$ and $(v,G)$-orbit sequence $(\mathcal{H}_0,\ldots,\mathcal{H}_n,\mathcal{H}_0',\ldots,\mathcal{H}_n')$; as $g \in \mathcal{G}_{\varepsilon+1}(H,K)$, we have $n \leq \varepsilon+1$. For $0 \leq i \leq n$, set $\mathfrak{g}_i = \mathfrak{g}(v;p,\mathcal{H}_0,\ldots,\mathcal{H}_i)$ and $\mathfrak{g}_i' = \mathfrak{g}(v;p^g,\mathcal{H}_n',\ldots,\mathcal{H}_i')$; let also $\mathfrak{g}_{-1}=p$ and $\mathfrak{g}_{n+1}' = p^g$.

We first claim that there exist hyperplanes $A_0,\ldots,A_n \in V(\mathcal{C}X)$ such that $\mathfrak{g}_{i-1},\mathfrak{g}_i \in \mathcal{N}(A_i)$ for each $i$. Indeed, this is clear if $g_i = p_{i+1}$ for each $i$, as in that case we may simply take $A_i = H_i$ for each $i$. Otherwise, let $k \in \{ 0,\ldots,n \}$ be minimal such that $g_k \neq p_{k+1}$, and let $A$ be a hyperplane separating $g_k$ and $p_{k+1}$ such that $g_k \in \mathcal{N}(A)$. For $0 \leq i \leq k-1$ we may take $A_i = H_i$, while for $k \leq i \leq n$ we can show (by induction on $i$, say) that $\mathfrak{g}_i \in \mathcal{N}(A)$. Indeed, the base case ($i = k$) is clear by construction; and if $\mathfrak{g}_{i-1} \in \mathcal{N}(A)$ for some $i > k$ and $\mathfrak{g}_{i-1}=q_0,\ldots,q_m=\mathfrak{g}_i$ is a geodesic in $X$, then $A$ cannot osculate with the hyperplane separating $q_{j-1}$ and $q_j$ by Proposition \ref{p:technical} \ref{i:ptechnical-noosc} and \ref{i:ptechnical-gateclose}, and so $q_j \in \mathcal{N}(A)$ by induction on $j$. Thus we may take $A_i = A$ for $k \leq i \leq n$, so that $\mathfrak{g}_{i-1},\mathfrak{g}_i \in \mathcal{N}(A_i)$ for each $i$, as claimed. A symmetric argument shows that there exist hyperplanes $B_n,\ldots,B_0 \in V(\mathcal{C}X)$ such that $\mathfrak{g}_{i+1}',\mathfrak{g}_i' \in \mathcal{N}(B_i)$ for each $i$.

Now, we pass the sequence $(\mathfrak{g}_{-1},\ldots,\mathfrak{g}_n)$ to a subsequence $(\mathfrak{g}_{k_0},\ldots,\mathfrak{g}_{k_a})$ by removing those $\mathfrak{g}_i$ for which $\mathfrak{g}_{i-1}=\mathfrak{g}_i$. It then follows that $\mathfrak{g}_{k_{i-1}} \neq \mathfrak{g}_{k_i}$ and that $\mathfrak{g}_{k_{i-1}}, \mathfrak{g}_{k_i} \in \mathcal{N}(A_{k_i})$ for $1 \leq i \leq a$, where $a \leq n+1 \leq \varepsilon+2$. Similarly, we may pass the sequence $(\mathfrak{g}_{n+1}',\ldots,\mathfrak{g}_0')$ to a subsequence $(\mathfrak{g}_{k_0'},\ldots,\mathfrak{g}_{k_b'})$ such that $\mathfrak{g}_{k_{i-1}'} \neq \mathfrak{g}_{k_i'}$ and that $\mathfrak{g}_{k_{i-1}'}, \mathfrak{g}_{k_i}' \in \mathcal{N}(B_{k_i})$ for $1 \leq i \leq b$, where $b \leq n+1 \leq \varepsilon+2$.

Now as $d_{\mathcal{C}X}(H,H^g)+d_{\mathcal{C}X}(K,K^g)+3 \leq 2(\varepsilon+1)+3 = 2\varepsilon+5 \leq d_{\mathcal{C}X}(H,K)$, it follows from Corollary \ref{c:technical} that $\mathfrak{S}$ is also a $v^g$-minimal contact sequence and that $\mathfrak{g}(v;p,\mathcal{H}_0,\ldots,\mathcal{H}_n) = \mathfrak{g}(v^g;p,\mathcal{H}_0,\ldots,\mathcal{H}_n)$. Therefore, by Proposition \ref{p:technical} \ref{i:ptechnical-main} and the discussion above,
\begin{equation} \label{e:gates}
\begin{aligned}
\mathfrak{g}(v;p,\mathcal{H}_{k_1},\ldots,\mathcal{H}_{k_a}) &= \mathfrak{g}(v;p,\mathcal{H}_0,\ldots,\mathcal{H}_n) = \mathfrak{g}(v^g;p,\mathcal{H}_0,\ldots,\mathcal{H}_n) \\ &= \mathfrak{g}(v^g;p^g,\mathcal{H}_n',\ldots,\mathcal{H}_0') = \mathfrak{g}(v;p,\mathcal{H}_n',\ldots,\mathcal{H}_0')^g \\ &= \mathfrak{g}(v;p,\mathcal{H}_{k_1'}',\ldots,\mathcal{H}_{k_b'}')^g.
\end{aligned}
\end{equation}
As the stabiliser of $\mathfrak{g}(v;p,\mathcal{H}_{k_1},\ldots,\mathcal{H}_{k_a})$ has cardinality $\leq D$, it follows that, given any subsets $\mathcal{H}_{k_1},\ldots,\mathcal{H}_{k_a},\mathcal{H}_{k_1'}',\ldots,\mathcal{H}_{k_b'}' \subseteq V(\mathcal{C}X/G)$, there are at most $D$ elements $g \in G$ satisfying \eqref{e:gates}. But as $\mathfrak{g}_{k_{i-1}} \neq \mathfrak{g}_{k_i}$, as $\mathfrak{g}_{k_i}$ lies on a geodesic between $\mathfrak{g}_{k_{i-1}}$ and $v$, and as $\mathfrak{g}_{k_{i-1}},\mathfrak{g}_{k_i} \in \mathcal{N}(A_i)$, it follows from Lemma \ref{l:fewplanes} that there are at most $(D+1)^2$ possible choices for $A_{k_i}$ (for $1 \leq i \leq a$). Moreover, given a choice of $A_{k_i}$, as $\mathcal{H}_{k_i} \subseteq \st_{\Delta X/G}(A_{k_i}^G)$ and by assumption $|\st_{\Delta X/G}(A_{k_i}^G)| \leq D+1$, there exist at most $2^{D+1}$ choices for $\mathcal{H}_{k_i}$. It follows that there exist at most $N^a$ choices for the subsets $\mathcal{H}_{k_1},\ldots,\mathcal{H}_{k_a} \subseteq V(\mathcal{C}X/G)$, where $N = (D+1)^2 2^{D+1}$; similarly, there exist at most $N^b$ choices for the subsets $\mathcal{H}_{k_1'}',\ldots,\mathcal{H}_{k_b'}' \subseteq V(\mathcal{C}X/G)$. In particular,
\[
|\mathcal{G}_{\varepsilon+1}(H,K)| \leq D \left( \sum_{a=0}^{\varepsilon+2} N^a \right) \left( \sum_{b=0}^{\varepsilon+2} N^b \right) < D \left( \frac{N^{\varepsilon+3}}{N-1} \right)^2 = N_\varepsilon,
\]
as required.
\end{proof}

\section{Application to graph products} \label{s:gp}

We use this section to deduce results about graph products from Theorems \ref{t:qtree} and \ref{t:main}: namely, we show Corollary \ref{c:gp} in Section \ref{ss:nonelem} and Corollary \ref{c:AHacc} in Section \ref{ss:AHacc}. Throughout this section, let $\Gamma$ be a simplicial graph and let $\mathcal{G} = \{ G_v \mid v \in V(\Gamma) \}$ be a collection of non-trivial groups. Let $X$ be the quasi-median graph associated to $\GG$, as given by Theorem \ref{t:gengp}: namely, $X$ is the Cayley graph of $\GG$ with respect to the (not necessarily finite) generating set $S = \bigsqcup_{v \in V(\Gamma)} (G_v \setminus \{1\})$.

For the rest of the paper, given a subset $A \subseteq V(\Gamma)$, we write $\Gamma_A$ for the full subgraph of $\Gamma$ spanned by $A$ and denote $\{ G_v \mid v \in A \}$ by $\mathcal{G}_A$, and we view the graph product $\GG[A]$ as a subgroup of $\GG$. We will use the following result.

\begin{thm}[Genevois {\cite[Section 8.1]{genthesis}}; Genevois--Martin {\cite[Theorem 2.10]{gm}}] \label{t:gmNH}
For $v \in V(\Gamma)$, let $H_v$ be the hyperplane dual to the clique $G_v \subseteq X$. Then any hyperplane $H$ in $X$ is of the form $H_v^g$ for some $v \in V(\Gamma)$ and $g \in \GG$. Moreover, the vertices in $\mathcal{N}(H)$ are precisely $\GG[\st(v)]g \subseteq V(X)$.
\end{thm}

\begin{rmk}
Due to our convention to consider only right actions, the Cayley graph $X = \mathrm{Cay}(\GG,S)$ defined in Theorem \ref{t:gengp} is the \emph{left} Cayley graph: for $s \in S$ and $g \in \GG$, an edge labelled $s$ joins $g \in V(X)$ to $sg \in V(X)$. Therefore, contrary to the convention in \cite{genthesis,gm}, the vertices in the carrier of a hyperplane will form a \emph{right} coset of $\GG[\st(v)]$ for some $v \in V(\Gamma)$.
\end{rmk}

\subsection{Acylindrical hyperbolicity} \label{ss:nonelem}

Here we prove Corollary \ref{c:gp}. It is clear from Theorem \ref{t:gengp} and Lemma \ref{l:gengp-dxiso} that we may apply Theorems \ref{t:qtree} and \ref{t:main} to the quasi-median graph $X$ associated to a graph product $\GG$. In particular, it follows that the contact graph $\mathcal{C}X$ is a quasi-tree and $\GG$ acts on it acylindrically. We thus only need to show that, given that $|V(\Gamma)| \geq 2$ and the complement $\Gamma^C$ of $\Gamma$ is connected, either the action $\GG \curvearrowright \mathcal{C}X$ is non-elementary or $\GG \cong C_2 * C_2$.

\begin{lem} \label{l:starunbounded}
The following are equivalent:
\begin{enumerate}[label={\normalfont({\roman*})}]
\item \label{i:lsu1} $\mathcal{C}X$ is unbounded;
\item \label{i:lsu2} $\Gamma^C$ is connected and $|V(\Gamma)| \geq 2$.
\end{enumerate}
\end{lem}

\begin{proof}
We first show \ref{i:lsu1} $\Rightarrow$ \ref{i:lsu2}. Indeed, if $\Gamma$ is a single vertex $v$, then $X$ is a single clique and so $\mathcal{C}X$ is a single vertex. On the other hand, if $\Gamma^C$ is disconnected, then we have a partition $V(\Gamma) = A \sqcup B$ where $A$ and $B$ are adjacent and non-empty. In particular, $\GG = \GG[A] \times \GG[B]$, and so any vertex $g \in \GG$ of $X$ can be expressed as $g = g_A g_B$ for some $g_A \in \GG[A]$ and $g_B \in \GG[B]$. Thus, if $H \in V(\mathcal{C}X)$ then by Theorem \ref{t:gmNH} $\mathcal{N}(H) = \GG[\st(v)] g_A g_B$ for some $g_A \in \GG[A]$, $g_B \in \GG[B]$ and $v \in V(\Gamma)$: without loss of generality, suppose $v \in A$. Then $g_B \in \GG[B] \leq \GG[\st(v)]$ and $g_A \in \GG[A] \leq \GG[\st(u)]$ for any $u \in B$, and so $g_A \in \mathcal{N}(H) \cap \mathcal{N}(H_u)$; therefore, $d_{\mathcal{C}X}(H,H_u) \leq 1$. Since $1 \in \mathcal{N}(H_u) \cap \mathcal{N}(H_v)$ and so $d_{\mathcal{C}X}(H_u,H_v) \leq 1$ for any $u,v \in V(\Gamma)$, it follows that $d_{\mathcal{C}X}(H,H') \leq 3$ for any $H,H' \in \mathcal{C}X$ and so $\mathcal{C}X$ is bounded, as required.

To show \ref{i:lsu2} $\Rightarrow$ \ref{i:lsu1}, suppose that $\Gamma$ is a graph with at least $2$ vertices and connected complement. Thus, there exists a closed walk $(v_0,v_1,\ldots,v_\ell)$ on the complement of $\Gamma$ that visits every vertex -- in particular, we have $v_i \in V(\Gamma)$ with $v_\ell = v_0$ and $v_{i-1} \neq v_i$, $(v_{i-1},v_i) \notin E(\Gamma)$ for $1 \leq i \leq \ell$. Pick arbitrary non-identity elements $g_i \in G_{v_i}$ for $i = 1,\ldots,\ell$, and consider the element $g = g_1 \cdots g_\ell \in \GG$.

Now let $n \in \N$, and let $A,B \in V(\mathcal{C}X)$ be such that $1 \in \mathcal{N}(A)$ and $g^n \in \mathcal{N}(B)$. Let $A=A_0,\ldots,A_m=B$ be the geodesic in $\mathcal{C}X$ and let $1=p_0,\ldots,p_{m+1}=g^n$ be the vertices in $X$ given by Proposition \ref{p:liftinggeodesics}. It follows from the normal form theorem for graph products \cite[Theorem 3.9]{green} that $\overbrace{(g_1 \cdots g_\ell) \cdots (g_1 \cdots g_\ell)}^n$ is the \emph{unique} normal form for the element $g^n$. In particular, as geodesics in $X$ are precisely the words spelling out normal forms of elements of $\GG$, we have $p_i = g_{\ell n-c_i+1} g_{\ell n-c_i+2} \cdots g_{\ell n}$, where $0 = c_0 \leq c_1 \leq \cdots \leq c_{m+1} = \ell n$ and indices are taken modulo $\ell$.

We now claim that $c_{i+1}-c_i < \ell$ for each $i$. Indeed, suppose $c_{i+1}-c_i \geq \ell$ for some $i$. Note that, as $p_i,p_{i+1} \in \mathcal{N}(A_i)$, it follows from Theorem \ref{t:gmNH} that $\GG[\st(v)] p_i = V(\mathcal{N}(A_i)) = \GG[\st(v)] p_{i+1}$ for some $v \in V(\Gamma)$, and therefore we have $p_{i+1} p_i^{-1} \in \GG[\st(v)]$. But as $g_{\ell n-c_{i+1}} g_{\ell n-c_{i+1}+1} \cdots g_{\ell n-c_i-1}$ is a normal form for $p_{i+1}p_i^{-1}$ (where indices are taken modulo $\ell$), it follows that $v_j \in \st(v)$ for $\ell n-c_{i+1} \leq j < \ell n-c_i$ (with indices again modulo $\ell$). But as by assumption $c_{i+1}-c_i \geq \ell$ and as $\{v_1,\ldots,v_\ell\} = V(\Gamma)$, this implies that $\st(v) = V(\Gamma)$, and so $v$ is an isolated vertex of $\Gamma^C$. This contradicts the fact that $\Gamma^C$ is connected; thus $c_{i+1}-c_i < \ell$ for each $i$, as claimed.

In particular, we get $\ell n = \sum_{i=0}^m (c_{i+1}-c_i) < (m+1) \ell$, and so $m+1 > n$. Thus $d_{\mathcal{C}X}(A,B) = m \geq n$ and so $\mathcal{C}X$ is unbounded, as required.
\end{proof}

It is now easy to deduce when the action of $\GG$ on $\mathcal{C}X$ is non-elementary acylindrical.

\begin{proof}[Proof of Corollary \ref{c:gp}]
By the argument above, we only need to show the last part. Thus, suppose that $\Gamma$ is a graph with at least $2$ vertices and connected complement. Then, by Lemma \ref{l:starunbounded}, the graph $\mathcal{C}X$ is unbounded. In particular, given any $H \in V(\mathcal{C}X)$ and $n \in \N$, we may pick $H' \in V(\mathcal{C}X)$ such that $d_{\mathcal{C}X}(H,H') \geq n+1$. Since the action $\GG \curvearrowright X$ is transitive on vertices, it follows that given any vertex $p \in \mathcal{N}(H)$ there exists $g \in \GG$ such that $p^g \in \mathcal{N}(H')$, and in particular $d_{\mathcal{C}X}(H^g,H') \leq 1$. Thus $d_{\mathcal{C}X}(H,H^g) \geq n$, and so the action $\GG \curvearrowright \mathcal{C}X$ has unbounded orbits.

It is left to show that $\GG$ is either isomorphic to $C_2 * C_2$ or not virtually cyclic. Indeed, if $\Gamma$ has only two vertices ($v$ and $w$, say) and $G_v \cong G_w \cong C_2$, then $\Gamma$ is a discrete graph with two vertices and $\GG \cong G_v * G_w \cong C_2 * C_2$. Otherwise, we claim that there exists a subset $A \subseteq V(\Gamma)$ such that $\GG[A]$ is a free product of two non-trivial groups, not both $C_2$. If $|V(\Gamma)| = 2$ (and so $V(\Gamma) = \{ v,w \}$) then $\GG \cong G_v * G_w$ and so we may take $A = V(\Gamma)$. Otherwise, since $|V(\Gamma)| \geq 3$ and $\Gamma^C$ is connected, $\Gamma^C$ contains a path of length $2$, and so there exist vertices $v_1,v_2,w \in \Gamma$ such that $v_1 \nsim w \nsim v_2$. Let $A = \{v_1,v_2,w\}$ and $H = \GG[\{v_1,v_2\}]$ (so either $H \cong G_{v_1} \times G_{v_2}$ or $H \cong G_{v_1} * G_{v_2}$). Since the groups $G_v$ are non-trivial for each $v \in V(\Gamma)$, we have $|H| \geq 4 > 2$ and so $\GG[A] \cong G_w * H$ satisfies our conditions, which proves the claim.

Thus, in either case $\GG[A]$ has infinitely many ends, and so is not virtually cyclic. In particular, since the subgroup $\GG[A] \leq \GG$ is not virtually cyclic, neither is $\GG$, as required.
\end{proof}

\begin{rmk} \label{r:ah}
After appearance of the first version of this paper in a form of a preprint, it has been brought to the author's attention that most of the results stated in Corollary \ref{c:gp} have already been proved by Genevois. In \cite[Theorem 4.11]{gen2}, Genevois shows that $\Delta X$ is quasi-isometric to a tree whenever it is connected and $\Gamma$ is finite, so in particular, by Theorem \ref{t:main} \ref{i:tmain-qi}, $\mathcal{C}X$ is a quasi-tree as well. Moreover, methods used by Genevois to prove \cite[Theorem 4.1]{gen1} can be adapted to show that the action of $\GG$ on $\mathcal{C}X$ is non-uniformly acylindrical; here, the \emph{non-uniform acylindricity} of an action $G \curvearrowright X$ is a weaker version of acylindricity, defined by replacing the phrase `is bounded above by $N_\varepsilon$' by `is finite' in Definition \ref{d:ah}. Corollary \ref{c:gp} strengthens this statement.
\end{rmk}

\subsection{\texorpdfstring{$\mathcal{AH}$}{AH}-accessibility} \label{ss:AHacc}

Here we study $\mathcal{AH}$-accessibility, introduced in \cite{abo} by Abbott, Balasubramanya and Osin, of graph products. In particular, we show that if $\Gamma$ is connected, non-trivial, and the groups in $\mathcal{G}$ are infinite, then the action of $\GG$ on $\mathcal{C}X$ is the `largest' acylindrical action of $\GG$ on a hyperbolic metric space. Hence we prove Corollary \ref{c:AHacc}.

We briefly recall the terminology of \cite{abo}. Given two isometric actions $G \curvearrowright X$ and $G \curvearrowright Y$ of a group $G$, we say $G \curvearrowright X$ \emph{dominates} $G \curvearrowright Y$, denoted $G \curvearrowright Y \preceq G \curvearrowright X$, if there exist $x \in X$, $y \in Y$ and a constant $C$ such that
\[
d_Y(y,y^g) \leq Cd_X(x,x^g)+C
\]
for all $g \in G$. The actions $G \curvearrowright X$ and $G \curvearrowright Y$ are said to be \emph{weakly equivalent} if $G \curvearrowright X \preceq G \curvearrowright Y$ and $G \curvearrowright Y \preceq G \curvearrowright X$. This partitions all such actions into equivalence classes.

It is easy to see that $\preceq$ defines a preorder on the set of all isometric actions of $G$ on metric spaces. Therefore, $\preceq$ defines a partial order on the set of equivalence classes of all such actions. We may restrict this to a partial order on the set $\mathcal{AH}(G)$ of equivalence classes of \emph{acylindrical} actions of $G$ on a \emph{hyperbolic} space. We then say the group $G$ is \emph{$\mathcal{AH}$-accessible} if the partial order $\mathcal{AH}(G)$ has a largest element (which, if exists, must necessarily be unique), and we say $G$ is \emph{strongly $\mathcal{AH}$-accessible} if a representative of this largest element is a Cayley graph of $G$.

Recall that for an action $G \curvearrowright X$ by isometries with $X$ hyperbolic, an element $g \in G$ is said to be \emph{loxodromic} if, for some (or any) $x \in X$, the map $\Z \to X$ given by $n \mapsto x^{g^n}$ is a quasi-isometric embedding. It is clear from the definitions that the `largest' action $G \curvearrowright X$ will also be \emph{universal}, in the sense that every element of $G$, that is loxodromic with respect to \emph{some} acylindrical action of $G$ on a hyperbolic space, will be loxodromic with respect to $G \curvearrowright X$.

In \cite[Theorem 2.18 (c)]{abo}, it is shown that the all right-angled Artin groups are $\mathcal{AH}$-accessible (and more generally, so are all hierarchically hyperbolic groups -- in particular, groups acting properly and cocompactly on a CAT(0) cube complex possessing a factor system \cite[Theorem A]{abd}). Here we generalise this result to `most' graph products of infinite groups. The proof is very similar to that of \cite[Lemma 7.16]{abo}.

\begin{proof}[Proof of Corollary \ref{c:AHacc}]
It is easy to show -- for instance, by Theorem \ref{t:gmNH} -- that $\mathcal{C}X$ is ($G$-equivariantly) quasi-isometric to the Cayley graph of $\GG$ with respect to $\bigcup_{v \in V(\Gamma)} \GG[\st(v)]$.

We prove the statement by induction on $|V(\Gamma)|$. If $|V(\Gamma)| = 1$ ($V(\Gamma)| = \{v\}$, say), then $v$ is an isolated vertex of $\Gamma$ and so, by the assumption, $\GG \cong G_v$ is strongly $\mathcal{AH}$-accessible.

Suppose now that $|V(\Gamma)| \geq 2$. If $\Gamma$ has an isolated vertex ($\Gamma = \Gamma_A \sqcup \{v\}$ for some partition $V(\Gamma) = A \sqcup \{v\}$, say), then $\GG \cong \GG[A] * G_v$ is hyperbolic relative to $\{ \GG[A], G_v \}$. By the induction hypothesis, both $\GG[A]$ and $G_v$ are strongly $\mathcal{AH}$-accessible, and hence, by \cite[Theorem 7.9]{abo}, so is $\GG$. If, on the other hand, the complement $\Gamma^C$ of $\Gamma$ is disconnected ($\Gamma^C = \Gamma_A^C \sqcup \Gamma_B^C$ for some partition $V(\Gamma) = A \sqcup B$, say), then $\GG \cong \GG[A] \times \GG[B]$ is not acylindrically hyperbolic by \cite[Corollary 7.2]{osinAH}, as both $\GG[A]$ and $\GG[B]$ are infinite. It then follows from \cite[Example 7.8]{abo} that $\GG$ is strongly $\mathcal{AH}$-accessible; it also follows that \emph{any} acylindrical action of $\GG$ on a hyperbolic metric space ($\GG \curvearrowright \mathcal{C}X$, say) represents the largest element of $\mathcal{AH}(\GG)$.

Hence, we may without loss of generality assume that $\Gamma$ is a graph with no isolated vertices and connected complement. It then follows that $|V(\Gamma)| \geq 4$, and so by Corollary \ref{c:gp}, $\mathcal{C}X$ is a hyperbolic metric space and $\GG$ acts on it non-elementarily acylindrically. It is easy to see from Theorem \ref{t:gmNH} that, given two hyperplanes $H,H' \in V(\mathcal{C}X)$, they are adjacent in $\mathcal{C}X$ if and only if there exist distinct $u,v \in V(\Gamma)$ and $g \in \GG$ such that $H = H_u^g$ and $H' = H_v^g$. It follows that the quotient space $\mathcal{C}X/\GG$ is the complete graph on $|V(\Gamma)|$ vertices, and in particular, the action $\GG \curvearrowright \mathcal{C}X$ is cocompact.

Moreover, it follows from Theorem \ref{t:gmNH} that the stabiliser of an arbitrary vertex $H_v^g$ of $\mathcal{C}X$ is precisely $G \cong (\GG[\st(v)])^g \cong G_v^g \times (\GG[\lk(v)])^g$. Since $\Gamma$ has no isolated vertices, $\lk(v) \neq \varnothing$, and so, as all groups in $\mathcal{G}$ are infinite, both $G_v^g$ and $(\GG[\lk(v)])^g$ are infinite groups. Thus, $G$ is a direct product of two infinite groups, and so -- by \cite[Corollary 7.2]{osinAH}, say -- $G$ does not possess a non-elementary acylindrical action on a hyperbolic space. Since $G$ is not virtually cyclic, for every acylindrical action of $\GG$ on a hyperbolic space $Y$, the induced action of $G$ on $Y$ has bounded orbits. It then follows from \cite[Proposition 4.13]{abo} that $\GG$ is strongly $\mathcal{AH}$-accessible -- and in particular, $\GG \curvearrowright \mathcal{C}X$ represents the largest element of $\mathcal{AH}(\GG)$.
\end{proof}

\begin{rmk} \label{r:AHacc}
Corollary \ref{c:AHacc} gives some explicit descriptions for the class of hierarchically hyperbolic groups, introduced by Behrstock, Hagen and Sisto in \cite{bhs}. In particular, a result by Berlai and Robbio \cite[Theorem C]{br} says that if all vertex groups $G_v$ are hierarchically hyperbolic with the intersection property and clean containers, then the same can be said about $\GG$. Moreover, Abbott, Behrstock and Durham show in \cite[Theorem A]{abd} that all hierarchically hyperbolic groups are $\mathcal{AH}$-accessible, which implies Corollary \ref{c:AHacc} in the case when the vertex groups $G_v$ are hierarchically hyperbolic with the intersection property and clean containers.

More precisely, every hierarchically hyperbolic group $G$ comes with an action on a space $\mathcal{X}$, such that there exist projections $\pi_Y: \mathcal{X} \to 2^{\hat{\mathcal{C}}Y}$ to some collection of $\delta$-hyperbolic spaces $\{ \hat{\mathcal{C}}Y \mid Y \in \mathfrak{S} \}$, where $\mathfrak{S}$ is a partial order that contains a (unique) largest element, $S \in \mathfrak{S}$, say. Moreover, the action of $G$ on $\mathcal{X}$ induces an action of $G$ on (a space quasi-isometric to) $\mathcal{U} = \bigcup_{x \in \mathcal{X}} \pi_S(x) \subseteq \hat{\mathcal{C}}S$, and in \cite[Theorem 14.3]{bhs} it is shown that this action is acylindrical. In \cite{abd}, this construction is modified so that the action $G \curvearrowright \mathcal{U}$ represents the largest element of $\mathcal{AH}(G)$. If $\Gamma$ is connected, non-trivial, and the groups $G_v$ are infinite and hierarchically hyperbolic (with the intersection property and clean containers), then the proof of Corollary \ref{c:AHacc} gives this action $\GG \curvearrowright \mathcal{U}$ explicitly. This is potentially useful for studying hierarchical hyperbolicity of graph products.
\end{rmk}

\begin{rmk}
Note that the condition on the $G_v$ being infinite is necessary for the proof to work. Indeed, suppose
$\Gamma = \begin{tikzpicture}
\draw (0,0) -- (1.5,0);
\fill (0,0) circle (2pt) node [above] {$a$};
\fill (0.5,0) circle (2pt) node [above] {$b$};
\fill (1,0) circle (2pt) node [above] {$c$};
\fill (1.5,0) circle (2pt) node [above] {$d$};
\end{tikzpicture}$
is a path of length $3$, and $G_v = \langle g_v \rangle \cong C_2$ for each $v \in V(\Gamma)$, so that $\GG$ is the right-angled Coxeter group over $\Gamma$. Notice that $\GG \cong A *_C B$, where $A = G_b \times (G_a*G_c)$, $B = G_c \times G_d$ and $C = G_c$. In particular, since $C$ is finite, $\GG$ is hyperbolic relative to $\{A,B\}$. Hence the Cayley graph $\operatorname{Cay}(\GG,A \cup B)$ is hyperbolic and the usual action of $\GG$ on it is acylindrical.

It is easy to verify from the normal form theorem for amalgamated free products that the element $g_b g_d$ will be loxodromic with respect to $\GG \curvearrowright \operatorname{Cay}(\GG,A \cup B)$. However, as $g_b g_d \in \GG[\st(c)]$, we know that $g_b g_d$ stabilises the hyperplane dual to $G_c \subseteq V(X)$ under the action of $\GG$ on $\mathcal{C}X$, and so $g_bg_d$ is not loxodromic with respect to $\GG \curvearrowright \mathcal{C}X$. In particular, the equivalence class of $\GG \curvearrowright \mathcal{C}X$ cannot be the largest element of $\mathcal{AH}(\GG)$. It is straightforward to generalise this argument to show that if $c \in V(\Gamma)$ is a separating vertex of a connected finite simplicial graph $\Gamma$, then for \emph{any} graph product $\GG$ with $G_c$ finite, the action $\GG \curvearrowright \mathcal{C}X$ will not be the `largest' one.

On the other hand, note that this particular group $\GG$ (and indeed any right-angled Coxeter group) will be $\mathcal{AH}$-accessible: see \cite[Theorem A (4)]{abd}.
\end{rmk}

\section{Equational noetherianity of graph products} \label{s:eqn}

In this section we prove Theorem \ref{t:en}. To do this, we use the methods that Groves and Hull exhibited in \cite{gh}. Here we briefly recall their terminology.

The approach to equationally noetherian groups used in \cite{gh} is through sequences of homomorphisms. In particular, let $G$ be any group, let $F$ be a finitely generated group and let $\varphi_i: F \to G$ be a sequence of homomorphisms ($i \in \N$). Let $\omega: \mathcal{P}(\N) \to \{0,1\}$ be a non-principal ultrafilter. We say a sequence of properties $(P_i)_{i \in \N}$ holds \emph{$\omega$-almost surely} if $\omega(\{ i \in \N \mid P_i \text{ holds} \}) = 1$. We define the \emph{$\omega$-kernel} of $F$ with respect to $(\varphi_i)$ to be
\[
F_{\omega,(\varphi_i)} = \{ f \in F \mid \varphi_i(f) = 1 \text{ $\omega$-almost surely} \};
\]
we write $F_\omega$ for $F_{\omega,(\varphi_i)}$ if the sequence $(\varphi_i)$ is clear. It is easy to check that $F_\omega$ is a normal subgroup of $F$. We say $\varphi_i$ \emph{factors through $F_\omega$ $\omega$-almost surely} if $F_\omega \subseteq \ker(\varphi_i)$ $\omega$-almost surely.

The idea behind all these definitions is the following result.

\begin{thm}[Groves and Hull {\cite[Theorem 3.5]{gh}}] \label{t:ghen}
Let $\omega$ be a non-principal ultrafilter. Then the following are equivalent for any group $G$:
\begin{enumerate}[label={\normalfont({\roman*})}]
\item $G$ is equationally noetherian;
\item for any finitely generated group $F$ and any sequence of homomorphisms $(\varphi_i: F \to G)$, $\varphi_i$ factors through $F_\omega$ $\omega$-almost surely.
\end{enumerate}
\end{thm}

\begin{rmk}
Note that Definition \ref{d:en} differs from the usual definition of equationally noetherian groups, as we do not allow `coefficients' in our equations: that is, we restrict to subsets $S \subseteq F_n$ instead of $S \subseteq G * F_n$. However, the two concepts coincide when $G$ is finitely generated -- see \cite[\S 2.2, Proposition 3]{bmr}. We use this (weaker) definition of equationally noetherian groups as it is more suitable for our methods. In particular, we use an equivalent characterisation of equationally noetherian groups given by Theorem \ref{t:ghen}.
\end{rmk}

The structure of this section is as follows. In Section \ref{ss:linking}, we introduce `admissible' graphs and show that being equationally noetherian is preserved under taking graph products over connected admissible graphs. In Section \ref{ss:dvkds}, we introduce (polygonal) dual van Kampen diagrams for graph products -- a tool that we use in our proof of Theorem \ref{t:en}. In Section \ref{ss:pr}, we introduce minimal polygonal representations for words in a free group, closely related to polygonal van Kampen diagrams defined in the preceding subsection. In Section \ref{ss:girth}, we give our main technical results in order to show that all connected graphs of girth $\geq 6$ are admissible. In Section \ref{ss:ten}, we combine these results to prove Theorem \ref{t:en}.

\subsection{Reduction to sequences of linking homomorphisms} \label{ss:linking}

Suppose now that the group $G$ acts by isometries on a metric space $(Y,d)$. As before, let $F$ be a finitely generated group, $\omega$ a non-principal ultrafilter, and $(\varphi_i: F \to G)_{i=1}^\infty$ a sequence of homomorphisms. Pick a finite generating set $S$ for $F$. We say that the sequence of homomorphisms $(\varphi_i)$ is \emph{non-divergent} if
\[
\lim_\omega \inf_{y \in Y} \max_{s \in S} d(y,y^{\varphi_i(s)}) < \infty.
\]
We say that $(\varphi_i)$ is \emph{divergent} otherwise. It is easy to see that this does not depend on the choice of a generating set $S$ for $F$.

The main technical result of \cite{gh} states that in case $Y$ is hyperbolic and the action of $G$ on $Y$ is non-elementary acylindrical, it is enough to consider non-divergent sequences of homomorphisms (cf Theorem \ref{t:ghen}).

\begin{thm}[Groves and Hull {\cite[Theorem B]{gh}}] \label{t:ghmain}
Let $Y$ be a hyperbolic metric space and $G$ a group acting non-elementarily acylindrically on $Y$. Suppose that for any finitely generated group $F$ and any non-divergent sequence of homomorphisms $(\varphi_i: F \to G)$, $\varphi_i$ factors through $F_\omega$ $\omega$-almost surely. Then $G$ is equationally noetherian.
\end{thm}

We now consider the particular case when $G$ is a graph product and $Y = \mathcal{C}X$ is the crossing graph of the quasi-median graph $X$ associated to $G$. Thus, as before, let $\Gamma$ be a finite simplicial graph and let $\mathcal{G} = \{ G_v \mid v \in V(\Gamma) \}$ be a collection of non-trivial groups. It turns out that in this case we may reduce any non-divergent sequence of homomorphisms to one of the following form: see the proof of Theorem \ref{t:engood}.

\begin{defn} \label{d:linking}
%Let $\mathcal{F} = \{ F_v \mid v \in V(\Gamma) \}$ be a collection of finitely generated groups, and let $\varphi: \GF \to \GG$ be a homomorphism. We say $\varphi$ is \emph{linking} if $\varphi(F_v) \subseteq \GG[\lk(v)]$ for each $v \in V(\Gamma)$. We say the graph $\Gamma$ is \emph{admissible} if for every collection of non-trivial equationally noetherian groups $\mathcal{G} = \{ G_v \mid v \in V(\Gamma) \}$ and every sequence of linking homomorphisms $(\varphi_i: \GF \to \GG)_{i=1}^\infty$, $\varphi_i$ factors through $(\GF)_\omega$ $\omega$-almost surely.
Let $F$ be a finitely generated free group, let $S \subset F$ be a free basis for $F$, and let $\varphi: F \to \GG$ be a homomorphism. We say $\varphi$ is \emph{linking} if for each $s \in S$, there exists $v = v(s) \in V(\Gamma)$ such that $\varphi(s) \in \GG[\lk(v)]$; although in general this depends on the choice of the free basis $S$, we will usually fix $S$ and omit references to it. We say the graph $\Gamma$ is \emph{admissible} if for every collection of non-trivial equationally noetherian groups $\mathcal{G} = \{ G_v \mid v \in V(\Gamma) \}$ and every sequence of linking homomorphisms $(\varphi_i: F \to \GG)_{i=1}^\infty$, $\varphi_i$ factors through $F_\omega$ $\omega$-almost surely.
\end{defn}

The proof of Theorem \ref{t:engood} uses the following result.

\begin{lem} \label{l:sgphsgood}
Full subgraphs of admissible graphs are admissible.
\end{lem}

\begin{proof}
Let $\Gamma$ be a admissible graph, let $\mathcal{G} = \{ G_v \mid v \in V(\Gamma) \}$ be a collection of non-trivial equationally noetherian groups, and let $F$ be a finitely generated free group. Let $A \subseteq V(\Gamma)$, so that $\Gamma_A$ is a full subgraph of $\Gamma$, and let $(\varphi_i^A: F \to \GG[A])_{i=1}^\infty$ be a sequence of linking homomorphisms. Let $\omega$ be a non-principal ultrafilter. We aim to show that $\varphi_i^A$ factors through $F_{\omega,(\varphi_i^A)}$ $\omega$-almost surely.

Note that we have a canonical inclusion of subgroup $\iota_A: \GG[A] \to \GG$. For each $i$, let $\varphi_i = \iota_A \circ \varphi_i^A: F \to \GG$. It is easy to see that the $\varphi_i$ are linking homomorphisms. In particular, since $\Gamma$ is admissible, we have $F_{\omega,(\varphi_i)} \subseteq \ker\varphi_i$ $\omega$-almost surely. Moreover, since $\iota_A$ is injective, we have $\ker\varphi_i^A = \ker\varphi_i$ for each $i$ and $F_{\omega,(\varphi_i^A)} = F_{\omega,(\varphi_i)}$. Thus $F_{\omega,(\varphi_i^A)} \subseteq \ker\varphi_i^A$ $\omega$-almost surely, and so $\Gamma_A$ is admissible, as required.
\end{proof}

\begin{thm} \label{t:engood}
For any connected admissible graph $\Gamma$ and any collection $\mathcal{G} = \{ G_v \mid v \in V(\Gamma) \}$ of equationally noetherian groups, the graph product $\Gamma\mathcal{G}$ is equationally noetherian.
\end{thm}

\begin{proof}
We proceed by induction on $|V(\Gamma)|$. If $|V(\Gamma)| = 1$ ($V(\Gamma) = \{v\}$, say) then $\GG \cong G_v$, and so the result is clear. Thus, assume that $|V(\Gamma)| \geq 2$.

%If $\Gamma$ is disconnected, then we have a partition $V(\Gamma) = A \sqcup B$ into non-empty subsets such that $\Gamma = \Gamma_A \sqcup \Gamma_B$, and so $\GG \cong \GG[A] * \GG[B]$. By Lemma \ref{l:sgphsgood}, both $\Gamma_A$ and $\Gamma_B$ are admissible, and so by the induction hypothesis, both $\GG[A]$ and $\GG[B]$ are equationally noetherian. By Theorem \ref{t:sela}, $\GG$ is equationally noetherian as well, as required. Thus, without loss of generality, we may assume that $\Gamma$ is connected.

If the complement of $\Gamma$ is disconnected, then we have a partition $V(\Gamma) = A \sqcup B$ such that $\GG \cong \GG[A] \times \GG[B]$. By Lemma \ref{l:sgphsgood}, both $\Gamma_A$ and $\Gamma_B$ are admissible, and so $\GG[A]$ and $\GG[B]$ are equationally noetherian by the induction hypothesis. It is clear from the definition that a direct product $G \times H$ of equationally noetherian groups $G$ and $H$ is equationally noetherian: indeed, this follows from the cartesian product decomposition $V_{G \times H}(S) = V_G(S) \times V_H(S)$, for any $S \subseteq F_n$. Thus $\GG$ is equationally noetherian. % in this case as well.

Therefore, we may without of loss of generality assume that $\Gamma$ is a connected graph with a connected complement and with $|V(\Gamma)| \geq 2$ (and, therefore, $|V(\Gamma)| \geq 4$). In this case, Corollary \ref{c:gp} shows that $\mathcal{C}X$ is a hyperbolic metric space and the action of $\GG$ on it is non-elementary acylindrical. We thus may use Theorem \ref{t:ghmain} to show that $\GG$ is equationally noetherian.

In particular, let $F$ be a finitely generated group and let $(\varphi_i: F \to \GG)_{i=1}^\infty$ be a non-divergent sequence of homomorphisms. By Theorem \ref{t:ghmain}, it is enough to show that $\varphi_i$ factors through $F_\omega$ $\omega$-almost surely.

We proceed as in the proof of \cite[Theorem D]{gh}. Let $S$ be a finite generating set for $F$. Note that, by Theorem \ref{t:gmNH}, we may conjugate each $\varphi_i$ (if necessary) to assume that the minimum (over all hyperplanes $H$ in $X$) of $\max_{s \in S} d_{\mathcal{C}X}(H,H^{\varphi_i(s)})$ is attained for $H = H_u$ for some $u \in V(\Gamma)$. Moreover, it is easy to see from Theorem \ref{t:gmNH} that $\left| \|g\|_* - d_{\mathcal{C}X}(H_u,H_u^g) \right| \leq 1$ for any $g \in \GG$ and $u \in V(\Gamma)$, where we write $\|g\|_*$ for the minimal integer $\ell \in \N$ such that $g = g_1 \cdots g_\ell$ and $g_i \in \GG[\st(v_i)]$ for some $v_i \in V(\Gamma)$. In particular, since the sequence $(\varphi_i)$ is non-divergent, it follows that
\[
\lim_\omega \max_{s \in S} \|\varphi_i(s)\|_* < \infty.
\]

It follows that for each $s \in S$, there exists $\hat{n}_s \in \mathbb{N}$ such that $\|\varphi_i(s)\|_* = \hat{n}_s$ $\omega$-almost surely. Moreover, for each $s \in S$, there exist $\hat{v}_{s,1},\ldots,\hat{v}_{s,\hat{n}_s} \in V(\Gamma)$ such that we have
\[
\varphi_i(s) = \hat{g}_{i,s,1} \cdots \hat{g}_{i,s,\hat{n}_s}
\]
with $\hat{g}_{i,s,j} \in \GG[\st(\hat{v}_{s,j})]$ $\omega$-almost surely. But since $\GG[\st(v)] = G_v \times \GG[\lk(v)]$ for each $v \in V(\Gamma)$, we have $\hat{g}_{i,s,j} = g_{i,s,2j-1}g_{i,s,2j}$, where $g_{i,s,2j-1} \in G_{\hat{v}_{s,j}} \leq \GG[\lk(v_{s,2j-1})]$ with any choice of vertex $v_{s,2j-1} \in \lk(\hat{v}_{s,j})$ (which exists since $\Gamma$ is connected and $|V(\Gamma)| \geq 2$), and $g_{i,s,2j} \in \GG[\lk(v_{s,2j})]$ with $v_{s,2j} = \hat{v}_{s,j}$. It follows that, after setting $n_s = 2\hat{n}_s$, we $\omega$-almost surely may write
\[
\varphi_i(s) = g_{i,s,1} \cdots g_{i,s,n_s}
\]
with $g_{i,s,j} \in \GG[\lk(v_{s,j})]$.

Now for each $s \in S$, define abstract letters $h_{s,1},\ldots,h_{s,n_s}$. For each $v \in V(\Gamma)$, let
\[
S_v = \{ h_{s,j} \mid v_{s,j} = v \},
\]
let $\overline{S} = \sqcup_{v \in V(\Gamma)} S_v$, and let $\overline{F} = F(\overline{S})$, the free group on $\overline{S}$. We can define a map from $S$ to $F$ by sending $s \in S$ to $h_{s,1} \cdots h_{s,n_s}$. Let $N$ be the normal subgroup of $\overline{F}$ generated by images of all the relators of $F$ under this map. This gives a group homomorphism $\rho: F \to \overline{F}/N$.

The map $\widehat\varphi_i: \overline{F}/N \to \GG$, obtained by sending $h_{s,j}N$ to $g_{i,s,j}$, is $\omega$-almost surely a well-defined homomorphism. Indeed, all the relators in $\overline{F}/N$ are of the form $\phi(\{ h_{s,1} \cdots h_{s,n_s} \mid s \in S \})$, where $\phi(S)$ is a relator in $F$, and so $\omega$-almost surely map to the identity under $\widehat\varphi_i$. It is also clear that $\varphi_i = \widehat\varphi_i \circ \rho$ $\omega$-almost surely.

Now let $\pi: \overline{F} \to \overline{F}/N$ be the quotient map. Then, by construction, the homomorphisms $\varphi_i' = \widehat\varphi_i \circ \pi: \overline{F} \to \GG$ are linking (when they are well-defined). Since $\Gamma$ is admissible and the groups $G_v$ are equationally noetherian, it follows that $\varphi_i'$ factors through $\overline{F}_\omega$ $\omega$-almost surely. Since $\pi$ is surjective, this implies that $(\overline{F}/N)_\omega \subseteq \ker\hat\varphi_i$ $\omega$-almost surely. Thus $\varphi_i = \widehat\varphi_i \circ \rho$ factors through $F_\omega = \rho^{-1}((\overline{F}/N)_\omega)$ $\omega$-almost surely, as required.
\end{proof}

We expect that the class of equationally noetherian groups is closed under taking arbitrary graph products. Although we are not able to show this in full generality, in the next subsections we show that any connected graph $\Gamma$ that is triangle-free, square-free and pentagon-free, is admissible, and therefore (by Theorem \ref{t:engood}) the class of equationally noetherian groups is closed under taking graph products over such graphs $\Gamma$.

\subsection{Dual van Kampen diagrams} \label{ss:dvkds}

Before embarking on a proof of Theorem \ref{t:en}, let us define the following notion. Following methods of \cite{cw} and \cite{kk}, we consider \emph{dual van Kampen diagrams} for words representing the identity in $\GG$; recently, dual van Kampen diagrams for graph products have been independently introduced by Genevois in \cite{gendual}. Here we explain their construction and properties.

%As opposed to right-angled Artin groups, we do not in general have a finite presentation that is easy to work with, even if all the groups in $\mathcal{G}$ are finitely presented. Indeed, although such a presentation can still be chosen to contain relators of the form $[g_v,g_w]$ where $g_v$, $g_w$ are elements of the generating sets of $G_v$, $G_w$ (for some $v \sim w$), respectively,  there will also be relations between the generators of each of the $G_v$, which we have no control over for arbitrary finite presentations. Therefore, we choose to work with a generating set that in general will not be finite.
We consider van Kampen diagrams in the quasi-median graph $X$ given by Theorem \ref{t:gengp}, viewed as a Cayley graph.
In particular, note that we have a presentation
\begin{equation} \label{e:pres}
\GG = \langle S \mid R_\triangle \sqcup R_\square \rangle
\end{equation}
with generators
\[
S = \bigsqcup_{v \in V\Gamma} (G_v \setminus \{1\})
\]
and relators of two types: the `triangular' relators
\[
R_\triangle = \bigsqcup_{v \in V(\Gamma)} \{ ghk^{-1} \mid g,h,k \in G_v \setminus \{1\}, gh=k \text{ in } G_v \}
\]
and the `rectangular' relators
\[
R_\square = \bigsqcup_{(v,w) \in E(\Gamma)} \{ [g_v,g_w] \mid g_v \in G_v \setminus \{1\}, g_w \in G_w \setminus \{1\} \}.
\]

We now dualise the notion of van Kampen diagrams with respect to the presentation \eqref{e:pres}. Let $D \subseteq \R^2$ be a van Kampen diagram with boundary label $w$, for some word $w \in S^*$ representing the identity in $\GG$, with respect to the presentation \eqref{e:pres}. It is convenient to pick a colouring $V(\Gamma) \to \N$ and to colour edges of $D$ according to their labels. Suppose that $w = g_1 \cdots g_n$ for some syllables $g_i$, and let $e_1,\ldots,e_n$ be the corresponding edges on the boundary of $D$. We add a `vertex at infinity' $\infty$ somewhere on $\R^2 \setminus D$, and for each $i = 1,\ldots,n$, we attach to $D$ a triangular `boundary' face whose vertices are the endpoints of $e_i$ and $\infty$. We get the dual van Kampen diagram $\Delta$ corresponding to $D$ by taking the dual of $D$ as a polyhedral complex and removing the face corresponding to $\infty$: thus, $\Delta$ is a tesselation of a disk. See Figure \ref{f:dvkd}.

\begin{figure}[ht]
\begin{tikzpicture}[scale=0.63]
\draw [very thick,red,->] (0,0) -- (1,1) node [left] {$a_1$};
\draw [very thick,red] (1,1) -- (2,2);
\draw [very thick,red,->] (-1.4,1.4) -- (-0.4,2.4) node [left] {$a_1$};
\draw [very thick,red] (-0.4,2.4) -- (0.6,3.4);
\draw [very thick,red,->] (2,2) -- (2,1) node [right] {$a_2$};
\draw [very thick,red] (2,1) -- (2,0);
\draw [very thick,red,->] (4,2) -- (4,1) node [right] {$a_2$};
\draw [very thick,red] (4,1) -- (4,0);
\draw [very thick,red,->] (3.6,-3.4) -- (4.3,-2.7) node [left] {$a_3$};
\draw [very thick,red] (4.3,-2.7) -- (5,-2);
\draw [very thick,red,->] (5,-4.8) -- (5.7,-4.1) node [right] {$a_3$};
\draw [very thick,red] (5.7,-4.1) -- (6.4,-3.4);
\draw [very thick,red,->] (0,0) -- (1,0) node [below] {$a_4$};
\draw [very thick,red] (1,0) -- (2,0);
\draw [very thick,red,->] (0,-2) -- (1,-2) node [below] {$a_4$};
\draw [very thick,red] (1,-2) -- (2,-2);
\draw [very thick,blue,->] (0.6,3.4) -- (1.3,2.7) node [right] {$b_1$};
\draw [very thick,blue] (1.3,2.7) -- (2,2);
\draw [very thick,blue,->] (-1.4,1.4) -- (-0.7,0.7) node [left] {$b_1$};
\draw [very thick,blue] (-0.7,0.7) -- (0,0);
\draw [very thick,blue,->] (4,0) -- (4.5,-1) node [right] {$b_2$};
\draw [very thick,blue] (4.5,-1) -- (5,-2);
\draw [very thick,green!70!black,->] (2,2) -- (3,2) node [above] {$c_1$};
\draw [very thick,green!70!black] (3,2) -- (4,2);
\draw [very thick,green!70!black,->] (2,0) -- (3,0) node [above] {$c_1$};
\draw [very thick,green!70!black] (3,0) -- (4,0);
\draw [very thick,green!70!black,->] (5,-2) -- (5.7,-2.7) node [right] {$c_2$};
\draw [very thick,green!70!black] (5.7,-2.7) -- (6.4,-3.4);
\draw [very thick,green!70!black,->] (3.6,-3.4) -- (4.3,-4.1) node [below] {$c_2$};
\draw [very thick,green!70!black] (4.3,-4.1) -- (5,-4.8);
\draw [very thick,green!70!black,->] (2,-2) -- (3,-1) node [right] {$c_3$};
\draw [very thick,green!70!black] (3,-1) -- (4,0);
\draw [very thick,green!70!black,->] (0,0) -- (0,-1) node [left] {$c_4$};
\draw [very thick,green!70!black] (0,-1) -- (0,-2);
\draw [very thick,green!70!black,->] (2,0) -- (2,-1) node [left] {$c_4$};
\draw [very thick,green!70!black] (2,-1) -- (2,-2);
\fill (-1.4,1.4) circle (3pt);

\draw [very thick] (-1.4,1.4) -- (-2.4,1.4);
\draw [very thick,dotted] (-2.4,1.4) -- (-3,1.4);
\draw [thick] (0.6,3.4) -- (0.6,4.4);
\draw [thick,dotted] (0.6,4.4) -- (0.6,5);
\draw [thick] (2,2) -- (2.7,2.7);
\draw [thick,dotted] (2.7,2.7) -- (3.1,3.1);
\draw [thick] (4,2) -- (4.8,2.4);
\draw [thick,dotted] (4.8,2.4) -- (5.3,2.65);
\draw [thick] (4,0) -- (4.9,0.225);
\draw [thick,dotted] (4.9,0.225) -- (5.5,0.375);
\draw [thick] (5,-2) -- (6,-2);
\draw [thick,dotted] (6,-2) -- (6.6,-2);
\draw [thick] (6.4,-3.4) -- (7.2,-3.8);
\draw [thick,dotted] (7.2,-3.8) -- (7.7,-4.05);
\draw [thick] (5,-4.8) -- (5.4,-5.6);
\draw [thick,dotted] (5.4,-5.6) -- (5.65,-6.1);
\draw [thick] (3.6,-3.4) -- (3.6,-4.4);
\draw [thick,dotted] (3.6,-4.4) -- (3.6,-5);
\draw [thick] plot [smooth] coordinates { (4,0) (3.8,-1.2) (2.8,-2) (2.302,-2.996) (2.3,-3) };
\draw [thick,dotted] (2.3,-3) -- (2.05,-3.5);
\draw [thick] plot [smooth] coordinates { (5,-2) (3.8,-2.2) (3.2,-2.8) (3.001,-3.596) (3,-3.6)  };
\draw [thick,dotted] (3,-3.6) -- (2.85,-4.2);
\draw [thick] (2,-2) -- (1.3,-2.7);
\draw [thick,dotted] (1.3,-2.7) -- (0.9,-3.1);
\draw [thick] (0,-2) -- (-0.8,-2.4);
\draw [thick,dotted] (-0.8,-2.4) -- (-1.3,-2.65);
\draw [thick] (0,0) -- (-0.9,-0.225);
\draw [thick,dotted] (-0.9,-0.225) -- (-1.5,-0.375);
\end{tikzpicture}%
\hfill%
\begin{tikzpicture}[scale=0.63]
\fill [white] (0,-6.5) circle (0);
\draw [very thick,blue] plot [smooth] coordinates { (-3.412,-0.779) (-2.7,0.5) (-2.736,2.182) };
\draw [very thick,red] plot [smooth] coordinates { (-2,0) (-2.7,0.5) (-3.412,0.779) };
\draw [very thick,red] plot [smooth] coordinates { (-2,0) (-1.7,-1) (-1.519,-3.153) };
\draw [very thick,red] plot [smooth] coordinates { (-2,0) (-0.8,1.5) (0,3.5) };
\draw [very thick,green!70!black] plot [smooth] coordinates { (-2.736,-2.182) (-1.7,-1) (-0.5,-0.5) };
\draw [very thick,green!70!black] plot [smooth] coordinates { (-1.519,3.153) (-0.8,1.5) (-0.5,-0.5) };
\draw [very thick,green!70!black] plot [smooth] coordinates { (0,-3.5) (0,-2) (-0.5,-0.5) };
\draw [very thick,blue] plot [smooth] coordinates { (1.519,-3.153) (1,0) (1.519,3.153) };
\draw [very thick,red] plot [smooth] coordinates { (2.736,-2.182) (2.5,-1) (2.8,0) (3.412,0.779) };
\draw [very thick,green!70!black] plot [smooth] coordinates { (2.736,2.182) (2.5,1) (2.8,0) (3.412,-0.779) };

\fill (-2.7,0.5) circle (2pt);
\fill (-1.7,-1) circle (2pt);
\fill (-0.8,1.5) circle (2pt);
\fill (2.8,0) circle (2pt);

\draw [very thick] (0,0) circle (3.5);
\draw [very thick] (-3.8,0) -- (-3.2,0);
\draw [very thick,->] (-3.500,0.000) arc (180.00:167.14:3.5) node [left] {$a_1$};                                                       \draw [very thick,|->] (-3.153,1.519) arc (154.29:141.43:3.5) node [left] {$b_1$};                                                       \draw [very thick,|->] (-2.182,2.736) arc (128.57:115.71:3.5) node [above] {$c_1$};                                                       \draw [very thick,|->] (-0.779,3.412) arc (102.86:90.00:3.5) node [above] {$a_2$};                                                        \draw [very thick,|->] (0.779,3.412) arc (77.14:64.29:3.5) node [above] {$b_2$};                                                          \draw [very thick,|->] (2.182,2.736) arc (51.43:38.57:3.5) node [right] {$c_2$};                                                          \draw [very thick,|-<] (3.153,1.519) arc (25.71:12.86:3.5) node [right] {$a_3$};                                                          \draw [very thick,|-<] (3.500,0.000) arc (0.00:-12.86:3.5) node [right] {$c_2$};                                                              \draw [very thick,|->] (3.153,-1.519) arc (-25.71:-38.57:3.5) node [right] {$a_3$};                                                             \draw [very thick,|-<] (2.182,-2.736) arc (-51.43:-64.29:3.5) node [below] {$b_2$};                                                             \draw [very thick,|-<] (0.779,-3.412) arc (-77.14:-90.00:3.5) node [below] {$c_3$};                                                             \draw [very thick,|-<] (-0.779,-3.412) arc (-102.86:-115.71:3.5) node [below] {$a_4$};                                                          \draw [very thick,|-<] (-2.182,-2.736) arc (-128.57:-141.43:3.5) node [left] {$c_4$};                                                          \draw [very thick,|-<] (-3.153,-1.519) arc (-154.29:-167.14:3.5) node [left] {$b_1$};
\end{tikzpicture}
\caption{Van Kampen diagram ($D$, left) and its dual ($\Delta$, right) with the word $a_1 b_1 c_1 a_2 b_2 c_2 a_3^{-1} c_2^{-1} a_3 b_2^{-1} c_3^{-1} a_4^{-1} c_4^{-1} b_1^{-1}$ as its boundary label, where $a_i \in G_a$ with $a_1a_2 = a_4$, $b_i \in G_b$, $c_i \in G_c$ with $c_4c_3 = c_1$, and $b \sim a \sim c$ in $\Gamma$. The black edges on $D$ represent the boundary faces attached: the non-visible endpoint of each black edge is the point $\infty$. The dual van Kampen diagram $\Delta$ contains $6$ components in total: $2$ components corresponding to each of the vertices $a$, $b$ and $c$.} \label{f:dvkd}
\end{figure}
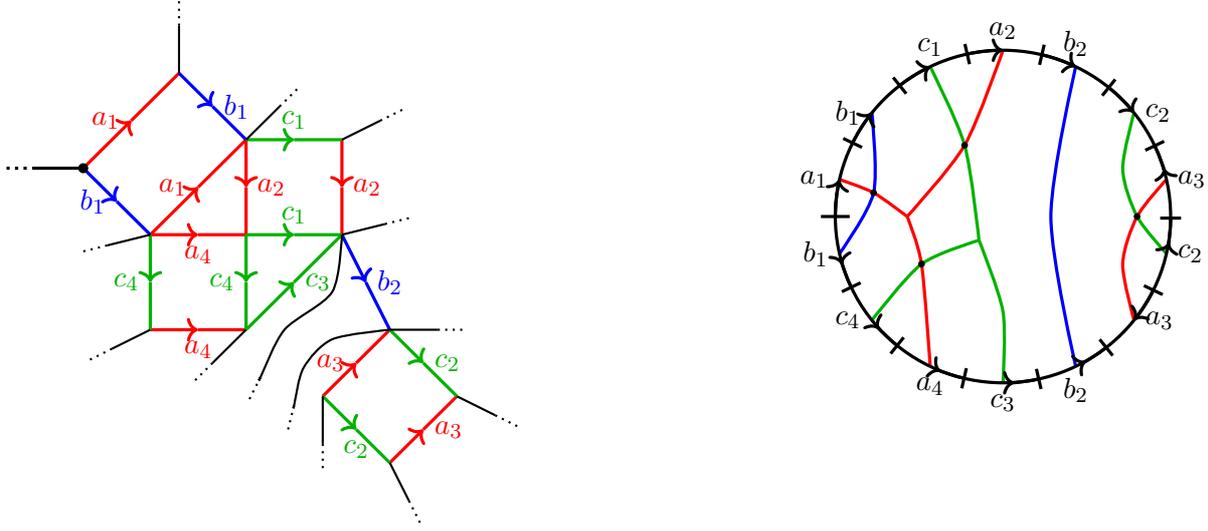

We lift the colouring of edges in $D$ to a colouring of edges of $\Delta$: this gives a corresponding vertex $v \in V(\Gamma)$ for each internal edge of $\Delta$. We say a $1$-subcomplex (a subgraph) of $\Delta$ is a \emph{$v$-component} (or just a \emph{component}) for some $v \in V(\Gamma)$ if it is a maximal connected subgraph each of whose edges correspond to the vertex $v$. We call a vertex of $\Delta$ an \emph{intersection point} (respectively \emph{branch point}, \emph{boundary point}) if it comes from a triangular (respectively rectangular, boundary) face in $D$. It is easy to see that boundary, intersection and branch points lying on a component $C$ will be precisely the vertices of $C$ of degree $1$, $2$ and $3$, respectively.

We now equip dual van Kampen diagrams with some additional structure. Some of the concepts introduced in the following Definition are displayed in Figure \ref{f:pvkd}. %Given a word $W$ over $\overline{S}$ such that $W \in \ker(\overline\varphi_i)$ for some $i$, we will consider a dual van Kampen diagram whose boundary label is a concatenation of geodesic words representing $\overline\varphi_i(p)$ for pieces $p$ in a polygonal representation of $W$. Therefore, our discussion of dual van Kampen diagrams is an attempt to reflect the discussion of (minimal) polygonal representations above.
\begin{defn} \label{d:polyg-vkd}
Let $\mathbf{W}$ be a cyclic word over $\bigsqcup_{v \in V(\Gamma)} (G_v \setminus \{1\})$ such that $\mathbf{W} = 1$ in $\Gamma\mathcal{G}$. A \emph{polygonal van Kampen diagram} for $\mathbf{W}$ is a dual van Kampen diagram with boundary label $\mathbf{W}$ together with the following information:
\begin{itemize}
\item a subdivision $\mathbf{W} = g_1 \cdots g_k$ of $\mathbf{W}$ into geodesic subwords $g_\ell$, called \emph{van Kampen pieces}, such that for each $\ell \in \{ 1,\ldots,k \}$ we have $\supp g_\ell \subseteq \link(v)$ for some $v \in V(\Gamma)$; %For $1 \leq l,l' \leq k$, pieces $p_l$ and $p_{l'}$ are said to be \emph{adjacent} if $k \geq 2$ and either $|l-l'| = 1$ or $\{ l,l' \} = \{ 1,k \}$.
\item an assignment of a \emph{label} from $V(\Gamma) \times \{ \dot{\ }, \widehat{\ } \}$ for each van Kampen piece, such that $\supp g_\ell \subseteq \{ v \}$ if $g_\ell$ has label $(v,\dot{\ })$ and $\supp g_\ell \subseteq \link(v)$ if $g_\ell$ has label $(v,\widehat{\ })$.
\end{itemize}
For any $v \in V(\Gamma)$, we write $\dot{v}$ for $(v,\dot{\ })$, and call each van Kampen piece with label $\dot{v}$ a $\dot{v}$-piece; we write $\widehat{v}$ for $(v,\widehat{\ })$, and call each van Kampen piece with label $\widehat{v}$ a $\widehat{v}$-piece.

We say a van Kampen piece $g$ is \emph{trivial} if $|g| = 0$, \emph{small} if $|g| = 1$, and \emph{large} otherwise. With $g_1,\ldots,g_k$ the van Kampen pieces as above, we say $g_{j_1},\ldots,g_{j_m}$ are \emph{consecutive} van Kampen pieces if $2 \leq m \leq k$ and $j_{i+1} \equiv j_i+1 \pmod{k}$ for $1 \leq i \leq m-1$. We say van Kampen pieces $g$ and $h$ are \emph{adjacent} if either $g,h$ or $h,g$ are consecutive.
%$g$ is the van Kampen piece \emph{preceding} a van Kampen piece $h$ if $g,h$ are consecutive, and $g$ is the van Kampen piece \emph{following} $h$ if $h,g$ are consecutive; in either case, we say that $g$ is \emph{adjacent} to $h$.

We adopt the convention that $\mathbf{W}$ is read out by following the boundary of $\Delta$ clockwise. For a van Kampen piece $g$ of $\Delta$, we write $\mathcal{I}_g(\Delta)$ for the subinterval of the boundary $\partial\Delta$ of $\Delta$ corresponding to the subword $g$. Given two distinct points $P,Q \in \partial\Delta$, we write $[P,Q]$ for the closed subinterval of $\partial\Delta$ going clockwise from $P$ to $Q$. We say a component $C$ of $\Delta$ is \emph{supported} at a van Kampen piece $g$ if it has a boundary point on $\mathcal{I}_g(\Delta)$. A \emph{supporting interval} of $C$ is an interval of the form $[P,Q]$, where $P,Q \in \partial\Delta$ are distinct boundary points of $C$ such that all boundary points of $C$ are contained in $[P,Q]$.

An \emph{oriented component} of $\Delta$ is a pair $(C,[P,Q])$, where $C$ is a component of $\Delta$ and $[P,Q]$ is a supporting interval of $C$. An oriented component $(C,[P,Q])$ is said to \emph{enclose} an oriented component $(C',[P',Q'])$ (or we may simply say that $(C,[P,Q])$ encloses $C'$) if $[P',Q'] \subsetneq [P,Q]$. %, and $(C,[P,Q])$ is said to \emph{partially enclose} $(C',[P',Q'])$ if $[P,Q] \nsubseteq [P',Q']$ and $[P',Q'] \nsubseteq [P,Q]$.
We say that $(C,[P,Q])$ is \emph{trivial} if $P \in \mathcal{I}_g(\Delta)$, $Q \in \mathcal{I}_h(\Delta)$ and $g,h$ are consecutive van Kampen pieces; $(C,[P,Q])$ is said to be \emph{minimal} if it is not trivial and does not enclose any non-trivial components, and \emph{almost minimal} if it is neither trivial nor minimal and does not enclose any non-trivial non-minimal components. We say $(C,[P,Q])$ is an oriented component \emph{starting} (respectively \emph{ending}) at a van Kampen piece $g$ if $P \in \mathcal{I}_g(\Delta)$ (respectively $Q \in \mathcal{I}_g(\Delta)$).

%Let $C$ be a component of a polygonal van Kampen diagram $\Delta$. If the boundary points of $C$ are on the subintervals of the boundary of $\Delta$ corresponding to van Kampen pieces $g_{j_1},\ldots,g_{j_\ell}$, we say that $\{ j_1,\ldots,j_\ell \}$ is the \emph{support} of $C$, denoted $\supp(C)$, and that $C$ is \emph{supported} at $j_l$ for any $l \in \{ 1,\ldots,\ell \}$. If $\supp(C) = \{ j_1,\ldots,j_\ell \}$ with $1 \leq j_1 < \cdots < j_\ell \leq k$, then a \emph{supporting interval} for $C$ is either the tuple $( j_1, j_1+1, \ldots, j_\ell )$ or $( j_m, j_m+1, \ldots, k, 1, 2, \ldots, j_{m-1} )$ for some $m \in \{ 2,\ldots,k \}$. An \emph{oriented component} is a pair $(C,I)$ where $C$ is a component of $\Delta$ and $I$ is its supporting interval, and an oriented component $(C,I)$ is said to \emph{enclose} an oriented component $(C',I')$ if $I'$ is a proper subsequence of $I$.
\end{defn}

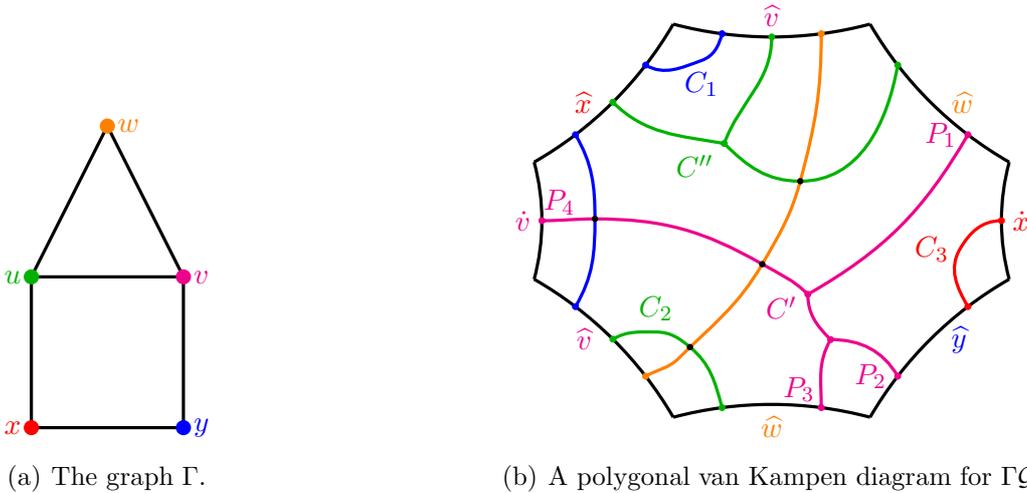
\begin{figure}[ht]
\begin{subfigure}[b]{0.3\textwidth}
\centering
\begin{tikzpicture}
\draw[very thick] (0,2) -- (0,0) -- (2,0) -- (2,2) -- (0,2) -- (1,4) -- (2,2);
\fill[red] (0,0) circle (0.1) node[left] {$x$};
\fill[blue] (2,0) circle (0.1) node[right] {$y$};
\fill[green!70!black] (0,2) circle (0.1) node[left] {$u$};
\fill[magenta] (2,2) circle (0.1) node[right] {$v$};
\fill[orange] (1,4) circle (0.1) node[right] {$w$};
\end{tikzpicture}
\caption{The graph $\Gamma$.}
\end{subfigure}
\hfill
\begin{subfigure}[b]{0.6\textwidth}
\centering
\begin{tikzpicture}
\draw[very thick] (0,0) arc (-15:15:3) node[midway,left,color=magenta] {$\dot{v}$} node[at end] (e1) {} node[midway] (p4) {};
\draw[very thick] (e1) arc (-60:-30:5) node[midway,left,color=red] {$\widehat{x}\ $} node[at end] (e2) {} node[near start] (d22) {} node[midway] (cpp1) {} node[near end] (c11) {};
\draw[very thick] (e2) arc (-105:-75:5) node[midway,above,color=magenta] {$\widehat{v}$} node[at end] (e3) {} node[near start] (c12) {} node[midway] (cpp2) {} node[near end] (d11) {};
\draw[very thick] (e3) arc (-150:-120:5) node[midway,right,color=orange] {$\ \widehat{w}$} node[at end] (e4) {} node[near start] (cpp3) {} node[near end] (p1) {};
\draw[very thick] (e4) arc (165:195:3) node[midway,right,color=red] {$\dot{x}$} node[at end] (e5) {} node[midway] (c31) {};
\draw[very thick] (e5) arc (120:150:5) node[midway,right,color=blue] {$\ \widehat{y}$} node[at end] (e6) {} node[near start] (c32) {} node[near end] (p2) {};
\draw[very thick] (e6) arc (75:105:5) node[midway,below,color=orange] {$\widehat{w}$} node[at end] (e7) {} node[near start] (p3) {} node[near end] (c21) {};
\draw[very thick] (e7) arc (30:60:5) node[midway,left,color=magenta] {$\widehat{v}\ $} node[near start] (d12) {} node[midway] (c22) {} node[near end] (d21) {};
\draw[very thick,blue] (c11.center) circle (0.7pt) to[out=-30,in=-160] (2.2,2.9) node[below] {$C_1$} to[out=20,in=-105] (c12.center) circle (0.7pt);
\draw[very thick,green!70!black] (c22.center) circle (0.7pt) to[out=30,in=180] (1.6,-0.7) node[above] {$C_2$} to[out=0,in=140] (2.05,-0.9) node (ic2d1) {} to[out=-40,in=105] (c21.center) circle (0.7pt);
\draw[very thick,red] (c31.center) circle (0.7pt) to[out=-180,in=60] (5.6,0.4) node[left] {$C_3$} to[out=-120,in=120] (c32.center) circle (0.7pt);
\draw[very thick,orange] (d12.center) circle (0.7pt) to[out=30,in=-135] (ic2d1) to[out=45,in=-120] (3,0.2) node (icpd1) {} to[out=60,in=-105] (3.5,1.3) node (icppd1) {} to[out=75,in=-90] (d11.center) circle (0.7pt);
\draw[very thick,blue] (d21.center) circle (0.7pt) to[out=60,in=-90] (0.8,0.8) node (icpd2) {} to[out=90,in=-60] (d22.center) circle (0.7pt);
\draw[very thick,green!70!black] (cpp1.center) circle (0.7pt) to[out=-45,in=170] (2.5,1.8) node (bcpp) {} circle (0.7pt) node[below left] {$C''$};
\draw[very thick,green!70!black] (cpp2.center) circle (0.7pt) to[out=-90,in=70] (bcpp.center) to[out=-45,in=-180] (icppd1.center) to[out=0,in=-100] (cpp3.center) circle (0.7pt);
\draw[very thick,magenta] (p4.center) circle (0.7pt) node[right,yshift=7pt] {$\!\!P_4$} to[out=0,in=180] (icpd2.center) to[out=0,in=150] (icpd1.center) to[out=-30,in=135] (3.6,-0.2) circle (0.7pt) node (bcp1) {} node[left,yshift=-5pt] {$C'$};
\draw[very thick,magenta] (p1.center) circle (0.7pt) node[left] {$P_1$} to[out=-120,in=30] (bcp1.center) to[out=-90,in=135] (3.9,-0.8) circle (0.7pt) node (bcp2) {};
\draw[very thick,magenta] (p3.center) circle (0.7pt) node[left,yshift=7pt] {$P_3\!$} to[out=90,in=-120] (bcp2.center) to[out=0,in=120] (p2.center) circle (0.7pt) node[left] {$P_2$};
\fill (icpd1) circle (1.2pt);
\fill (icppd1) circle (1.2pt);
\fill (ic2d1) circle (1.2pt);
\fill (icpd2) circle (1.2pt);
\end{tikzpicture}
\caption{A polygonal van Kampen diagram for $\Gamma\mathcal{G}$.}
\end{subfigure}

\caption{An example of a polygonal van Kampen diagram $\Delta$. The interval $\mathcal{I}_g(\Delta)$ corresponding to a van Kampen piece $g$ is drawn as a smooth curve, with the label of $g$ written next to it. The trivial components are $C_1$, $C_2$ and $C_3$ (with the obvious choices of supporting intervals). The oriented component $(C',[P_1,P_4])$ is minimal, but for $1 \leq i \leq 3$, the oriented component $(C',[P_{i+1},P_i])$ is not minimal as it encloses $C''$. The components $C_3$ and $C'$ are the first (penultimate) and second (last) components, respectively, supported at the unique $\widehat{y}$-piece of $\Delta$.%, and the $\widehat{y}$-piece itself is following an $\dot{x}$-piece and preceding a $\widehat{w}$-piece.
}
\label{f:pvkd}
\end{figure}

We order the components supported at a given van Kampen piece $g$ by looking at their boundary points on $\mathcal{I}_g(\Delta)$. In particular, for $i \geq 1$ and a van Kampen piece $g$, we say a component $C$ supported at $g$ is the \emph{$i$-th component supported at $g$} if the boundary point of $C$ in $\mathcal{I}_g(\Delta)$ is the $i$-th boundary point on $\mathcal{I}_g(\Delta)$ when following $\partial\Delta$ clockwise. We similarly define the \emph{last} (or \emph{penultimate}) \emph{component supported at $g$}. %If $C$ and $C'$ 

%An example of a polygonal van Kampen diagram is shown in Figure \ref{f:pvkd}.

The following two Lemmas will be used in Section \ref{ss:girth}.

\begin{lem} \label{l:nointer}
Let $\Delta$ be a polygonal van Kampen diagram, and suppose that $\Gamma$ is triangle-free. Then for any van Kampen piece $g$ of $\Delta$, no two components supported at $g$ intersect.
\end{lem}

\begin{proof}
Let $C$ be a $v$-component and $C'$ be a $w$-component (for some $v,w \in V(\Gamma)$), both supported at a van Kampen piece $g$. Suppose for contradition that $C$ and $C'$ intersect: then $w \in \link(v)$. But we have $\{ v,w \} \subseteq \supp g \subseteq \link(u)$ for some $u \in V(\Gamma)$. Hence $u,v,w$ span a triangle in $\Gamma$, contradicting the fact that $\Gamma$ is triangle-free. Thus $C$ and $C'$ cannot intersect, as required.
\end{proof}

%{\color{red} Define large piece, small/large van Kampen piece, consecutive (van Kampen) pieces, ($i$-th) supported at (van Kampen) piece, (adjacent components?), $\mathcal{I}_p(\Delta)$, piece preceding/following another one, $g_p$.}

\begin{lem} \label{l:interhat}
Let $\Delta$ be a polygonal van Kampen diagram, and suppose that $\Gamma$ has girth $\geq 5$. Suppose $C_1$ and $C_2$ are distinct components of $\Delta$ supported at a van Kampen piece $g$ of $\Delta$, and $C$ is a $v$-component (for some $v \in V(\Gamma)$) that either intersects or coincides with each of $C_1$ and $C_2$. Then $g$ is a $\widehat{v}$-piece, and $C$ intersects both $C_1$ and $C_2$.
\end{lem}

\begin{proof}
Let $C_1$ and $C_2$ be the $i$-th and the $j$-th pieces (respectively) supported at $g$, and suppose, without loss of generality, that $i < j$.

Suppose first that $j \geq i+2$. Let $P_1,P_2 \in \mathcal{I}_p(\Delta)$ be the boundary points of $C_1,C_2$, respectively. Note that, by the assumption, any continuous path in $\Delta$ between a point in $[P_1,P_2] \setminus \{P_1,P_2\}$ and a point in $[P_2,P_1] \setminus \{P_1,P_2\}$ must intersect either $C$ or $C_1$ or $C_2$. Thus, if $C'$ is the $(i+1)$-st component supported at $p$, then $C'$ must either intersect or coincide with one of $C$, $C_1$ and $C_2$. But $C'$ cannot intersect $C_1$ or $C_2$ by Lemma \ref{l:nointer}, and it cannot coincide with either $C_1$ or $C_2$ since $g$ is a geodesic. Thus $C'$ must either intersect or coincide with $C$; by replacing $C_2$ with $C'$ if necessary, we may therefore assume, without loss of generality, that $j = i+1$.

Now let $v_1,v_2 \in V(\Gamma)$ be such that $C_1$ and $C_2$ are a $v_1$-component and a $v_2$-component, respectively. As $g$ is a geodesic and $j = i+1$, it follows that $v_1 \neq v_2$. This implies that we cannot have $C_1 = C = C_2$; on the other hand, if one of the $C_i$ intersected $C$ and the other one was equal to $C$, then $C_1$ and $C_2$ would intersect, contradicting Lemma \ref{l:nointer}. Thus both $C_1$ and $C_2$ must intersect $C$, and so $v_1,v_2 \in \link(v)$. As $\Gamma$ is square-free and $v_1 \neq v_2$, this implies that $g$ must be a $\widehat{v}$-piece, as required.
\end{proof}

\subsection{Minimal polygonal representations} \label{ss:pr}

Throughout this and the next subsections, let $\Gamma$ be a connected finite simple graph with at least two vertices, and let $\mathcal{G} = \{ G_v \mid v \in V(\Gamma) \}$ be a collection of groups. Let $F = \langle S_0 \mid \rangle$ be a finitely generated free group (so that $|S_0| < \infty$), let $S := S_0 \sqcup S_0^{-1} \subset F$, and let $\varphi_i: F \to \Gamma\mathcal{G}$ be a sequence of homomorphisms. For each $v \in V(\Gamma)$, let $S_v \subseteq S$ be the set of all $s \in S$ such that $\supp \varphi_i(s) \subseteq \link(v)$ $\omega$-almost surely. Suppose that the homomorphisms $\varphi_i$ are linking, and in particular, $S = \bigcup_{v \in V(\Gamma)} S_v$.

Now for each $v \in V(\Gamma)$, define abstract letters $T_v^{(0)} = \{ t_v(s) \mid s \in S_0 \}$. Let $\overline{F}$ be a free group with free basis $\overline{S}_0 := S_0 \sqcup \left( \bigsqcup_{v \in V(\Gamma)} T_v^{(0)} \right)$, so that $F$ is a free factor of $\overline{F}$. We may extend the homomorphisms $\varphi_i: F \to \Gamma\mathcal{G}$ to homomorphisms $\overline\varphi_i: \overline{F} \to \Gamma\mathcal{G}$ by setting $\overline\varphi_i(t_v(s)) = \pi_v(\varphi_i(s))$, where $\pi_v: \Gamma\mathcal{G} \to G_v$ is the canonical projection. Let $T_v := T_v^{(0)} \sqcup (T_v^{(0)})^{-1} \subset \overline{F}$ for $v \in V(\Gamma)$, and let $\overline{S} := \overline{S}_0 \sqcup \overline{S}_0^{-1} \subset \overline{F}$.

For each $v \in V(\Gamma)$, let $\dot{F}_v$ and $\widehat{F}_v$ be the subgroups of $\overline{F}$ generated by $T_v$ and $\widehat{S}_v := S_v \sqcup (\bigsqcup_{w \in \link(v)} T_w)$, respectively. Note that $\overline\varphi_i(\dot{F}_v) \subseteq G_v$ and $\overline\varphi_i(\widehat{F}_v) \subseteq \Gamma_{\link(v)} \mathcal{G}_{\link(v)}$ $\omega$-almost surely. For convenience, for any $v,w \in V(\Gamma)$ and $s \in S_0$ we also set $t_v(t_w(s)) = \begin{cases} t_w(s) & \text{if } v = w, \\ 1 & \text{otherwise}, \end{cases}$ so that we have well-defined functions $t_v: \overline{S}_0 \to T_v^{(0)} \cup \{1\}$ extending to homomorpisms $t_v: \overline{F} \to \dot{F}_v$.

The aim of this and the next subsections is to introduce several technical lemmas we will use in a proof of the following result.

\begin{prop} \label{p:girth6}
If $\Gamma$ has girth $\geq 6$, then $\overline{F}_{\omega,(\overline\varphi_i)}$ is the normal closure $K$ of
\[
\mathcal{K} = \left( \bigcup_{v \in V(\Gamma)} \left( \overline{F}_{\omega,(\overline\varphi_i)} \cap \widehat{F}_v \right) \right) \cup \left( \bigcup_{v \in V(\Gamma)} \left[ \dot{F}_v, \widehat{F}_v \right] \right)
\]
in $\overline{F}$.
\end{prop}

By construction, $\left[ \dot{F}_v, \widehat{F}_v \right] \subseteq \overline{F}_\omega$ $\omega$-almost surely for all $v \in V(\Gamma)$, and so $\mathcal{K} \subseteq \overline{F}_\omega$. Thus, we only need to show that $K$ contains every $f \in \overline{F}$ such that $\overline\varphi_i(f) = 1$ $\omega$-almost surely. We will use the remainder of this subsection to prove this. In order to do that, we work with `polygonal representations', which are defined below and closely related to polygonal van Kampen diagrams introduced in Section \ref{ss:dvkds}.

\begin{defn} \label{d:polyg-repn}
Let $W = s_1 \cdots s_n$ be a cyclic word over the alphabet $\overline{S}$. % such that $\overline\varphi_i(w) = 1$ $\omega$-almost surely. 
A \emph{polygonal representation} for $W$ consists of the following data:
\begin{itemize}
\item A subdivision $W = | s_{j_0} \cdots s_{j_1-1} | s_{j_1} \cdots s_{j_2-1} | \cdots | s_{j_{k-1}} \cdots s_{j_k-1} |$ of $W$, where $j_k = j_0$, into subwords $p_\ell = s_{j_{\ell-1}} \cdots s_{j_\ell-1}$ called \emph{pieces}, such that for each $\ell \in \{ 1,\ldots,k \}$ we have $s_{j_{\ell-1}},\ldots,s_{j_\ell-1} \in \widehat{S}_v$ for some $v \in V(\Gamma)$.
\item An assignment of a \emph{label} from $V(\Gamma) \times \{ \dot{\ }, \widehat{\ } \}$ for each piece. For any $v \in V(\Gamma)$, we write $\dot{v}$ for $(v,\dot{\ })$, and call each piece with label $\dot{v}$ a $\dot{v}$-piece; we write $\widehat{v}$ for $(v,\widehat{\ })$, and call each piece with label $\widehat{v}$ a $\widehat{v}$-piece. We require the assignment to be such that for any piece $p_\ell = s_{j_{\ell-1}} \cdots s_{j_\ell-1}$, we have $s_{j_{\ell-1}},\ldots,s_{j_\ell-1} \in T_v$ if $p_\ell$ is a $\dot{v}$-piece, and $s_{j_{\ell-1}},\ldots,s_{j_\ell-1} \in \widehat{S}_v$ if $p_l$ is a $\widehat{v}$-piece.
%\item For each $i$ such that $\overline\varphi_i(w) = 1$, a dual van Kampen diagram with boundary label $q_{i,1} \cdots q_{i,k}$ in $\Gamma\mathcal{G}$, where $q_{i,l}$ is a geodesic word representing $\overline\varphi_i(p_l)$.
\end{itemize}
The \emph{length} of a piece $p_\ell$ is the number $\lim^\omega |\overline\varphi_i(p_\ell)| \in \mathbb{N} \cup \{\infty\}$, where $|g|$ is the word length of $g \in \Gamma\mathcal{G}$ over the alphabet $\bigsqcup_{v \in V(\Gamma)} (G_v \setminus \{1\})$. Clearly, any $\dot{v}$-piece has length $\leq 1$ for any $v \in V(\Gamma)$. If a piece $p_\ell$ has length $L < \infty$, then (as $\Gamma$ is finite) there exist vertices $v_1,\ldots,v_L \in V(\Gamma)$ such that we $\omega$-almost surely have $\overline\varphi_i(p_\ell) = g_{i,1} \cdots g_{i,L}$ for some $g_{i,m} \in G_{v_m}$; if in addition we have $v_m \neq v_{m'}$ for $1 \leq m < m' \leq L$, then $p_\ell$ is said to be a \emph{small} piece. A piece that is not small is called \emph{large}.

Given a finite sequence $p_{j_1},\ldots,p_{j_m}$ of pieces, we say $p_{j_1},\ldots,p_{j_m}$ are \emph{consecutive} if $2 \leq m \leq k$ and $j_{i+1} \equiv j_i+1 \pmod{k}$ for $1 \leq i \leq m-1$. We say a piece $p$ is \emph{adjacent} to a piece $q$ if either $p,q$ or $q,p$ are consecutive.
\end{defn}

It is clear that every word has at least one polygonal representation: for instance, the one where pieces are single letters.

We say two words $U,U'$ over $\overline{S}$ are \emph{equivalent} if $UK$ and $U'K$ are conjugate in $\overline{F}/K$, and we say two cyclic words $W,W'$ over $\overline{S}$ are \emph{equivalent} if some (equivalently, any) representative words $U$ of $W$ and $U'$ of $W'$ are equivalent. It is clear that this defines an equivalence relation on the set of cyclic words over $\overline{S}$.

Given a cyclic word $W'$ with a polygonal representation $\mathcal{D}'$, we denote by $\dot{P}(W',\mathcal{D}')$ and $\widehat{P}(W',\mathcal{D}')$ the number of $\dot{v}$-pieces (for $v \in V(\Gamma)$) and the number of $\widehat{v}$-pieces (for $v \in V(\Gamma)$), respectively, in $\mathcal{D}'$. A polygonal representation $\mathcal{D}$ for a word $W$ is said to be \emph{minimal} if, for all cyclic words $W'$ equivalent to $W$ and all polygonal representations $\mathcal{D}'$ for $W'$, we have $(\widehat{P}(W,\mathcal{D}),\dot{P}(W,\mathcal{D})) \leq (\widehat{P}(W',\mathcal{D}'),\dot{P}(W',\mathcal{D}'))$ in the lexicographical order.

\begin{rmk}
There is an a priori noticeable discrepancy between the definition of a `small / large piece' in Definition \ref{d:polyg-repn} and the definition of a `trivial / small / large van Kampen piece' in Definition \ref{d:polyg-vkd}. However, in view of parts \ref{i:lred-0} and \ref{i:lred-1} of Lemma \ref{l:red} below, the two definitions are related in the following sense.

Let $W$ be a cyclic word over $\overline{S}$, let $\mathcal{D}$ be a polygonal representation for $W$ with pieces $p_1,\ldots,p_k$, and suppose that $\mathcal{D}$ is a minimal polygonal representation. Suppose moreover that $W \in \overline{F}_\omega$: in particular, $\omega$-almost surely there exists a polygonal van Kampen diagram $\Delta_i$ with van Kampen pieces $g_{i,1},\ldots,g_{i,k}$, where $g_{i,\ell}$ is a geodesic word representing $\overline\varphi_i(p_\ell)$ for each $\ell$, and where the label of the van Kampen piece $g_{i,\ell}$ coincides with the label of the piece $p_\ell$. Then $p_\ell$ is a small (respectively large) piece if and only if $g_{i,\ell}$ is $\omega$-almost surely a small (respectively large) van Kampen piece.

Since in the subsequent argument we do not consider polygonal representations that are not minimal, we chose our terminology the way we did in Definitions \ref{d:polyg-vkd} and \ref{d:polyg-repn}.
\end{rmk}

\begin{lem} \label{l:red}
A minimal polygonal representation satisfies the following:
\begin{enumerate}[label=\textup{(\roman*)}]
\item \label{i:lred-0} there are no pieces of length $0$;
\item \label{i:lred-1} every small piece is a $\dot{v}$-piece for some $v \in V(\Gamma)$ (in particular, any small piece has length $1$, and there are no pieces of length $2$);
%\item \label{i:lred-2} there are no pieces of length $2$;
\item \label{i:lred-adperm} if $p_1,\ldots,p_m$ are consecutive pieces, $p_1$ and $p_m$ are a $\widehat{w}_1$-piece and a $\dot{v}$-piece (respectively) for some $w_1,v \in V(\Gamma)$, and, for each $j \in \{ 2,\ldots,m \}$, $p_j$ is a $\dot{w}_j$-piece for some $w_j \in V(\Gamma)$, then we have $\{ w_1,w_2,\ldots,w_{m-1} \} \nsubseteq \link(v)$;
\item \label{i:lred-adperm2} if $p_1,\ldots,p_m$ are consecutive pieces, $p_1$ and $p_m$ are $\dot{v}$-pieces for some $v \in V(\Gamma)$, and, for each $j \in \{ 2,\ldots,m-1 \}$, $p_j$ is a $\dot{w}_j$-piece for some $w_j \in V(\Gamma)$, then $\{ w_2,\ldots,w_{m-1} \} \nsubseteq \link(v)$;
\item \label{i:lred-adsame} no two pieces with the same label are adjacent;
\item \label{i:lred-adad} for any $v \in V(\Gamma)$ and $w \in \link(v)$, no $\dot{w}$-piece is adjacent to a $\widehat{v}$-piece.
\item \label{i:lred-adperm3} for any $v \in V(\Gamma)$ and $w \in \link(v)$, no $\widehat{v}$-piece is adjacent to both a $\widehat{w}$-piece and a $\dot{v}$-piece.
\item \label{i:lred-adperm4} for any $v \in V(\Gamma)$, no $\widehat{v}$-piece is adjacent to two $\dot{v}$-pieces.
\item \label{i:lred-adperm5} for any $v \in V(\Gamma)$, no $\dot{v}$-piece is adjacent to two $\widehat{v}$-pieces.
\end{enumerate}
\end{lem}

{
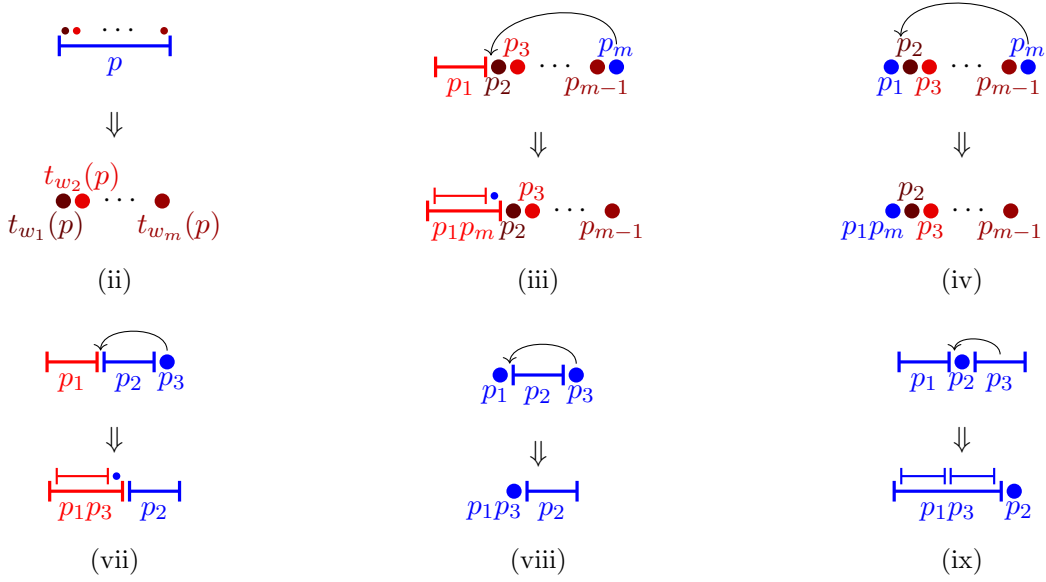
\begin{figure}[ht]
\captionsetup[subfigure]{labelformat=empty}
\begin{subfigure}[b]{0.3\textwidth}
\centering
\begin{tikzpicture}
\draw[very thick,blue,|-|] (0,0) -- (1.5,0) node[midway,below] {$p$};
\fill[red!40!black] (0.1,0.2) circle (0.05);
\fill[red!90!black] (0.25,0.2) circle (0.05);
\node at (0.825,0.2) {$\cdots$};
\fill[red!60!black] (1.4,0.2) circle (0.05);
\end{tikzpicture}
\[\Downarrow\]
\begin{tikzpicture}
\fill[red!40!black] (0.1,0.0) circle (0.1) node[below] {$\!\!\!\!\!\!\!t_{w_1}(p)$};
\fill[red!90!black] (0.35,0.0) circle (0.1) node[above] {$t_{w_2}(p)$};
\node at (0.875,0.0) {$\cdots$};
\fill[red!60!black] (1.4,0.0) circle (0.1) node[below] {$t_{w_m}(p)\!\!\!\!\!\!\!$};
\end{tikzpicture}
\caption{\ref{i:lred-1}}
\end{subfigure}
\hfill
\begin{subfigure}[b]{0.3\textwidth}
\centering
\begin{tikzpicture}
\draw[very thick,red,|-|] (0,0) -- (0.7,0) node[midway,below] {$p_1$};
\fill[red!40!black] (0.85,0.0) circle (0.1) node[below] {$p_2$};
\fill[red!90!black] (1.1,0.0) circle (0.1) node[above] {$p_3$};
\node at (1.625,0.0) {$\cdots$};
\fill[red!60!black] (2.15,0.0) circle (0.1) node[below] {$p_{m-1}$};
\fill[blue] (2.4,0.0) circle (0.1) node[above] {$p_m$};
\draw[->] (2.4,0.3) to[out=90,in=90] (0.75,0.15);
\end{tikzpicture}
\[\Downarrow\]
\begin{tikzpicture}
\draw[very thick,red,|-|] (-0.3,0) -- (0.7,0) node[midway,below] {$p_1p_m$};
\fill[red!40!black] (0.85,0.0) circle (0.1) node[below] {$p_2$};
\fill[red!90!black] (1.1,0.0) circle (0.1) node[above] {$p_3$};
\node at (1.625,0.0) {$\cdots$};
\fill[red!60!black] (2.15,0.0) circle (0.1) node[below] {$p_{m-1}$};
\draw[thick,red,|-|] (-0.2,0.2) -- (0.5,0.2);
\fill[blue] (0.6,0.2) circle (0.05);
\end{tikzpicture}
\caption{\ref{i:lred-adperm}}
\end{subfigure}
\hfill
\begin{subfigure}[b]{0.3\textwidth}
\centering
\begin{tikzpicture}
\fill[blue] (0.6,0.0) circle (0.1) node[below] {$p_1$};
\fill[red!40!black] (0.85,0.0) circle (0.1) node[above] {$p_2$};
\fill[red!90!black] (1.1,0.0) circle (0.1) node[below] {$p_3$};
\node at (1.625,0.0) {$\cdots$};
\fill[red!60!black] (2.15,0.0) circle (0.1) node[below] {$p_{m-1}$};
\fill[blue] (2.4,0.0) circle (0.1) node[above] {$p_m$};
\draw[->] (2.4,0.3) to[out=90,in=90] (0.725,0.4);
\end{tikzpicture}
\[\Downarrow\]
\begin{tikzpicture}
\fill[blue] (0.6,0.0) circle (0.1) node[below] {$\!\!\!\!\!\!\!\!p_1p_m$};
\fill[red!40!black] (0.85,0.0) circle (0.1) node[above] {$p_2$};
\fill[red!90!black] (1.1,0.0) circle (0.1) node[below] {$p_3$};
\node at (1.625,0.0) {$\cdots$};
\fill[red!60!black] (2.15,0.0) circle (0.1) node[below] {$p_{m-1}$};
\end{tikzpicture}
\caption{\ref{i:lred-adperm2}}
\end{subfigure}

\bigskip

\begin{subfigure}[b]{0.3\textwidth}
\centering
\begin{tikzpicture}
\draw[very thick,red,|-|] (0,0) -- (0.7,0) node[midway,below] {$p_1$};
\draw[very thick,blue,|-|] (0.75,0) -- (1.45,0) node[midway,below] {$p_2$};
\fill[blue] (1.6,0.0) circle (0.1) node[below] {$p_3\!\!$};
\draw[->] (1.6,0.15) to[out=90,in=90] (0.725,0.15);
\end{tikzpicture}
\[\Downarrow\]
\begin{tikzpicture}
\draw[very thick,red,|-|] (-0.3,0) -- (0.7,0) node[midway,below] {$p_1p_3$};
\draw[very thick,blue,|-|] (0.75,0) -- (1.45,0) node[midway,below] {$p_2$};
\draw[thick,red,|-|] (-0.2,0.2) -- (0.5,0.2);
\fill[blue] (0.6,0.2) circle (0.05);
\end{tikzpicture}
\caption{\ref{i:lred-adperm3}}
\end{subfigure}
\hfill
\begin{subfigure}[b]{0.3\textwidth}
\centering
\begin{tikzpicture}
\fill[blue] (0.6,0.0) circle (0.1) node[below] {$\!\!p_1$};
\draw[very thick,blue,|-|] (0.75,0) -- (1.45,0) node[midway,below] {$p_2$};
\fill[blue] (1.6,0.0) circle (0.1) node[below] {$p_3\!\!$};
\draw[->] (1.6,0.15) to[out=90,in=90] (0.725,0.15);
\end{tikzpicture}
\[\Downarrow\]
\begin{tikzpicture}
\fill[blue] (0.6,0.0) circle (0.1) node[below] {$\!\!\!\!\!\!\!\!p_1p_3$};
\draw[very thick,blue,|-|] (0.75,0) -- (1.45,0) node[midway,below] {$p_2$};
\end{tikzpicture}
\caption{\ref{i:lred-adperm4}}
\end{subfigure}
\hfill
\begin{subfigure}[b]{0.3\textwidth}
\centering
\begin{tikzpicture}
\draw[very thick,blue,|-|] (-0.25,0) -- (0.45,0) node[midway,below] {$p_1$};\fill[blue] (0.6,0.0) circle (0.1) node[below] {$p_2$};
\draw[very thick,blue,|-|] (0.75,0) -- (1.45,0) node[midway,below] {$p_3$};
\draw[->] (1.1,0.1) to[out=90,in=90] (0.5,0.15);
\end{tikzpicture}
\[\Downarrow\]
\begin{tikzpicture}
\draw[very thick,blue,|-|] (-1.0,0) -- (0.45,0) node[midway,below] {$p_1p_3$};
\fill[blue] (0.6,0.0) circle (0.1) node[below] {$p_2\!\!$};
\draw[thick,blue,|-|] (-0.9,0.2) -- (-0.3,0.2);
\draw[thick,blue,|-|] (-0.25,0.2) -- (0.35,0.2);
\end{tikzpicture}
\caption{\ref{i:lred-adperm5}}
\end{subfigure}
\caption[Proof of Lemma \ref{l:red}]{Proof of Lemma \ref{l:red}. The colours represent different vertices, $v$ being blue and $w$ or $w_i$ reddish. The $\dot{u}$-pieces and $\widehat{u}$-pieces (for some $u \in V(\Gamma)$) are denoted by
\begin{tikzpicture}[baseline=-0.1cm]
\fill (0.0,0.0) circle (0.1);
\end{tikzpicture}
and
\begin{tikzpicture}[baseline=-0.1cm]
\draw[very thick,|-|] (0,0) -- (0.6,0);
\end{tikzpicture}%
, respectively.
}
\label{f:lred}
\end{figure}
}

\begin{proof}
Let $\mathcal{D}$ be a minimal polygonal representation for a word $W$. We claim that $\mathcal{D}$ satisfies the conclusions of the Lemma. We check the parts \ref{i:lred-0}--\ref{i:lred-adperm5} in order. Some of these parts are visualised in Figure \ref{f:lred}.

\begin{enumerate}[label=(\roman*)]

\item Note that if $p$ is a piece of length $0$, then $\overline\varphi_i(p) = 1$ $\omega$-almost surely, and so $p \in \overline{F}_\omega \cap \widehat{F}_u$, where $u = v$ if $p$ is a $\widehat{v}$-piece and $u$ is any vertex in $\link(v)$ if $p$ is a $\dot{v}$-piece (such a vertex $u$ exists since $\Gamma$ is connected). In particular, $p \in \mathcal{K} \subseteq K$, and so deleting the subword $p$ of $W$ and the piece $p$ from $\mathcal{D}$ results in a new word $W'$, equivalent to $W$, and a polygonal representation $\mathcal{D}'$ for $W'$. But then we either have $(\widehat{P}(W',\mathcal{D}'),\dot{P}(W',\mathcal{D}')) = (\widehat{P}(W,\mathcal{D}),\dot{P}(W,\mathcal{D})-1)$ or $(\widehat{P}(W',\mathcal{D}'),\dot{P}(W',\mathcal{D}')) = (\widehat{P}(W,\mathcal{D})-1,\dot{P}(W,\mathcal{D}))$. This contradicts minimality of $\mathcal{D}$.

\item Suppose that $p$ is a small $\widehat{v}$-piece for some $v \in V(\Gamma)$. Then there exist distinct vertices $w_1,\ldots,w_m \in \link(v)$ such that we have $\overline\varphi_i(p) = g_{i,1} \cdots g_{i,m}$ $\omega$-almost surely, where $g_{i,\ell} \in G_{w_\ell}$. In particular, we have $(t_{w_1}(p) \cdots t_{w_m}(p))^{-1} p \in \overline{F}_\omega \cap \widehat{F}_v \subseteq K$. Thus, replacing the subword $p$ of $W$ with a subword $t_{w_1}(p) \cdots t_{w_m}(p)$ results in a new word $W'$, equivalent to $W$, and replacing the $\widehat{v}$-piece $p$ in $\mathcal{D}$ with $m$ pieces $t_{w_1}(p),\ldots,t_{w_m}(p)$ (with labels $\dot{w}_1,\ldots,\dot{w}_m$, respectively) we obtain a polygonal representation $\mathcal{D}'$ for $W'$. However, then $\mathcal{D}'$ satisfies $(\widehat{P}(W',\mathcal{D}'),\dot{P}(W',\mathcal{D}')) = (\widehat{P}(W,\mathcal{D})-1,\dot{P}(W,\mathcal{D})+\ell)$, contradicting minimality of $\mathcal{D}$.

\item Suppose for contradiction that this is not the case. We then have $[p_m,p_2 \cdots p_{m-1}] \in \left[ \dot{F}_v, \widehat{F}_v \right] \subseteq K$, and so multiplying (a conjugate of) the word $W = p_1 \cdots p_m W_0$ by $[p_m,p_2 \cdots p_{m-1}]$ we get a word $W' = p_1 p_m p_2 \cdots p_{m-1} W_0$ equivalent to $W$. We define a polygonal representation $\mathcal{D}'$ for $W'$ by letting the pieces be $p_1 p_m, p_2, \ldots, p_{m-1}$ together with the pieces of $W_0$, where $p_1 p_m$ is a $\widehat{w_1}$-piece and all the other pieces have the same labels as the corresponding pieces in $\mathcal{D}$. But then $(\widehat{P}(W',\mathcal{D}'),\dot{P}(W',\mathcal{D}')) = (\widehat{P}(W,\mathcal{D}),\dot{P}(W,\mathcal{D})-1)$, again contradicting minimality of $\mathcal{D}$.

\item Suppose for contradiction that this is not the case. We construct a word $W'$ equivalent to $W$ and a polygonal representation $\mathcal{D}'$ for $W$ the same way as in part \ref{i:lred-adperm}, except for letting $p_1 p_m$ be a $\dot{v}$-piece this time. Then again $(\widehat{P}(W',\mathcal{D}'),\dot{P}(W',\mathcal{D}')) = (\widehat{P}(W,\mathcal{D}),\dot{P}(W,\mathcal{D})-1)$, contradicting minimality of $\mathcal{D}$.

\item If two adjacent pieces $p$ and $q$ have the same label, then we can concatenate them (resulting in a single piece with the same label) to obtain a new polygonal representation $\mathcal{D}'$ of $W$. But then we have $(\widehat{P}(W,\mathcal{D}'),\dot{P}(W,\mathcal{D}')) = (\widehat{P}(W,\mathcal{D}),\dot{P}(W,\mathcal{D})-1)$ if $p$ and $q$ are $\dot{v}$-pieces, and $(\widehat{P}(W,\mathcal{D}'),\dot{P}(W,\mathcal{D}')) = (\widehat{P}(W,\mathcal{D})-1,\dot{P}(W,\mathcal{D}))$ if $p$ and $q$ are $\widehat{v}$-pieces (for some $v \in V(\Gamma)$). This contradicts minimality of $\mathcal{D}$.

\item The case when $p,q$ are consecutive pieces such that $p$ is a $\widehat{v}$-piece and $q$ is a $\dot{w}$-piece (for some $w \in \link(v)$) is exactly the case $m = 2$ in part \ref{i:lred-adperm}. If instead $q,p$ are consecutive, a symmetric argument works.
%If $p_l$ is a $v$-piece and $p_{l+1}$ is a $w$-piece for some $l$ (indices taken modulo $k$), then this follows directly from part \ref{i:lred-adperm} by taking $l' = l+1$ and $w_{l'} = w$. Otherwise -- if $p_l$ is a $v$-piece and $p_{l-1}$ is a $w$-piece -- a symmetric argument works. %Similarly, if $v \in V(\Gamma)$ and $w \in \link(v)$, then any two adjacent pieces with labels $\widehat{v}$ and $\dot{w}$ can be concatenated (resulting in a single $\widehat{v}$-piece) to obtain a new polygonal representation $\mathcal{D}'$ of $W$ with $(\widehat{P}(W,\mathcal{D}'),\dot{P}(W,\mathcal{D}')) = (\widehat{P}(W,\mathcal{D}),\dot{P}(W,\mathcal{D})-1)$, again contradicting minimality of $\mathcal{D}$. \qedhere

\item Suppose that $p_1,p_2,p_3$ are consecutive pieces such that $p_1$, $p_2$ and $p_3$ have labels $\widehat{w}$, $\widehat{v}$ and $\dot{v}$, respectively; if instead $p_1$, $p_2$ and $p_3$ have labels $\dot{v}$, $\widehat{v}$ and $\widehat{w}$, respectively, then a symmetric argument works.  We then have $[p_3,p_2] \in \left[ \dot{F}_v, \widehat{F}_v \right] \subseteq K$, and so multiplying (a conjugate of) $W = p_1 p_2 p_3 W_0$ by $[p_3,p_2]$ we get a word $W' = p_1 p_3 p_2 W_0$ equivalent to $W$. We define a polygonal representation $\mathcal{D}'$ for $W'$ by letting the pieces be $p_1 p_3$ and $p_2$ together with the pieces of $W_0$, where $p_1 p_3$ is a $\widehat{w}$-piece and all the other pieces have the same labels as the corresponding pieces in $\mathcal{D}$. But then $(\widehat{P}(W',\mathcal{D}'),\dot{P}(W',\mathcal{D}')) = (\widehat{P}(W,\mathcal{D}),\dot{P}(W,\mathcal{D})-1)$, again contradicting minimality of $\mathcal{D}$.

\item This follows exactly as in part \ref{i:lred-adperm3}, except that the labels of the piece $p_1$ of $\mathcal{D}$ and the piece $p_1 p_3$ of $\mathcal{D}'$ are now both $\dot{v}$ instead of $\widehat{w}$.

\item Suppose that $p_1,p_2,p_3$ are consecutive pieces with labels $\widehat{v}$, $\dot{v}$ and $\widehat{v}$, respectively.  We then have $[p_2,p_3] \in \left[ \dot{F}_v, \widehat{F}_v \right] \subseteq K$, and so multiplying a conjugate of $W = p_1p_2p_3W_0$ by $[p_2,p_3]^{-1}$ we get a word $W' = p_1 p_3 p_2 W_0$ equivalent to $W$. We define a polygonal representation $\mathcal{D}'$ for $W'$ by letting the pieces be $p_1 p_3$ and $p_2$ together with the pieces of $W_0$, where $p_1 p_3$ is a $\widehat{v}$-piece and all the other pieces have the same labels as the corresponding pieces in $\mathcal{D}$. But then $(\widehat{P}(W',\mathcal{D}'),\dot{P}(W',\mathcal{D}')) = (\widehat{P}(W,\mathcal{D})-1,\dot{P}(W,\mathcal{D}))$, again contradicting minimality of $\mathcal{D}$. \qedhere

\end{enumerate}

\end{proof}

\subsection{Graphs of large girth} \label{ss:girth}

The idea of our proof of Proposition \ref{p:girth6} is to use properties of polygonal van Kampen diagrams analogous to the ones enjoyed by minimal polygonal representations. In particular, we define the following.

\begin{defn}
A polygonal van Kampen diagram is said to be \emph{minimal} if $\Delta$ satisfies the conditions \ref{i:lred-0}--\ref{i:lred-adperm5} in Lemma \ref{l:red}, with each appearance of the word `piece' replaced with `van Kampen piece'.
\end{defn}

In particular, we will refer to Lemma \ref{l:red} when talking about minimal polygonal van Kampen diagrams. Therefore, in the following discussion we will use the fact that minimal polygonal van Kampen diagrams have no trivial pieces (by Lemma \ref{l:red}\ref{i:lred-0}), and so all their van Kampen pieces are either small or large. The following results (Lemma \ref{l:mtsupplarge}, Lemma \ref{l:noam} and Corollary \ref{c:nothing}) are the main auxiliary results required in our proof of Proposition \ref{p:girth6}.

\begin{lem} \label{l:mtsupplarge}
Let $\Delta$ be a minimal polygonal van Kampen diagram, and suppose that $\Gamma$ has girth $\geq 6$. If $(C,[P,Q])$ is a component of $\Delta$ that is either trivial or minimal, then the following hold:
\begin{enumerate}[label=\textup{(\roman*)}]
\item \label{i:lmtsl-nencl} $(C,[P,Q])$ does not enclose any other components;
\item \label{i:lmtsl-large} if $C$ is supported at a van Kampen piece $g$, then $g$ is large;
\item \label{i:lmtsl-first} if $C$ is supported at a van Kampen piece $g$, then either $P \in \mathcal{I}_g(\Delta)$ and $C$ is the last component supported at $g$, or $Q \in \mathcal{I}_g(\Delta)$ and $C$ is the first component supported at $g$;
\end{enumerate}
%If a component $C$ of $\Delta$ that is either trivial or minimal is supported at a van Kampen piece $g$, then $g$ is a large piece and $C$ is either the first or the last component supported at $p$. Moreover, no minimal component encloses a trivial component, and every minimal component is supported at precisely two pieces.
\end{lem}

\begin{figure}[ht]
\begin{subfigure}[b]{0.48\textwidth}
\centering
\begin{tikzpicture}
\draw (0,0) arc (-10:10:3) node (a) {};
\draw[very thick,{|[width=5]}-{|[width=5]}] (a) arc(120:90:4) node[near start] (c1) {}  node[pos=0.65,above,yshift=-2pt] {$g$} node[pos=0.32] (cp1) {} node[at end] (b) {};
\draw[very thick,-{|[width=5]}] (b) arc(90:60:4) node[near start] (cp2) {} node[near end] (c2) {} node[midway,above] {$h$} node[at end] (c) {};
\draw (c) arc (170:190:3);
\draw[blue,very thick] (c1.center) circle (0.5pt) node[below] {$\!\!\!P$} to[out=-60,in=180] (1.9,0.6) node[below,yshift=1.5pt] {$C$} to[out=0,in=-120] (c2.center) circle (0.5pt) node[below] {$Q\!\!\!$};
\draw[red,very thick] (cp1.center) circle (0.5pt) node[above,yshift=-3pt] {$\!\!\!P'$} to[out=-60,in=180] (1.5,1.1) node[below,yshift=3pt] {$C'$} to[out=0,in=-120] (cp2.center) circle (0.5pt) node[below] {$Q'\!\!\!$};
\end{tikzpicture}
\caption{$(C,[P,Q])$ trivial, part \ref{i:lmtsl-nencl}.}
\label{f:lmtsl-triv}
\end{subfigure}
\hfill
\begin{subfigure}[b]{0.48\textwidth}
\centering
\begin{tikzpicture}
\draw (0,0) arc (-10:10:3) node[very near start] (c1) {} node[at end] (a) {};
\draw[very thick,{|[width=5]}-{|[width=5]}] (a) arc(-60:-30:2) node[midway,left] {$g_1$} node[near end] (c1) {} node[very near end] (cp1) {} node[at end] (b) {};
\draw[very thick,-{|[width=5]}] (b) arc(150:100:1.2) node[very near start] (cp2) {} node[near start] (cpp1) {} node[midway,above] {$g_2$} node[at end] (c) {};
\draw (c) arc (100:30:1.2) node (d) {};
\draw[very thick,{|[width=5]}-{|[width=5]}] (d) arc(-150:-120:2) node[very near start] (c2) {} node[midway,right] {$\,g_\ell$} node[at end] (e) {};
\draw (e) arc (170:190:3);
\draw[blue,very thick] (c1.center) circle (0.5pt) to[out=-60,in=180] (1.5,1) node[below,yshift=1.5pt] {$C$} -- (1.9,1) node (cpp2) {} to[out=0,in=-120] (c2.center) circle (0.5pt);
\draw[red,very thick] (cp1.center) circle (0.5pt) to[out=0,in=-120] (0.9,1.65) node[right] {$\!C'$} to[out=60,in=-60] (cp2.center) circle (0.5pt);
\draw[green!70!black,very thick] (cpp1.center) circle (0.5pt) to[out=-20,in=120] (1.7,1.6) node[right] {$C_+$} to[out=-60,in=90] (cpp2.center) circle (0.5pt);
\draw[green!70!black,very thick,dotted] (cpp2.center) -- (1.9,0.5);
\draw[fill=white] (cpp2) circle (0.1);
\end{tikzpicture}
\caption{$(C,[P,Q])$ minimal, part \ref{i:lmtsl-nencl}, $j = 1$.}
\label{f:lmtsl-mj1}
\end{subfigure}

\begin{subfigure}[b]{0.48\textwidth}
\centering
\begin{tikzpicture}
\draw (0,0) arc (-10:10:3) node[very near start] (c1) {} node[at end] (a) {};
\draw[very thick,{|[width=5]}-{|[width=5]}] (a) arc(-60:-30:2) node[midway,left] {$g_1$} node[very near end] (c1) {} node[at end] (a1) {};
\draw (a1) arc (-80:-30:0.5) node (b) {};
\draw[very thick,{|[width=5]}-{|[width=5]}] (b) arc(120:90:2) node[very near start] (c11) {} node[near start] (c21) {} node[midway,above] {$g_j$} node[very near end] (cp1) {} node[at end] (c) {};
\draw[very thick,-{|[width=5]}] (c) arc(90:60:2) node[very near start] (cp2) {} node[near start] (cpp1) {} node[midway,above] {$g_{j+1}$} node[at end] (d0) {};
\draw (d0) arc (-150:-100:0.5) node (d) {};
\draw[very thick,{|[width=5]}-{|[width=5]}] (d) arc(-150:-120:2) node[very near start] (c2) {} node[midway,right] {$\,g_\ell$} node[at end] (e) {};
\draw (e) arc (170:190:3);
\draw[blue,very thick] (c1.center) circle (0.5pt) to[out=-60,in=180] (1.5,1) node (c12) {} -- (1.8,1) node (c22) {} -- (2.2,1) node[below,yshift=1.5pt] {$C$} -- (2.6,1) node (cpp2) {} to[out=0,in=-120] (c2.center) circle (0.5pt);
\draw[red,very thick] (cp1.center) circle (0.5pt) to[out=-60,in=-180] (2.08,2.05) node[below] {$C'$} to[out=0,in=-120] (cp2.center) circle (0.5pt);
\draw[green!60!black,very thick] (c11.center) circle (0.5pt) to[out=-70,in=90] (c12.center) circle (0.5pt);
\draw[green!80!black,very thick] (c21.center) circle (0.5pt) to[out=-60,in=90] (c22.center) circle (0.5pt);
\draw[green!70!black,very thick] (cpp1.center) circle (0.5pt) to[out=-60,in=90] (2.6,1.6) node[right] {$\!C_+$} -- (cpp2.center) circle (0.5pt);
\draw[green!60!black,very thick,dotted] (c12.center) -- (1.5,0.5);
\draw[green!80!black,very thick,dotted] (c22.center) -- (1.8,0.5);
\draw[green!70!black,very thick,dotted] (cpp2.center) -- (2.6,0.5);
\draw[fill=white] (c12) circle (0.1);
\draw[fill=white] (c22) circle (0.1);
\draw[fill=white] (cpp2) circle (0.1);
\end{tikzpicture}
\caption{$(C,[P,Q])$ minimal, part \ref{i:lmtsl-nencl}, $j > 1$.}
\label{f:lmtsl-mj2}
\end{subfigure}
\hfill
\begin{subfigure}[b]{0.48\textwidth}
\centering
\begin{tikzpicture}
\draw (0,0) arc (-10:10:3) node[very near start] (c1) {} node[at end] (a) {};
\draw[very thick,{|[width=5]}-{|[width=5]}] (a) arc(-60:-30:2) node[midway,left] {$g_1$} node[very near end] (c1) {} node[at end] (a1) {};
\draw (a1) arc (-80:-30:0.5) node (b) {};
\draw[very thick,{|[width=5]}-{|[width=5]}] (b) arc(120:90:2) node[near end] (c3) {} node[midway,above] {$g_{j'}$} node[at end] (b1) {};
\draw (b1) arc (90:68:2) node (c) {};
\draw[very thick,{|[width=5]}-{|[width=5]}] (c) arc(68:60:2) node[midway] (c4) {} node[midway,above] {$g_j$} node[at end] (d0) {};
\draw (d0) arc (-150:-100:0.5) node (d) {};
\draw[very thick,{|[width=5]}-{|[width=5]}] (d) arc(-150:-120:2) node[very near start] (c2) {} node[midway,right] {$\,g_\ell$} node[at end] (e) {};
\draw (e) arc (170:190:3);
\draw[blue,very thick] (c1.center) circle (0.5pt) to[out=-60,in=180] (1.8,1) node (c5) {} -- (2.2,1) node[below,yshift=1.5pt] {$C$} -- (2.6,1) node (c6) {} to[out=0,in=-120] (c2.center) circle (0.5pt);
\draw[blue,very thick] (c3.center) circle (0.5pt) to[out=-100,in=90] (c5.center);
\draw[blue,very thick] (c4.center) circle (0.5pt) to[out=-120,in=90] (c6.center);
\end{tikzpicture}
\caption{$(C,[P,Q])$ minimal, part \ref{i:lmtsl-large}.}
\label{f:lmtsl-m1c}
\end{subfigure}

\caption[Proof of Lemma \ref{l:mtsupplarge}.]{Proof of Lemma \ref{l:mtsupplarge}. 
\begin{tikzpicture}[baseline=-0.1cm]
\draw[blue,very thick] (-0.25,-0.25) node[left] {$C_0$} -- (0.25,0.25);
\draw[green!70!black,very thick] (-0.25,0.25) node[left] {$C'_0$} -- (0,0);
\draw[green!70!black,very thick,dotted] (0,0) -- (0.25,-0.25);
\draw[fill=white] (0,0) circle (0.1);
\end{tikzpicture}
denotes that a component $C'_0$ either intersects or coincides with a component $C_0$.
}
\end{figure}
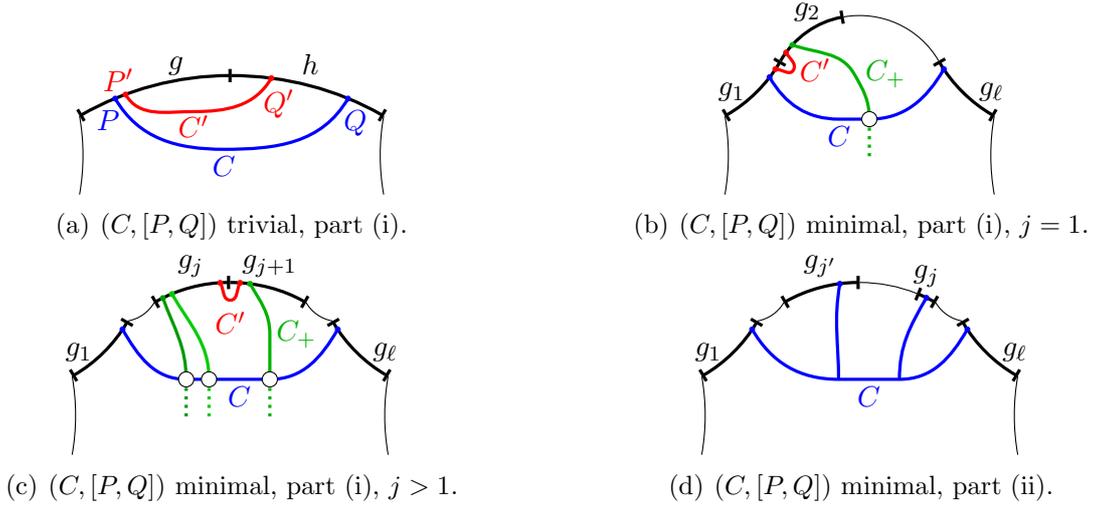

\begin{proof}
Let $(C,[P,Q])$ be an oriented $v$-component. We will show properties \ref{i:lmtsl-nencl}--\ref{i:lmtsl-first} when $(C,[P,Q])$ is trivial (see \ref{i:lmtsl-nencl-t}--\ref{i:lmtsl-first-t} below) and when it is minimal (see \ref{i:lmtsl-nencl-m}--\ref{i:lmtsl-first-m} below) separately.

\begin{description}

\item[If ${(C,[P,Q])}$ is trivial] Then there are two consecutive van Kampen pieces $g,h$ such that $P \in \mathcal{I}_g(\Delta)$ and $Q \in \mathcal{I}_h(\Delta)$. %We will show parts \ref{i:lmtsl-nencl}--\ref{i:lmtsl-first} in order.

\begin{enumerate}[label=(T.\roman*)]

\item \label{i:lmtsl-nencl-t} Let $i \geq 1$ be such that $C$ is the $i$-th component supported at $g$. Suppose that $(C,[P,Q])$ encloses a component $(C',[P',Q'])$. Then we must have $P' \in \mathcal{I}_g(\Delta)$, $Q' \in \mathcal{I}_h(\Delta)$, and $C'$ must be the $j$-th component supported at $g$ for some $j > i$; see Figure \ref{f:lmtsl-triv}. But then, whenever $C''$ is the $j'$-th component supported at $g$, where $i+1 \leq j' \leq j$, we know by Lemma \ref{l:nointer} that $C$ and $C''$ do not intersect, and so $(C,[P,Q])$ encloses $C''$. In particular, we may assume, without loss of generality, that $j = i+1$.

Now let $v' \in V(\Gamma)$ be the label of $C'$. As $g$ is a geodesic, it follows that $v' \neq v$. As $\Gamma$ is square-free, it follows that $|\link(v) \cap \link(v')| \leq 1$. Therefore, as both $C$ and $C'$ are supported at both $g$ and $h$, it follows that $\link(v) \cap \link(v') = \{w\}$ and that $g$ and $h$ are both $\widehat{w}$-pieces. But this contradicts Lemma \ref{l:red}\ref{i:lred-adsame}. Thus $(C,[P,Q])$ cannot enclose any other component, as required.

\item \label{i:lmtsl-large-t} If $g$ (or $h$) is small then it is a $\dot{v}$-component by Lemma \ref{l:red}\ref{i:lred-1}, and if it is large then it is a $\widehat{w}$-component for some $w \in \link(v)$. Therefore, if $g$ and $h$ were both small, then they would be two adjacent $\dot{v}$-pieces, contradicting Lemma \ref{l:red}\ref{i:lred-adsame}; if one of them was small and the other one large, then they would be a $\dot{v}$-piece adjacent to a $\widehat{w}$-piece for some $w \in \link(v)$, contradicting Lemma \ref{l:red}\ref{i:lred-adad}. Thus both $g$ and $h$ must be large, as required.

\item \label{i:lmtsl-first-t} As $g$ and $h$ are adjacent, it is clear that $C$ is not supported at van Kampen pieces other than $g$ and $h$. If $C$ is not the first component supported at $h$, then let $C'$ be the first component supported at $h$. By Lemma \ref{l:nointer}, $C$ and $C'$ do not intersect, and so $(C,[P,Q])$ must enclose $C'$ -- but this contradicts part \ref{i:lmtsl-nencl-t}. Thus $C$ must be the first component supported at $h$; similarly, $C$ must be the last component supported at $g$.

\end{enumerate}

\item[If ${(C,[P,Q])}$ is minimal] Let $g_1,\ldots,g_\ell$ be consecutive van Kampen pieces of $\Delta$ such that $P \in \mathcal{I}_{g_1}(\Delta)$ and $Q \in \mathcal{I}_{g_\ell}(\Delta)$. %We will show parts \ref{i:lmtsl-nencl}--\ref{i:lmtsl-first} for $(C,[P,Q])$ in order, by applying the results of parts \ref{i:lmtsl-nencl}--\ref{i:lmtsl-first} to trivial components if necessary.

\begin{enumerate}[label=(M.\roman*)]

\item \label{i:lmtsl-nencl-m} Suppose that $(C,[P,Q])$ encloses a $v'$-component $(C',[P',Q'])$ for some $v' \in V(\Gamma)$. As $(C,[P,Q])$ is minimal, it follows that $(C',[P',Q'])$ is trivial, and so $P' \in \mathcal{I}_{g_j}(\Delta)$ and $Q' \in \mathcal{I}_{g_{j+1}}(\Delta)$ for some $j \in \{ 1,\ldots,\ell-1 \}$. Without loss of generality, assume that $C'$ is chosen in such a way that $j$ is as small as possible. By part \ref{i:lmtsl-large-t} applied to $(C',[P',Q'])$, $g_{j+1}$ is a large van Kampen piece, and so there are at least $3$ components supported at $g_{j+1}$. Therefore, the second component $C_+$ supported at $g_{j+1}$ -- say a $v_+$-component for some $v_+ \in V(\Gamma)$ -- is not trivial by part \ref{i:lmtsl-first-t}, and so, as $(C,[P,Q])$ is minimal, $C_+$ must either intersect $C$ or coincide with $C$. As $g_{j+1}$ is a geodesic and as $C'$, $C_+$ are the first and the second components supported at $g_{j+1}$ (respectively), we also have $v_+ \neq v'$. Moreover, as $g_{j+1}$ is large, it must be a $\widehat{u}$-piece for some $u \in V(\Gamma)$ by Lemma \ref{l:red}\ref{i:lred-1}.

Suppose first that $j = 1$; see Figure \ref{f:lmtsl-mj1}. Then the components $C$ and $C'$ are supported at $g_1$, and so $g_1$ is a large van Kampen piece; therefore, by Lemma \ref{l:red}\ref{i:lred-1}, $g_1$ is a $\widehat{w}$-piece for some $w \in \link(v) \cap \link(v')$. Now if we had $C = C_+$, then we would have $v = v_+ \neq v'$, and so, as $C$ and $C'$ are supported at both $g_1$ and $g_2$, both $g_1$ and $g_2$ would have to be $\widehat{u}$-pieces for the \emph{unique} (as $\Gamma$ is square-free) vertex $u \in \link(v) \cap \link(v')$, contradicting Lemma \ref{l:red}\ref{i:lred-adsame}. On the other hand, if $C_+$ intersected $C$, then $v,w,v',u,v_+,v$ would be a closed walk on $\Gamma$ of length $5$, contradicting the fact that $\Gamma$ is triangle-free and pentagon-free. Thus we arrive at a contradiction in either case, and so we cannot have $j = 1$.

Suppose now that $j > 1$: then $\mathcal{I}_{g_j}(\Delta) \subseteq [P,Q]$, and so, by minimality of $(C,[P,Q])$, $C$ must either intersect or coincide with every non-trivial component supported at $g_j$. Moreover, $g_j$ must be a large piece by part \ref{i:lmtsl-large-t}, and so there must be at least $3$ components supported at $g_j$ by Lemma \ref{l:red}\ref{i:lred-1}. In particular, by part \ref{i:lmtsl-first-t} and the minimality of $j$, the first and the second components supported at $g_j$ are non-trivial, and so they must each either intersect or coincide with $C$; see Figure \ref{f:lmtsl-mj2}. It follows from Lemma \ref{l:interhat} that $g_j$ is a $\widehat{v}$-piece, and in particular $v' \in \link(v)$. But now if we had $C_+ = C$, then $v,v',u$ would span a triangle in $\Gamma$, contradicting the fact that $\Gamma$ is triangle-free. On the other hand, if $C_+$ intersected $C$, then -- as $v_+ \neq v'$ and $C'$, $C_+$ are both supported at $g_{j+1}$, and as $\Gamma$ is square-free -- it would follow that $g_{j+1}$ must be a $\widehat{v}$-piece, as well as $g_j$, contradicting Lemma \ref{l:red}\ref{i:lred-adsame}. Thus we cannot have $j \neq 1$ either.

\item \label{i:lmtsl-large-m} Suppose that $C$ is supported at a small van Kampen piece $g_j$, and so $g_j$ is a $\dot{v}$-piece. Then either $j > 1$ or $j < \ell$; without loss of generality, assume that $j > 1$. Let $j' < j$ be largest such that $C$ is supported at $j'$, so that $j' \geq 1$; see Figure \ref{f:lmtsl-m1c}. It follows from part \ref{i:lmtsl-nencl-m} that $(C,[P,Q])$ does not enclose any components. In particular, for each $m \in \{ j'+1,\ldots,j-1 \}$, $C$ intersects every component supported at $m$, and so $g_m$ is either a $\widehat{v}$-piece or a $\dot{w}_m$-piece for some $w_m \in \link(v)$.

Now if $g_m$ is a $\widehat{v}$-piece for some $m \in \{ j'+1,\ldots,j-1 \}$, then we must have $j = j'+2$ and $m = j'+1$: indeed, otherwise one of $g_{m-1}$ and $g_{m+1}$ is either a $\widehat{v}$-piece or a $\dot{w}$-piece for some $w \in \link(v)$, contradicting either Lemma \ref{l:red}\ref{i:lred-adsame} or Lemma \ref{l:red}\ref{i:lred-adad}. But then $g_{j'+1}$ is a $\widehat{v}$-piece adjacent to a $\dot{v}$-piece $g_{j'+2}$ and to $g_{j'}$, which is either a $\dot{v}$-piece or a $\widehat{w}$-piece for some $w \in \link(v)$ (as $C$ is supported at $j'$); this contradicts either Lemma \ref{l:red}\ref{i:lred-adperm3} or Lemma \ref{l:red}\ref{i:lred-adperm4}. On the other hand, if for all $m \in \{ j'+1,\ldots,j-1 \}$, $g_m$ is a $\dot{w}_m$-piece for some $w_m \in \link(v)$, then, as before, $g_{j'}$ is either a $\dot{v}$-piece or a $\widehat{w}$-piece for some $w \in \link(v)$, contradicting either Lemma \ref{l:red}\ref{i:lred-adperm} or Lemma \ref{l:red}\ref{i:lred-adperm2}. Thus in either case, $g_j$ cannot be small, as required.

\item \label{i:lmtsl-first-m} Suppose that $C$ is supported at $g_j$ for some $j \in \{1,\ldots,\ell\}$. If $C$ is not the first component supported at $g_j$, then the first component $C'$ supported at $g_j$ does not intersect $C$ by Lemma \ref{l:nointer}; but $(C,[P,Q])$ does not enclose $C'$ by part \ref{i:lmtsl-nencl-m}, implying that $j = 1$. Similarly, if $C$ is not the last component supported at $g_j$, then $j = \ell$. But by part \ref{i:lmtsl-large-m}, the van Kampen piece $g_j$ is large, and so $C$ cannot be both the first and the last component supported at $g_j$. As $1 \neq \ell$, it follows that either $j = 1$ and $C$ is the last component supported at $g_j$, or $j = \ell$ and $C$ is the first component supported at $g_j$, as required. \qedhere

\end{enumerate}

\end{description}
\end{proof}

\begin{lem} \label{l:noam}
%Let $\mathcal{D}$ be a minimal polygonal representation for a word $W$ over $\overline{S}$ such that $W \in \ker^\omega(\overline\varphi_i)$. If $\Gamma$ has girth $\geq 6$ and if $i \in \mathcal{A}(\mathcal{D},W)$, then $\Delta_i(W,\mathcal{D})$ has no almost minimal components.
Let $\Delta$ be a minimal polygonal van Kampen diagram, and suppose that $\Gamma$ has girth $\geq 6$. Then $\Delta$ has no almost minimal components.
\end{lem}

\begin{figure}[ht]
\begin{subfigure}[b]{0.35\textwidth}
\centering
\begin{tikzpicture}
\draw (0,0) arc (-10:10:5) node[very near start] (c1) {} node[at end] (a) {};
\draw[very thick,{|[width=5]}-{|[width=5]}] (a) arc(-60:-30:2) node[midway,left] {$g'$} node[very near end] (cp1) {} node[at end] (b) {};
\draw (b) arc (-100:-80:5) node (c) {};
\draw[very thick,{|[width=5]}-{|[width=5]}] (c) arc(-150:-120:2) node[very near start] (cp2) {} node[near start] (cpp1) {} node[midway,right] {$\,h'$} node[at end] (d) {};
\draw (d) arc (170:190:5) node[very near end] (c2) {};
\draw[blue,very thick] (c1.center) circle (0.5pt) node[left] {$P$} to[out=-10,in=180] (1.6,-0.05) node[below] {$C$} to[out=0,in=-175] (2.2,0) to[out=5,in=-170] (c2.center) circle (0.5pt) node[right] {$Q$};
\draw[blue,very thick] (cp1.center) circle (0.5pt) node[left,yshift=5] {$P'\!\!\!$} to[out=-10,in=180] (1.6,2.25) node[below] {$C'$} to[out=0,in=-170] (cp2.center) circle (0.5pt) node[right,yshift=5] {$\!\!Q'$};
\draw[green!70!black,very thick] (cpp1.center) circle (0.5pt) to[out=-120,in=80] (2.2,1) node[right] {$C''$} to[out=-100,in=100] (2.2,0);
\draw[green!70!black,very thick,dotted] (2.2,0) to[out=-80,in=110] (2.35,-0.4);
\draw[fill=white] (2.2,0) circle (0.1);
\end{tikzpicture}
\caption{$v = v'$;}
\label{f:lnoam-vvp}
\end{subfigure}
\hfill
\begin{subfigure}[b]{0.37\textwidth}
\centering
\begin{tikzpicture}
\draw (0,0) arc (-10:10:5) node[very near start] (c1) {} node[at end] (a) {};
\draw[very thick,{|[width=5]}-{|[width=5]}] (a) arc(-60:-30:2) node[midway,left] {$g'$} node[very near end] (cp1) {} node[at end] (b) {};
\draw (b) arc (-75:-70:5) node (b1) {};
\draw[very thick,{|[width=5]}-{|[width=5]}] (b1) arc (-100:-80:3) node[very near start] (c11) {} node[near start] (c21) {} node[midway,above] {$g_\ell$} node[at end] (b2) {};
\draw (b2) arc (-110:-105:5) node (c) {};
\draw[very thick,{|[width=5]}-{|[width=5]}] (c) arc(-150:-120:2) node[very near start] (cp2) {} node[near start] (cpp1) {} node[midway,right] {$\,h'$} node[at end] (d) {};
\draw (d) arc (170:190:5) node[very near end] (c2) {};
\draw[blue,very thick] (c1.center) circle (0.5pt) node[left] {$P$} to[out=-10,in=180] (1.3,-0.05) node (c13) {} -- (1.6,-0.05) node (c23) {} -- (1.9,-0.05) node[below,yshift=2pt] {$C$} to[out=0,in=-170] (c2.center) circle (0.5pt) node[right] {$Q$};
\draw[red,very thick] (cp1.center) circle (0.5pt) node[left,yshift=5] {$P'\!\!\!$} to[out=-10,in=180] (1.35,2.25) node (c12) {} -- (1.5,2.25) node (c22) {} -- (1.9,2.25) node[below,yshift=2pt] {$C'$} to[out=0,in=-170] (cp2.center) circle (0.5pt) node[right,yshift=5] {$\!\!Q'$};
\draw[green!60!black,very thick] (c11.center) circle (0.5pt) to[out=-80,in=95] (c12.center) to[out=-85,in=90] (1.4,1) node[left] {$C_1\!$} to[out=-90,in=80] (c13.center);
\draw[green!80!black,very thick] (c21.center) circle (0.5pt) to[out=-80,in=95] (c22.center) to[out=-80,in=90] (1.65,1) node[right] {$\!C_2$} to[out=-90,in=85] (c23.center);
\draw[green!60!black,very thick,dotted] (c13.center) to[out=-100,in=75] (1.2,-0.4);
\draw[green!80!black,very thick,dotted] (c23.center) to[out=-95,in=80] (1.55,-0.4);
\fill (c12) circle (1pt);
\fill (c22) circle (1pt);
\draw[fill=white] (c13) circle (0.1);
\draw[fill=white] (c23) circle (0.1);
\end{tikzpicture}
\caption{$g_\ell$ large;}
\label{f:lnoam-gllarge}
\end{subfigure}
\hfill
\begin{subfigure}[b]{0.25\textwidth}
\centering
\begin{tikzpicture}
\draw (0,0) arc (-10:10:5) node[very near start] (c1) {} node[at end] (a) {};
\draw[very thick,{|[width=5]}-{|[width=5]}] (a) arc(-60:-30:2) node[near end,left] {$g'$} node[very near end] (cp1) {} node[at end] (b) {};
\draw[very thick] (b) arc (150:120:0.5) node[near start,yshift=7pt] {$\!\!\!\!g_1$} node[midway] (c11) {} node[at end] (b1) {};
\draw[very thick,{|[width=5]}-{|[width=5]}] (b1) arc (120:90:0.5) node[midway,label={[yshift=-5pt]$g_2$}] (c21) {} node[at end] (b2) {};
\draw (b2) arc (90:30:0.5) node (c) {};
\draw[very thick,{|[width=5]}-{|[width=5]}] (c) arc(-150:-120:2) node[very near start] (cp2) {} node[near start] (cpp1) {} node[near start,right] {$\,h'$} node[at end] (d) {};
\draw (d) arc (170:190:5) node[very near end] (c2) {};
\draw[blue,very thick] (c1.center) circle (0.5pt) to[out=-10,in=180] (0.9,-0.05) node (c13) {} -- (1.2,-0.05) node (c23) {} -- (1.3,-0.05) node[below,yshift=3pt] {$C\!\!\!\!\!\!\!\!$} to[out=0,in=-170] (c2.center) circle (0.5pt);
\draw[red,very thick] (cp1.center) circle (0.5pt) to[out=-10,in=180] (1,2.25) node (c12) {} -- (1.2,2.25) node (c22) {} node[below,yshift=4pt] {$C'\!\!\!\!\!\!\!\!\!\!$} to[out=0,in=-170] (cp2.center) circle (0.5pt);
\draw[green!60!black,very thick] (c11.center) circle (0.5pt) to[out=-40,in=110] (c12.center) to[out=-70,in=90] (1.15,1) node[left] {$C_1\!$} to[out=-90,in=70] (c13.center);
\draw[green!80!black,very thick] (c21.center) circle (0.5pt) to[out=-50,in=100] (c22.center) to[out=-80,in=90] (1.35,1) node[right] {$\!C_2$} to[out=-90,in=80] (c23.center);
\draw[green!60!black,very thick,dotted] (c13.center) to[out=-110,in=60] (0.7,-0.4);
\draw[green!80!black,very thick,dotted] (c23.center) to[out=-100,in=70] (1.05,-0.4);
\fill (c12) circle (1pt);
\fill (c22) circle (1pt);
\draw[fill=white] (c13) circle (0.1);
\draw[fill=white] (c23) circle (0.1);
\end{tikzpicture}
\caption{$m > 2$;}
\label{f:lnoam-m3}
\end{subfigure}

\begin{subfigure}[b]{0.22\textwidth}
\centering
\begin{tikzpicture}
\draw (0,0) arc (-10:10:5) node[very near start] (c1) {} node[at end] (a) {};
\draw[very thick,{|[width=5]}-{|[width=5]}] (a) arc(-60:-30:2) node[near end,left] {$g'$} node[very near end] (cp1) {} node[at end] (b) {};
\draw[very thick] (b) arc (120:60:0.4) node[midway,label={[yshift=-5pt]$g_1$}] (c11) {} node[at end] (c) {};
\draw[very thick,{|[width=5]}-{|[width=5]}] (c) arc(-150:-120:2) node[very near start] (cp2) {} node[near start] (cpp1) {} node[near start,right] {$\,h'$} node[at end] (d) {};
\draw (d) arc (170:190:5) node[very near end] (c2) {};
\draw[blue,very thick] (c1.center) circle (0.5pt) to[out=-10,in=180] (0.93,-0.05) node (c13) {} node[below,yshift=2pt] {$\!\!\!\!\!\!\!\!C$} -- (1.1,-0.05) node (cpp2) {} to[out=0,in=-170] (c2.center) circle (0.5pt);
\draw[red,very thick] (cp1.center) circle (0.5pt) to[out=-10,in=180] (0.93,2.3) node (c12) {} node[below,yshift=3pt] {$\!\!\!\!\!\!\!C'$} to[out=0,in=-170] (cp2.center) circle (0.5pt);
\draw[blue,very thick] (c11.center) circle (0.5pt) -- (c12.center) -- (0.93,1) node[left] {} -- (c13.center);
\draw[blue,very thick] (cpp1.center) circle (0.5pt) to[out=-110,in=90] (1.1,1) node[right] {} -- (cpp2.center);
\fill (c12) circle (1pt);
\end{tikzpicture}
\caption{$w_1 = v_+ = v$; \\~}
\label{f:lnoam-cc}
\end{subfigure}
\hfill
\begin{subfigure}[b]{0.24\textwidth}
\centering
\begin{tikzpicture}
\draw (0,0) arc (-10:10:5) node[very near start] (c1) {} node[at end] (a) {};
\draw[very thick,{|[width=5]}-{|[width=5]}] (a) arc(-60:-30:2) node[near end,left] {$g'$} node[very near end] (cp1) {} node[at end] (b) {};
\draw[very thick] (b) arc (120:60:0.4) node[midway,label={[yshift=-5pt]$g_1$}] (c11) {} node[at end] (c) {};
\draw[very thick,{|[width=5]}-{|[width=5]}] (c) arc(-150:-120:2) node[very near start] (cp2) {} node[near start] (cpp1) {} node[near start,right] {$\,h'$} node[at end] (d) {};
\draw (d) arc (170:190:5) node[very near end] (c2) {};
\draw[blue,very thick] (c1.center) circle (0.5pt) to[out=-10,in=180] (0.93,-0.05) node (c13) {} node[below,yshift=2pt] {$\!\!\!\!\!\!\!\!C$} -- (1.1,-0.05) node (cpp2) {} to[out=0,in=-170] (c2.center) circle (0.5pt);
\draw[red,very thick] (cp1.center) circle (0.5pt) to[out=-10,in=180] (0.93,2.3) node (c12) {} node[below,yshift=3pt] {$\!\!\!\!\!\!\!C'$} to[out=0,in=-170] (cp2.center) circle (0.5pt);
\draw[green!60!black,very thick] (c11.center) circle (0.5pt) -- (c12.center) -- (0.93,1) node[left] {$C_1\!\!$} -- (c13.center) -- (0.93,-0.4);
\draw[green!80!black,very thick] (cpp1.center) circle (0.5pt) to[out=-110,in=90] (1.1,1) node[right] {$\!C_+$} -- (cpp2.center) -- (1.1,-0.4);
\fill (c12) circle (1pt);
\fill (c13) circle (1pt);
\fill (cpp2) circle (1pt);
\end{tikzpicture}
\caption{$w_1,v_+ \in \link(v)$; \\~}
\label{f:lnoam-ii}
\end{subfigure}
\hfill
\begin{subfigure}[b]{0.22\textwidth}
\centering
\begin{tikzpicture}
\draw (0,0) arc (-10:10:5) node[very near start] (c1) {} node[at end] (a) {};
\draw[very thick,{|[width=5]}-{|[width=5]}] (a) arc(-60:-30:2) node[near end] (cmm1) {} node[near end,left] {$g'$} node[very near end] (cp1) {} node[at end] (b) {};
\draw[very thick] (b) arc (120:60:0.4) node[midway,label={[yshift=-5pt]$g_1$}] (c11) {} node[at end] (c) {};
\draw[very thick,{|[width=5]}-{|[width=5]}] (c) arc(-150:-120:2) node[very near start] (cp2) {} node[near start] (cpp1) {} node[near start,right] {$\,h'$} node[at end] (d) {};
\draw (d) arc (170:190:5) node[very near end] (c2) {};
\draw[blue,very thick] (c1.center) circle (0.5pt) to[out=-10,in=180] (0.76,-0.05) node (cmm2) {} node[below,yshift=2pt] {$\!\!\!\!\!\!\!\!C$} -- (0.93,-0.05) node (c13) {} -- (1.1,-0.05) node (cpp2) {} to[out=0,in=-170] (c2.center) circle (0.5pt);
\draw[red,very thick] (cp1.center) circle (0.5pt) to[out=-10,in=180] (0.93,2.3) node (c12) {} %node[below,yshift=3pt] {$\!\!\!\!\!\!\!C'$} 
to[out=0,in=-170] (cp2.center) circle (0.5pt);
\draw[blue,very thick] (c11.center) circle (0.5pt) -- (c12.center) -- (0.93,1) node[left] {} -- (c13.center);
\draw[green!80!black,very thick] (cpp1.center) circle (0.5pt) to[out=-110,in=90] (1.1,1) node[right] {$\!C_+$} -- (cpp2.center) -- (1.1,-0.4);
\draw[green!60!black,very thick] (cmm1.center) circle (0.5pt) to[out=-70,in=90] (0.76,1) node[left] {$C_-\!\!$} -- (cmm2.center) -- (0.76,-0.4);
\fill (c12) circle (1pt);
\fill (cpp2) circle (1pt);
\fill (cmm2) circle (1pt);
\end{tikzpicture}
\caption{$w_1 = v$ whereas $v_-,v_+ \in \link(v)$;}
\label{f:lnoam-ici}
\end{subfigure}
\hfill
\begin{subfigure}[b]{0.22\textwidth}
\centering
\begin{tikzpicture}
\draw (0,0) arc (-10:10:5) node[very near start] (c1) {} node[at end] (a) {};
\draw[very thick,{|[width=5]}-{|[width=5]}] (a) arc(-60:-30:2) node[near end] (cmm1) {} node[near end,left] {$g'$} node[very near end] (cp1) {} node[at end] (b) {};
\draw[very thick] (b) arc (120:60:0.4) node[midway,label={[yshift=-5pt]$g_1$}] (c11) {} node[at end] (c) {};
\draw[very thick,{|[width=5]}-{|[width=5]}] (c) arc(-150:-120:2) node[very near start] (cp2) {} node[near start] (cpp1) {} node[near start,right] {$\,h'$} node[at end] (d) {};
\draw (d) arc (170:190:5) node[very near end] (c2) {};
\draw[blue,very thick] (c1.center) circle (0.5pt) to[out=-10,in=180] (0.76,-0.05) node (cmm2) {} node[below,yshift=2pt] {$\!\!\!\!\!\!\!\!C$} -- (0.93,-0.05) node (c13) {} -- (1.1,-0.05) node (cpp2) {} to[out=0,in=-170] (c2.center) circle (0.5pt);
\draw[red,very thick] (cp1.center) circle (0.5pt) to[out=-10,in=180] (0.93,2.3) node (c12) {} %node[below,yshift=3pt] {$\!\!\!\!\!\!\!C'$} 
to[out=0,in=-170] (cp2.center) circle (0.5pt);
\draw[green!70!black,very thick] (c11.center) circle (0.5pt) -- (c12.center) -- (c13.center) -- (0.93,-0.4);
\draw[blue,very thick] (cpp1.center) circle (0.5pt) to[out=-110,in=90] (1.1,1) node[right] {} -- (cpp2.center);
\draw[blue,very thick] (cmm1.center) circle (0.5pt) to[out=-70,in=90] (0.76,1) node[left] {} -- (cmm2.center);
\fill (c12) circle (1pt);
\fill (c13) circle (1pt);
\draw[green!70!black,->] (0.6,1) node[left] {$C_1\!\!$} -- (0.9,1);
\end{tikzpicture}
\caption{$w_1 \in \link(v)$ and $v_- = v_+ = v$.}
\label{f:lnoam-cic}
\end{subfigure}

\caption{Ruling out impossible scenarios in the proof of Lemma \ref{l:noam}.}
\end{figure}
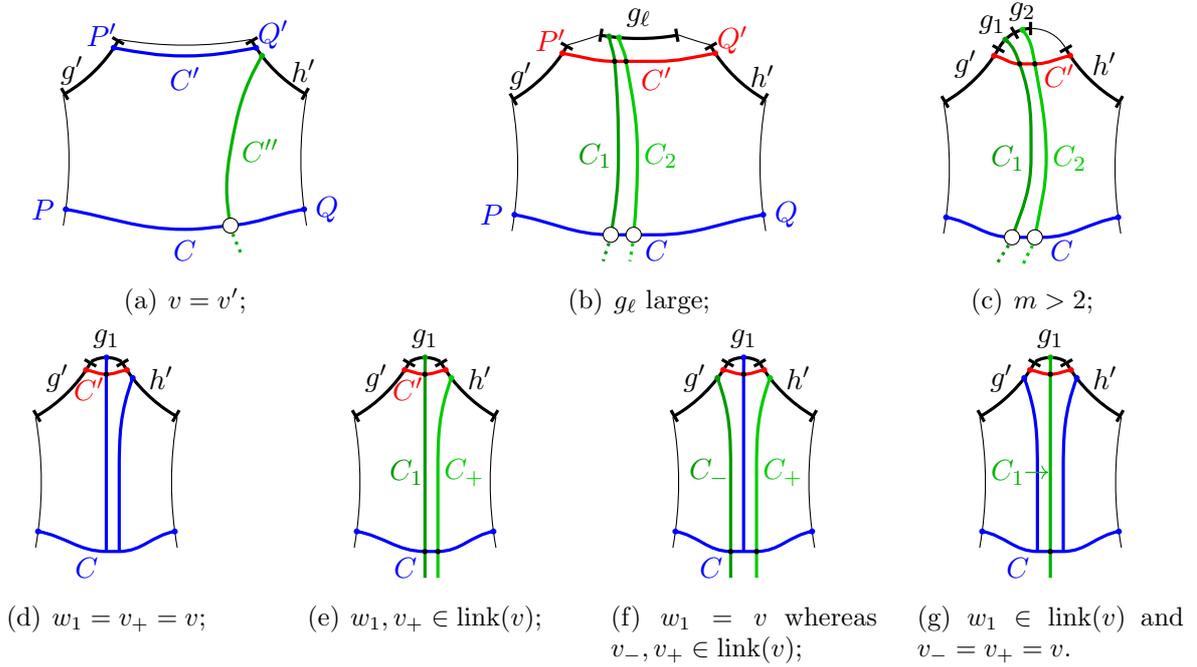

\begin{proof}
%Let $\Delta = \Delta_i(W,\mathcal{D})$, and let $p_1,\ldots,p_k$ (respectively $g_1,\ldots,g_k$) be the pieces of $\mathcal{D}$ (respectively the corresponding van Kampen pieces of $\Delta$) as above. We aim to show that the conclusion holds for $\Delta$.

Suppose that $(C,[P,Q])$ is an almost minimal $v$-component of $\Delta$. Let $g$ and $h$ be the van Kampen pieces of $\Delta$ such that $P \in \mathcal{I}_g(\Delta)$ and $Q \in \mathcal{I}_h(\Delta)$. Then $(C,[P,Q])$ must enclose a minimal $v'$-component $(C',[P',Q'])$ (for some $v' \in V(\Gamma)$); let $g'$ and $h'$ be the van Kampen pieces of $\Delta$ such that $P' \in \mathcal{I}_{g'}(\Delta)$ and $Q' \in \mathcal{I}_{h'}(\Delta)$. Suppose, without loss of generality, that $(C',[P',Q'])$ is chosen in such a way that $[P,P']$ is as small as possible: that is, if $(C'_0,[P'_0,Q'_0])$ is any minimal component enclosed by $(C,[P,Q])$, then $[P,P'] \subseteq [P,P'_0]$.

We first claim that $v \neq v'$. Indeed, as $(C',[P',Q'])$ is minimal and $C'$ is supported at $h'$, it follows from Lemma \ref{l:mtsupplarge}\ref{i:lmtsl-large} that $h'$ is a large van Kampen piece, and so, by Lemma \ref{l:red}\ref{i:lred-1}, there are at least $3$ components supported at $h'$. Therefore, the second component supported at $h'$ -- a $v''$-component $C''$, say -- is neither trivial nor minimal by Lemma \ref{l:mtsupplarge}\ref{i:lmtsl-first}, and so it must either intersect or coincide with $C$; see Figure \ref{f:lnoam-vvp}. As $h'$ is a geodesic and $C'$ and $C''$ are the first (by Lemma \ref{l:mtsupplarge}\ref{i:lmtsl-first}) and the second components (respectively) supported at $h'$, it follows that $v' \neq v''$. If $C = C''$, then this immediately implies $v \neq v'$. On the other hand, as $C'$ and $C''$ are supported at $h'$, it follows that $h'$ must be a $\widehat{u}$-piece for some $u \in \link(v') \cap \link(v'')$. Thus, if $C$ and $C''$ intersect, then $v'' \in \link(v)$, and so the fact that $\Gamma$ is triangle-free implies that $v \neq v'$. Thus we have $v \neq v'$ in either case, as claimed.

Now let $g'=g_0,g_1,\ldots,g_m=h'$ be consecutive van Kampen pieces in $\Delta$. We claim that the pieces $g_1,\ldots,g_{m-1}$ are small. Indeed, suppose for contradiction that $g_\ell$ is large for some $\ell \in \{ 1,\ldots,m-1 \}$, and so, by Lemma \ref{l:red}\ref{i:lred-1}, there are at least three components supported at $g_\ell$. Let $C_1$ and $C_2$ be the first and the second components supported at $g_\ell$, respectively. As $(C',[P',Q'])$ is minimal and not supported at $g_\ell$, it follows from Lemma \ref{l:mtsupplarge}\ref{i:lmtsl-nencl} that $C'$ intersects both $C_1$ and $C_2$, and that neither $C_1$ nor $C_2$ is trivial; see Figure \ref{f:lnoam-gllarge}. As $C_2$ is neither the first nor the last component supported at $g_\ell$, it follows from Lemma \ref{l:mtsupplarge}\ref{i:lmtsl-first} that $C_2$ is not minimal, and so must either intersect or coincide with $C$. On the other hand, if $C_1$ is not minimal, then it must likewise either intersect or coincide with $C$; whereas if $(C_1,[P_1,Q_1])$ is minimal for some $[P_1,Q_1]$, then we cannot have $P_1 \in [P',Q']$ by minimality of $(C',[P',Q'])$, and we cannot have $P_1 \in [P,P']$ by minimality of $[P,P']$, so $P_1 \notin [P,Q]$ and $C$ must intersect $C_1$. Thus $C_1$ and $C_2$ are two components supported at $g_\ell$ and each either intersecting or coinciding with $C$ and $C'$, so by Lemma \ref{l:interhat} $g_\ell$ must be both a $\widehat{v}$-piece and a $\widehat{v'}$-piece, which is impossible as $v \neq v'$.

Thus the pieces $g_1,\ldots,g_{m-1}$ are all small, as claimed. It follows that, for each $\ell \in \{ 1,\ldots,m-1 \}$, $g_\ell$ is a $\dot{w}_\ell$-piece for some $w_\ell \in V(\Gamma)$, and there is a unique $w_\ell$-component $C_\ell$ supported at $g_\ell$. Thus, by Lemma \ref{l:mtsupplarge}\ref{i:lmtsl-nencl}, each $C_\ell$ intersects $C'$ and so $w_\ell \in \link(v')$ for each $\ell$. Moreover, it also follows from Lemma \ref{l:mtsupplarge}\ref{i:lmtsl-large} that each $C_\ell$ is neither trivial nor minimal, and so must either intersect or coincide with $C$; therefore, $w_\ell \in \link(v) \cup \{v\}$.

We now claim that $m = 2$. Indeed, as $(C',[P',Q'])$ is not trivial, we cannot have $m < 2$. On the other hand, if $m > 2$, then consider $w_1,w_2 \in \link(v')$; see Figure \ref{f:lnoam-m3}. By Lemma \ref{l:red}\ref{i:lred-adsame}, we have $w_1 \neq w_2$, and so we cannot have $w_1 = v$ and $w_2 = v$. We also cannot have $w_1 = v$ and $w_2 \in \link(v)$: indeed, otherwise $v,v',w_2$ span a triangle in $\Gamma$, contradicting the fact that $\Gamma$ is triangle-free; similarly, $w_1 \in \link(v)$ and $w_2 = v$ is impossible. Finally, we have $|\link(v) \cap \link(v')| \leq 1$ (as $\Gamma$ is square free and $v \neq v'$), so as $w_1 \neq w_2$, we cannot have $w_1,w_2 \in \link(v)$. Thus we must have $m = 2$, as claimed. In particular, the van Kampen pieces $g',g_1,h'$ are consecutive, and $g_1$ is a $\dot{w}_1$-piece for some $w_1 \in \link(v') \cap (\link(v) \cup \{v\})$.

Now by Lemma \ref{l:mtsupplarge}\ref{i:lmtsl-large} and Lemma \ref{l:mtsupplarge}\ref{i:lmtsl-first}, $C'$ is the last component supported at $g'$ and the first component supported at $h'$, and both $g'$, $h'$ are large. In particular, by Lemma \ref{l:red}\ref{i:lred-1}, there are at least $3$ components supported at each $g'$ and $h'$. Moreover, by Lemma \ref{l:mtsupplarge}\ref{i:lmtsl-first}, the penultimate component supported at $g'$ and the second component supported at $h'$ -- say a $v_-$-component $C_-$ and a $v_+$-component $C_+$, respectively -- are both not trivial and not minimal, and so they must each either intersect or coincide with $C$; in particular, $v_-,v_+ \in \link(v) \cup \{v\}$. % Moreover, as $g'$ and $h'$ are geodesics, we have $v_- \neq v' \neq v_+$.

Now if $w_1 = v_+ = v$ (see Figure \ref{f:lnoam-cc}), then $C_+ = C = C_1$ and in particular $C_+$ and $C'$ intersect, contradicting Lemma \ref{l:nointer} (as both $C_+$ and $C'$ are supported at $h'$); similarly, we cannot have $w_1 = v_- = v$. If $w_1,v_+ \in \link(v)$ (see Figure \ref{f:lnoam-ii}) then, as $h'$ is large, we know that $h'$ is a $\widehat{u}$-piece for some $u \in \link(v') \cap \link(v_+)$, and $v,v_+,u,v',w_1,v$ is a closed walk in $\Gamma$ of length $5$, contradicting the fact that $\Gamma$ is triangle-free and pentagon-free; similarly, we cannot have $w_1,v_- \in \link(v)$.

It follows that either $w_1 = v$ and $v_-,v_+ \in \link(v)$ (see Figure \ref{f:lnoam-ici}), or $w_1 \in \link(v)$ and $v_- = v_+ = v$ (see Figure \ref{f:lnoam-cic}). In either case, $C'$ and $C_-$ are two pieces supported at $g'$ and intersecting $C_1$, and so $g'$ is a $\widehat{w}_1$-piece by Lemma \ref{l:interhat}; similarly, $h'$ is a $\widehat{w}_1$-piece. But then $g_1$ is a $\dot{w}_1$-piece adjacent to two $\widehat{w}_1$-pieces, contradicting Lemma \ref{l:red}\ref{i:lred-adperm5}. This concludes the proof.
\end{proof}

\begin{cor} \label{c:nothing}
Let $\Delta$ be a minimal polygonal van Kampen diagram, and suppose that $\Gamma$ has girth $\geq 6$. % over $\overline{S}$ such that $W \in \ker^\omega(\overline\varphi_i)$. If $\Gamma$ has girth $\geq 6$, then $\mathcal{D}$ is trivial, that is, has $0$ pieces.
Then $\Delta$ has $0$ van Kampen pieces.
\end{cor}

\begin{proof}
%Let $i \in \mathcal{A}(\mathcal{D},W)$ and let $\Delta = \Delta_i(W,\mathcal{D})$. It is enough to show that $\Delta$ is trivial, that is, it has $0$ van Kampen pieces.

We claim first that all components of $\Delta$ are either minimal or trivial. Indeed, otherwise $\Delta$ has a non-minimal non-trivial component $(C,[P,Q])$ with $[P,Q]$ minimal (with respect to inclusion) among all such components. Thus $(C,[P,Q])$ cannot enclose any non-minimal non-trivial component by minimality of $[P,Q]$, and so $(C,[P,Q])$ is almost minimal. This contradicts Lemma \ref{l:noam}.

Thus, as all components of $\Delta$ are either minimal or trivial, it follows by Lemma \ref{l:mtsupplarge}\ref{i:lmtsl-large} that $\Delta$ has no small van Kampen pieces. Suppose $\Delta$ has at least one van Kampen piece. Then any piece $p$ of $\Delta$ has length $\geq 3$ by Lemma \ref{l:red}\ref{i:lred-1}. Now if $C$ is the second component of $\Delta$ supported at a piece $p$, then $C$ is a component that is either minimal or trivial, but $C$ is neither the first nor the last component supported at $p$. This contradicts Lemma \ref{l:mtsupplarge}\ref{i:lmtsl-first}.

%Let $C_1$, $C_2$ and $C_3$ be the first, second and third pieces of $\Delta$ supported at $p$, and let $P_1,P_2,P_3 \in \mathcal{I}_p(\Delta)$ be boundary points of $C_1,C_2,C_3$, respectively. Then $C_1,C_2,C_3$ have supporting intervals $[P_1,Q_1],[P_2,Q_2],[P_3,Q_3]$, respectively, such that $[P_3,Q_3] \subset [P_2,Q_2] \subset [P_1,Q_1]$. As $(C_1,[P_1,Q_1])$ is either minimal or trivial, this implies that both $(C_2,[P_2,Q_2])$ and $(C_3,[P_3,Q_3])$ are trivial starting at $p$. But this 

Thus $\Delta$ has no van Kampen pieces, as required.
\end{proof}

\subsection{Proofs of Proposition \ref{p:girth6} and Theorem \ref{t:en}} \label{ss:ten}

In this subsection, we deduce Proposition \ref{p:girth6} from the results in Section \ref{ss:girth}. We then use Proposition \ref{p:girth6} and Theorem \ref{t:engood} to prove Theorem \ref{t:en}.

\begin{proof}[Proof of Proposition \ref{p:girth6}] As noted after the statement of the Proposition, it is clear that $K \subseteq \overline{F}_\omega$, and so we only need to show that $\overline{F}_\omega \subseteq K$. This is equivalent to saying that every (cyclic) word $W \in \overline{F}_\omega$ is equivalent to the trivial word $1 \in \overline{F}_\omega$ of length $0$. As the trivial word has a polygonal representation $\mathcal{D}_0$ with $0$ pieces, which is clearly minimal, it is enough to show that $\mathcal{D}_0$ is the \emph{unique} minimal polygonal representation for any word $W \in \overline{F}_\omega$.

Thus, let $\mathcal{D}$ be a minimal polygonal representation for a word $W \in \overline{F}_\omega$, and let $p_1,\ldots,p_k$ be the pieces of $\mathcal{D}$, so that $W = p_1 \cdots p_k$. Since $\overline\varphi_i(W) = 1$ $\omega$-almost surely, we may $\omega$-almost surely define a polygonal van Kampen diagram $\Delta_i$ for $\mathbf{W}_i = g_{i,1} \cdots g_{i,k}$ whose van Kampen pieces are $g_{i,1},\ldots,g_{i,k}$, where $g_{i,\ell}$ is a geodesic word over $\bigsqcup_{v \in V(\Gamma)} (G_v \setminus \{1\})$ representing $\overline\varphi_i(p_\ell)$, such that the label of a piece $g_{i,\ell}$ is $\dot{v}$ whenever $g_{i,\ell}$ represents an element in $G_v \setminus \{1\}$ for some $v \in V(\Gamma)$. Let $\mathcal{A} \subseteq \mathbb{N}$ for the set of all $i \in \mathbb{N}$ such that
\begin{itemize}
\item $\overline\varphi_i(W) = 1$;
%\item $g_{i,l}$ is a van Kampen piece in $\Delta_i(W,\mathcal{D})$ for all $l \in \{ 1,\ldots,k \}$;
\item for each $\ell \in \{ 1,\ldots,k \}$ and each $j \in \{ 0,1 \}$, the piece $p_\ell$ has length $j$ if and only if the van Kampen piece $g_{i,\ell}$ has length $j$;
\item for each $\ell \in \{ 1,\ldots,k \}$, the label of the piece $p_\ell$ of $\mathcal{D}$ coincides with the label of the van Kampen piece $g_{i,\ell}$ of $\Delta_i$.
\end{itemize}

It follows from Lemma \ref{l:red}\ref{i:lred-0}, Lemma \ref{l:red}\ref{i:lred-1} and the construction that $\mathcal{A}$ is an $\omega$-large subset of $\mathbb{N}$. But note that as $\mathcal{D}$ is minimal, the polygonal van Kampen diagram $\Delta_i$ is minimal for each $i \in \mathcal{A}$. It then follows from Corollary \ref{c:nothing} that $\Delta_i$ has $0$ van Kampen pieces for each $i \in \mathcal{A}$ (and so $\omega$-almost surely). Thus $\mathcal{D}$ has $0$ pieces, and so $\mathcal{D} = \mathcal{D}_0$, as required.
\end{proof}

\begin{proof}[Proof of Theorem \ref{t:en}]
If $|V(\Gamma)| = 1$, then we have $\GG \cong G_v$ for the unique vertex $v$ of $\Gamma$, and so $\GG$ is equationally noetherian by the assumption. Thus, we may assume, without loss of generality, that $|V(\Gamma)| \geq 2$.

Suppose first that $\Gamma$ is connected. We claim that $\Gamma$ is admissible. Thus, let $F$ be a finitely generated free group, and let $(\varphi_i: F \to \GG)_{i=1}^\infty$ be a sequence of linking homomorphisms. As in the beginning of Section \ref{ss:pr}, define a free group $\overline{F}$ containing $F$ and subgroups $\widehat{F}_v,\dot{F}_v \leq \overline{F}$ for each $v \in V(\Gamma)$, and extend the homomorphisms $\varphi_i$ to $\overline\varphi_i: \overline{F} \to \GG$. It follows from the construction that $\ker(\varphi_i) = \ker(\overline\varphi_i) \cap F$ for each $i$, and so $F_{\omega,(\varphi_i)} = \overline{F}_{\omega,(\overline\varphi_i)} \cap F$.

Now for any $v \in V(\Gamma)$, consider homomorphisms $\varphi_i^{(v)}: \widehat{F}_v \to \GG[\link(v)]$ given by $\varphi_i^{(v)}(g) = \overline\varphi_i(g)$. As $\Gamma$ is square-free, we have $\GG[\link(v)] \cong \mathop{*}_{w \in \link(v)} G_w$; in particular, it follows from Theorem \ref{t:sela} (and the fact that $|V(\Gamma)| < \infty$) that $\GG[\link(v)]$ is equationally noetherian. It follows that $\left(\widehat{F}_v\right)_{\omega,(\varphi_i^{(v)})} \subseteq \ker(\varphi_i^{(v)})$ $\omega$-almost surely. But it follows from the construction that $\ker(\varphi_i^{(v)}) = \ker(\overline\varphi_i) \cap \widehat{F}_v$ for each $i$, and in particular $\left(\widehat{F}_v\right)_{\omega,(\varphi_i^{(v)})} = \overline{F}_{\omega,(\overline\varphi_i)} \cap \widehat{F}_v$. Thus $\overline{F}_{\omega,(\overline\varphi_i)} \cap \widehat{F}_v \subseteq \ker(\overline\varphi_i)$ $\omega$-almost surely.

Note that, by our construction, we have $\left[ \dot{F}_v, \widehat{F}_v \right] \subseteq \ker(\overline\varphi_i)$ for each $i$ and each $v \in V(\Gamma)$. Together with the fact that $\overline{F}_{\omega,(\overline\varphi_i)} \cap \widehat{F}_v \subseteq \ker(\overline\varphi_i)$ $\omega$-almost surely for each $v \in V(\Gamma)$, this implies (by Proposition \ref{p:girth6}) that $\overline{F}_{\omega,(\overline\varphi_i)} \subseteq \ker(\overline\varphi_i)$ $\omega$-almost surely. In particular, it follows that $F_{\omega,(\varphi_i)} = \overline{F}_{\omega,(\overline\varphi_i)} \cap F \subseteq \ker(\overline\varphi_i) \cap F = \ker(\varphi_i)$ $\omega$-almost surely. Thus $\Gamma$ is admissible, as claimed. In particular, it follows from Theorem \ref{t:engood} that $\GG$ is equationally noetherian.

Finally, suppose that the graph $\Gamma$ is not connected. Then we have a partition $V(\Gamma) = A_1 \sqcup \cdots \sqcup A_m$ into non-empty subsets such that $\Gamma = \Gamma_{A_1} \sqcup \cdots \sqcup \Gamma_{A_m}$ is a disjoint union of connected subgraphs $\Gamma_{A_i}$, and so $\GG \cong \GG[A_1] * \cdots * \GG[A_m]$. By the argument above, it follows that $\GG[A_i]$ is equationally noetherian for each $i$. Therefore, by Theorem \ref{t:sela}, $\GG$ is equationally noetherian as well, as required.
\end{proof}

\bibliographystyle{amsalpha}
\bibliography{../../../all}

\end{document}